\documentclass[10pt]{article}
\usepackage[T1]{fontenc}
\usepackage{calligra}
\DeclareMathAlphabet{\mathcalligra}{T1}{calligra}{m}{n}
\usepackage[T1]{fontenc}
\usepackage{humanist}
\usepackage{rotunda}
\usepackage[utf8]{inputenc}

\usepackage{latexsym}
\usepackage{amsfonts}
\usepackage{amsmath}
\usepackage{esint}
\usepackage{amsthm}
\usepackage{amssymb}
\usepackage{mathrsfs}
\usepackage{dsfont}    
\usepackage{bbold}     
\usepackage[english]{babel}
\usepackage{caption}
\usepackage{float}
\usepackage{psfrag}
\usepackage{graphicx}
\usepackage{epsfig}
\usepackage{hyperref}
\newcommand{\ee}{\end{equation}}

\newtheorem*{theorem}{Theorem}
\newtheorem*{prop*}{Theorem}
\newtheorem{theo}{Theorem}[section]

\newtheorem{defi}[theo]{Definition}
\newtheorem{lemma}[theo]{Lemma}
\newtheorem{proposition}[theo]{Proposition}
\newtheorem{rmk}[theo]{Remark}

\newtheorem{const}[theo]{Constraint}
\newtheorem{hypo}{Hypothesis}[section]

\newtheorem*{metate}{Our Goal}

\newcommand{\zerarcounters}{\setcounter{equation}{0}}
\newcommand{\HH}{{\mathcal H}}
\newcommand{\ZZZ}{\mathds{Z}}
\newcommand{\CCC}{\mathds{C}}
\newcommand{\NNN}{\mathds{N}}

\newcommand{\RRR}{\mathds{R}}
\newcommand{\TTT}{\mathds{T}}
\newcommand{\uno}{\mathds{1}}

\newcommand{\db}{{\varrho}}
\newcommand{\calA}{{\mathcal A}}
\newcommand{\BB}{{\mathcal B}}

\newcommand{\DD}{{\mathcal D}}
\newcommand{\calE}{{\mathcal E}}
\newcommand{\calF}{{\mathcal F}}
\newcommand{\calG}{{\mathcal G}}
\newcommand{\calH}{{\mathcal H}}
\newcommand{\calI}{{\mathcal I}}

\newcommand{\calK}{{\mathcal K}}
\newcommand{\calL}{{\mathcal L}}

\newcommand{\NN}{{\mathcal N}}
\newcommand{\calO}{{\mathcal O}}
\newcommand{\calP}{{\mathcal P}}

\newcommand{\RR}{{\mathcal R}}
\newcommand{\SSSS}{{\mathcal S}}
\newcommand{\TT}{{\mathcal T}}
\newcommand{\calU}{{\mathcal U}}
\newcommand{\VV}{{\mathcal V}}

\newcommand{\calW}{{\mathcal W}}
\newcommand{\calX}{{\mathcal X}}

\newcommand{\tG}{{\mathtt G}}
\newcommand{\tR}{{\mathtt R}}

\newcommand{\gotp}{{\mathfrak p}}

\newcommand{\gotA}{{\mathfrak A}}
\newcommand{\gotB}{{\mathfrak B}}

\newcommand{\gotL}{{\mathfrak L}}

\newcommand{\gotN}{{\mathfrak N}}

\newcommand{\gotR}{{\mathfrak R}}

\newcommand{\gotU}{{\mathfrak U}}

\newcommand{\gotW}{{\mathfrak W}}

\newcommand{\ol}{\overline}

\newcommand{\Fullbox}{{\rule{2.0mm}{2.0mm}}}

\newcommand{\EP}{\hfill\Fullbox\vspace{0.2cm}}
\newcommand{\prova}{\noindent{\it Proof. }}
\newcommand{\io}{\infty}
\newcommand{\e}{\varepsilon}
\newcommand{\al}{\alpha}
\newcommand{\de}{\delta}
\newcommand{\be}{\beta}
\newcommand{\ze}{\zeta}

\newcommand{\x}{\xi}

\newcommand{\ka}{\kappa}
\newcommand{\g}{\gamma}
\newcommand{\om}{\omega}
\newcommand{\h}{\eta}
\newcommand{\la}{\lambda}

\newcommand{\f}{\varphi}
\newcommand{\s}{\sigma}

\newcommand{\B}{\boldsymbol{B}}
\newcommand{\del}{\partial}

\newcommand{\av}[1]{\langle #1 \rangle}

\newcommand{\oo}{\omega}

\newcommand{\ta}{\mathtt{a}}
\newcommand{\tb}{\mathtt{b}}
\newcommand{\tg}{\mathtt{g}}
\newcommand{\td}{\mathtt{d}}
\newcommand{\tr}{\mathtt{r}}
\newcommand{\tv}{\mathtt{v}}

\newcommand{{\resonance}}{relevant self-energy cluster }

\newcommand{\ii}{{\rm i}}

\newcommand{\Gl}{{\Gamma_{\gotp_1}}}

\newcommand{\Tl}{{\Theta_{\gotp_1}}}

\newcommand{\tino}{{\text{\tiny 0}}}

\def\leftinv#1{ {#1}^{-1}}
\def\ins#1#2#3{\vbox to0pt{\kern-#2 \hbox{\kern#1 #3}\vss}\nointerlineskip}

\makeindex
\begin{document}



\title{\bf Finite dimensional invariant KAM tori for tame vector fields}
\author{\bf L. Corsi$^*$, R. Feola$^{**}$, M. Procesi$^\dag$
\\
\small ${}^*$ Georgia Institute of Technology, Altanta, lcorsi6@math.gatech.edu; \\
\small
${}^{**}$ SISSA, Trieste, rfeola@sissa.it; 
\\
\small
${}^\dag$ Universit\`a di Roma Tre, procesi@mat.uniroma3.it}
%
%
\maketitle

\begin{abstract}
We discuss a Nash-Moser/ KAM algorithm for the construction of invariant tori for {\em tame} vector fields.
Similar algorithms have been studied widely both in finite and infinite dimensional contexts:
we are particularly interested in the second case where tameness properties of the vector fields become very important.
We focus on the formal aspects of the algorithm and particularly on the minimal hypotheses needed for convergence.
We discuss various applications where we show how our algorithm allows to reduce to solving only linear forced equations.
We remark that our algorithm works at the same time in analytic and Sobolev class.
\end{abstract}

\tableofcontents

\zerarcounters

\section{Introduction}
The aim of this paper is to provide a general and flexible approach to study 
  the existence,
for finite or infinite dimensional dynamical systems  of of finite dimensional  invariant tori carrying a quasi-periodic flow. 
To this purpose we discuss an iterative scheme for finding invariant tori for the dynamics of a vector field $F= N_0+G$, where $N_0$ is a linear vector field which admits an invariant torus and $G$ is a {\em perturbation}.\\
 By an { \em invariant torus} of $\dot u= F(u)$, with $u\in E$ a Banach space, we mean an embedding  $\TTT^d\to E$,  which is invariant under the dynamics of $F$, i.e. the vector field $F$ is tangent to the embedded torus. 
Similarly an {\em analytic} invariant torus is a map $\TTT_{s}^d\to E$ (with $s>0$)  where
 $\TTT^d_s$ is a {\em thickened} torus\footnote{ we use the standard notation $\TTT^d_s := \big\{ \theta \in \CCC^d \, : \, {\rm Re}(\theta)\in\TTT^{d},\ \max_{h=1, \ldots, d} |{\rm Im} \, \theta_h | < s
 	\big\}$}. 
	
It is very reasonable to work in the setting of the classical Moser scheme of \cite{M1}, namely $F$ acts  on a product space $(\theta,y,w)$ where $\theta\in \TTT^d_s,y\in \CCC^{d_1}$ while $w\in \ell_{a,p}$ some separable scale of Hilbert spaces.  The variables $\theta$ appear naturally as a parametrization for the invariant torus of $N_0$.
The $y$ variables are    constants of motion for $N_0$, 
 in applications they naturally appear
 as  ``conjugated'' to $\theta$, for instance in the Hamiltonian setting they come from the symplectic structure.
 The variables $w$ describe the dynamics in the directions orthogonal to the torus.
The main example that we have in mind is\footnote{we use the standard notation for vector fields $F=\sum_{\tv=\theta,y,w}F^{(\tv)}\del_{\tv}$}
 \begin{equation}\label{poesse}
 N_0= \omega^{(0)}\cdot \del_\theta + \Lambda^{(0)} w \del_w
 \end{equation}
 where $\omega^{(0)}\in\RRR^d$ is a constant vector while $\Lambda^{(0)}$ is a block-diagonal skew self--adjoint operator, independent of $\theta$.  
 Note that $N_0$ has the invariant torus $y=0,w=0$ , where the  vector field reduces to $\dot \f = \omega^{(0)}$.
 
Regarding the normal  variables $w$, we do not need to specify $\ell_{a,p}$ but only give some properties, see Hypothesis \ref{hyp22}, which essentially amount to requiring that $\ell_{a,p}$ is a weighted sequence space\footnote{a good example is to consider spaces of Sobolev or analytic functions on compact manifolds}
where $a\geq 0$ is an exponential weight while $p>0$ is polynomial, for example
 $$
 w=\{w_j\}_{j\in \calI\subseteq \NNN}\,,\quad w_j\in\CCC \,,\quad \| w\|_{a,p}^{2}:=\sum_{j\in\calI\subseteq\NNN} \lambda_j^{2p} e^{2a \lambda_j}|w_{j}|^{2}\,,\quad  0< \lambda_j\leq \lambda_{j+1}\ldots\quad \lambda_i\to \infty.
 $$
 Note that if $\calI$ is a finite set our space is finite dimensional.

\vskip.5truecm

%
%
The existence of an invariant torus 
in the variables $(\theta,y,w)$ means the existence of a map
$\TTT^{d}_s\overset{h}{\to} \CCC^{d_1}\times\ell_{a,p}$ of the form
$\theta \mapsto h(\theta)=(h^{(y)}(\theta),h^{(w)}(\theta))$ such that
\begin{equation}\label{gigi}
\calF(F,h)\equiv \calF(h):= F^{(\tv)}( \theta,h^{(y)}(\theta),h^{(w)}(\theta) )- \partial_\theta h^{(\tv)} \cdot  F^{(\theta)}(\theta,h^{(y)}(\theta),h^{(w)}(\theta))=0, \quad \tv=y,w,
\end{equation}
hence it coincides with the search for  zeros of the functional $\calF$ with unknown  $h$. Here $h$ lives in some Banach space, say
$H^{q}(\TTT^{d}_{s}; \CCC^{d_1}\times \ell_{a,p})$ on which $\calF$ is at least differentiable. Note moreover that $\calF(N_0,0)=0$ trivially. 
Since we are in a perturbative setting, once one has the torus embedding, one can study the dynamics of the variables $\theta$ restricted to the torus and look for 
a change of variables $\f\mapsto\theta(\f)$ which conjugates the dynamics of $\theta$ to the linear dynamics $\dot{\f}=\oo$ with $\oo\sim \oo^{(0)}$ a rationally  independent vector.
Obviously this could be done directly by looking for a {\em quasi-periodic solution} i.e. a map $$h : \f\to (h^{(\theta)}(\f), h^{(y)}(\f), h^{(w)}(\f)),\quad h\in H^{q}(\TTT^{d}_{s}; \CCC^d\times\CCC^{d_1}\times \ell_{a,p}).$$
which solves the functional equation
$$
\calF(h):= F( h^{(\theta)}(\varphi),h^{(y)}(\f),h^{(w)}(\varphi) )- \omega\cdot \partial_\f h=0.
$$
If we take $\ell_{a,p}=\emptyset$, $d_1=d$, this is the classical KAM framework of  Kolmogorov \cite{KKK}, Arnold \cite{Arn} and Moser \cite{Moser-Pisa-66}; see 
also \cite{Pos,Rus,Sev}.

Even in the simplest setting, equation \eqref{gigi} cannot be solved  by classical Implicit Function Theorem. In fact typically the operator $\calF$ in \eqref{gigi} linearized at $F=N_0,h=0$ is not invertible on $H^{q}(\TTT^{d}_{s}; \CCC^{d_1}\times \ell_{a,p})$ since it  has a spectrum which  accumulates to zero (the so-called small divisors).
To overcome this problem one can use  a {\em Nash-Moser} iterative scheme in order to find 
a sequence of approximate solutions rapidly converging to the true solution.
The fast convergence is used to control the loss of regularity due to the small divisors.
Such schemes are adaptations to Banach spaces of the  Newton method to find zeros of functions, see \cite{Ze}.

\begin{figure}[ht] 
	\centering 
	\ins{220pt}{-30pt}{$\calF(h)$}
	\ins{280pt}{-143pt}{$h_0$}
	\ins{245pt}{-143pt}{$h_1$}
	\ins{220pt}{-143pt}{$h_2$}
	\includegraphics[width=3in]{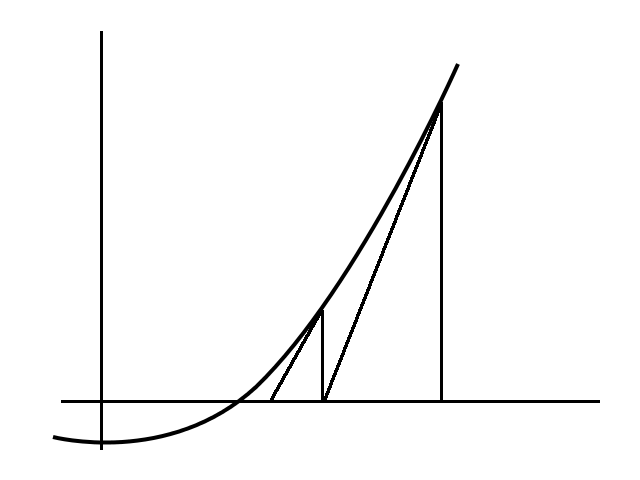} 
	\vskip.2truecm 
	\caption{Three steps of the Newton algorithm
		$h_{n+1} := h_n - (d_{h}\calF(h_n))^{-1} [\calF( h_n)] $ }
	\label{fig.newton} 
\end{figure} 
In order to run an algorithm of this type one must be able to control the linearized operator in a neighborhood 
of the expected solution; see Figure \ref{fig.newton}.
Due to the presence of small divisors it is not possible to invert such operator as a functional from a Sobolev space to itself (not even the operator linearized about zero).  
However, since the Newton scheme is quadratic, one may accept that $d_{h}\calF^{-1}$ is well defined as an unbounded ``tame'' operator 
provided that one has a good control on the loss of regularity.

Of course, if on the one hand 
in order to achieve such control it is in general not sufficient to give lower bounds on the eigenvalues, on the other hand one surely needs to avoid
zero or ``too small'' eigenvalues. To this end one typically  
uses  parameter modulation. 
Precisely one assumes that $\oo^{(0)}, \Lambda^{(0)}$  depend (non trivially and with some regularity)  on some parameters $\x\in \RRR^{k}$
for some $k$. Unfortunately, equations from physics may come with no such external parameters, so that one needs to extract them from the
initial data: however we shall not address this issue here.

\medskip

An equivalent approach to the Nash-Moser scheme is
 to find  
a change of coordinates $$(y,w)\rightsquigarrow  (\tilde y+ h^{(y)}(\theta),\tilde  w+ h^{(w)}(\theta))$$ such that the push forward of the vector field has an invariant torus at $\tilde y=0$, $\tilde w=0$.
Clearly  the equation for $h$  is  always \eqref{gigi},
hence the solvability conditions on the inverse of the linearized operator appear also in this context.

Following this general strategy  one can look for different types of changes of variables such that the push forward of the vector field has a  particularly simple form, and the existence of an invariant torus follows trivially.
For instance in classical KAM Theorems the  idea is to look for an {\em affine} change of variables  such that not only $y=0,w=0$ is an invariant torus but the vector filed $F$ linearized in the  directions normal to the torus is diagonal.  This means that in the quadratic scheme one not only performs a traslation but also a linear change of variables  which approximately diagonalizes the linearized operator at each step. Naturally this makes the inversion of $d_h\calF$ much simpler.
It is well known that it is possible to diagonalize a finite dimensional matrix if it has distinct eigenvalues.
Then, in order to diagonalize 
the linearized operator at an approximate solution,
one asks for lower bounds on the differences of the eigenvalues (the so called ``Second Mel'nikov'' conditions). Under these assumptions the bounds on the inverse follow by just imposing lower bounds on the eigenvalues (the so called  the First Mel'nikov conditions).
This requirement together with some structural  hypotheses on the system (Hamiltonianity, reversibility, ....) provides existence and  linear stability of the possible solution.
More in general if one wants to cancel non linear terms in the vector fields one needs to add some more restrictive conditions on the linear combinations of the eigenvalues (higher order Mel'nikov conditions). Naturally such conditions are not necessary for  the invertibility,  actually in many applications they cannot be imposed (already in the case of the NLS on the circle the eigenvalues are double) and typically in a Nash-Moser scheme they are not required.

\medskip

Quadratic algorithms for the construction of finite dimensional invariant tori have been used in the literature both in  
 finite, see for instance \cite{Fej} and references therein, or infinite dimensional setting. 
Starting from \cite{K1,W},
this problem has been widely studied in the context of \emph{Hamiltonian} PDEs; a certainly non exaustive list of classical results could be
for instance \cite{CW,Boj2,P1,Po2,KP,B2,Ku2,B3,CY,B5}, in which either the KAM scheme or
the Lyapunov-Schmidt decomposition and the Newton iteration method are used.

 The aforementioned literature is mostly restricted to semilinear PDEs on the circle.
 More recently also other extensions have been considered, such as cases where the spatial variable is higher dimensional,
 or when the perturbation is unbounded.
Regarding the higher dimensional cases, besides the classical papers by Bourgain, we mention also
\cite{GY,EK,BB2,BB1,GYX,PP3,BCP,CHP,GP}, where also manifolds other than the torus are considered.

The first results with
unbounded perturbations can be found in \cite{Ku2,KaP,ZGY,LY,BBiP1,BBiP2}; in these papers the authors follows the classical
KAM approach, which is based on the so-called ``second Mel'nikov'' conditions.
Unfortunately their approach fails in the case of \emph{quasi-linear} PDEs (i.e. when the nonlinearity contains derivatives of the
same order as the linear part). 
This problem has been overcome in 
\cite{IPT,Ba1,Ba2},  for the periodic case, and in \cite{BBM,BBM1}, first for forced and then also for autonomous cases, see also \cite{FP,M,BM1,G}.
The key idea of such papers is to incorporate into the reducibility scheme  techniques coming from pseudo-differential calculus.
The extensions for the autonomous cases are based on the ideas developed  in \cite{BB3},
where the Hamiltonian structure is exploited
 in order to extend results on forced equations to autonomous cases.

In all the results mentiond so far, the authors deal with two problems: the convergence of the iterative scheme,
 and the invertibility of the linearized operator in a neighborhood of the expected solution.
 Typically these problems are faced at the same time, by giving some non-degeneracy conditions on the spectrum
 of the linearized operator in order to get estimates on its inverse.
  However, while the bounds on the linearized operator clearly depend on the specific equation one
  is dealing with, the convergence of the scheme is commonly believed to be adjustable case by case.
 
Our purpose is  to separate the problems which rely only on abstract properties of the vector fields
from those depending on the particular equation under study.

Our point of view is to look for a change of coordinates, say $\Psi$, such that the push-forward of the vector field has an invariant torus at the origin. In fact all the 
results described above can be interpreted in this way. Typically one chooses a priori 
a \emph{group} $\mathcal G$ of changes of coordinates in which one looks for $\Psi$. Then, for many choices of $\mathcal G$, one may impose smallness conditions on the perturbation (depending on the choice of $\mathcal G$) and perform an iterative scheme which produces $\Psi$, provided that the parameters $\xi$ belong to some {\em Cantor like} set (again depending on the choice of $\mathcal G$). 
In this paper we impose some mild conditions on the group $\mathcal G$ such that an iterative algorithm can be performed. Then we explicitly state the {\em smallness conditions} and the  conditions on the parameters. Then, in Section \ref{appli} we show some particularly relevant choices of $\mathcal G$ (in fact we always describe the {\em algebra} which generates $\mathcal G$) and the  resulting
{\em Cantor like} sets. 
In this way, in order to apply our theorem to a particular vector field, one has to: first choose a group, then verify the smallness conditions and finally check that the  {\em Cantor like} set is not empty. 
As it might be intuitive, the simpler the structure of $\mathcal G$ the more complicated is to prove that the resulting {\em Cantor like} set in not empty. 
%
%



The present paper is mainly inspired by the approach in \cite{BB3}, but we follow a strategy more similar to the one of \cite{M1}.
In particular this allows us to cover also non-Hamiltonian cases, which require different techniques w.r.t. \cite{BB3}; compare Subsections \ref{exHam} and \ref{pensavo}.
We essentially produce an algorithm which interpolates a Nash-Moser scheme and KAM scheme. 
On the one hand we exploit the functional approach of the Nash-Moser scheme, which allows to use and preserve the 
``PDE structure'' of the problems; on the other hand we leave the freedom of choosing a convenient set of coordinates 
during the iteration (which is typical of a KAM scheme).
This allows us to deal with  more general classes of vector fields
and with analytic nonlinearities. This last point is particularly interesting in applications to quasi-linear PDEs, where
the only results in the literature are for Sobolev regularity; see \cite{FP2}.
In fact, we develop a formalism which allows us to cover cases with both analytic or Sobolev regularity, by exploiting
the properties of {\it tame} vector fields, introduced in Definition \ref{tame} and discussed in Appendix \ref{tameprop}.

Another feature of our algorithm is related to the ``smallness conditions''; see Constraints \ref{sceltapar} and \ref{sceltaparbis} and the
assumption \eqref{sizes}.
Clearly in every application the smallness is  given by the problem, and one needs adjusts the algorithm
accordingly. Again, while this is commonly believed to be easy to achieve, our point of view is the opposite, i.e. finding the mildest possible condition
that allow the algorithm to converge, and use them only when it is necessary. Of course if on the one hand this makes the conditions
intricated, on the other hand it allows more flexibility to the algorithm.

 \medskip

 {\em Description of the paper.}
 Let us discuss more precisely the aim of the present paper. We consider vector fields of  the form
 \begin{equation}\label{adattato}
\left\{
\begin{aligned}
&\dot\theta=F^{(\theta)}:=\omega^{(0)}(\x)+G^{(\theta)}(\theta,y,w)\\
&\dot y = F^{(y)}:=G^{(y)}(\theta,y,w)\\
&\dot w = F^{(w)}:=\Lambda^{(0)}(\x)w+G^{(w)}(\theta,y,w)
\end{aligned}
\right.
\end{equation}
where $N_0= \omega^{(0)}\cdot \del_\theta + \Lambda^{(0)} w \del_w$ and $G$ is a {\em  perturbation}, i.e. \eqref{adattato} admits an approximately invariant torus.
 Note that this does not necessarily mean that $G$ is small, but only that it is approximately tangent to the torus.
Recall that $\xi\in \calO_0\in \RRR^k$ is a vector of parameters.

Then the idea, which goes back to Moser \cite{M1}, is to find a change of coordinates
such that in the new coordinates the system \eqref{adattato} takes the form 
 \begin{equation}\label{adattato2}
\left\{
\begin{aligned}
&\dot\theta=\omega(\x)+\tilde{G}^{(\theta)}(\theta,y,w)\\
&\dot y = \tilde{G}^{(y)}(\theta,y,w)\\
&\dot w = \Lambda^{(0)}(\x)w+\tilde{G}^{(w)}(\theta,y,w)
\end{aligned}
\right.
\end{equation}
with $\oo\sim \oo^{(0)}$, the average of $\tilde{G}(\theta,0,0)$ is zero and $\tilde{G}^{(\tv)}(\theta,0,0)\equiv0$ for $\tv=y,w$.
More precisely in our main Theorem \ref{thm:kambis} we prove the convergence of an iterative algorithm which
provides a change of variables transforming \eqref{adattato} into \eqref{adattato2}, for all choices of $\xi$ in some explicitly
defined set $\calO_\infty$ (which however might be empty).
 
The changes of variables are not defined uniquely, and one can specify the problem by -- for instance -- identifying further terms in the Taylor expansion
 of $\tilde{G}$ w.r.t. the variables $y$ and $w$
which one wants to set to zero. Of course different choices of changes of variables modify the set $\calO_\infty$ so that in the
 applications it is not obvious to understand which is the best choice. In fact finding the setting in which one is able to
  prove that $\calO_\infty$ is non empty and possibly of positive measure is the most difficult part in the applications. We do not address this problem at all.
Our aim is instead to study very general  classes of changes of variables and find  general Hypotheses on the functional setting, the
vector field under study and the terms of the Taylor series that one wants to cancel, under which
 such an algorithm can be run, producing an explicit set $\calO_\infty$. 
 
 In particular in our phase space $\TTT^d_s\times \CCC^{d_1}\times \ell_{a,p}$ we {\em do not distinguish} the 
 cases where either $s$ or $a$ are equal to zero (Sobolev cases) from the analytic cases. In the same spirit we do not require that the vector field is analytic but only that it is $C^q$ for some large finite $q$. 
 The key ingredients of the paper are the following. 
 
 \noindent
 {\bf Tame Vector fields}. \label{uno} We require  that $F$  is $C^k$-{\em tame up to order $q$}, see Definition \ref{tame},  namely it is tame 
 together with its Taylor expansion up to finite order $k$ w.r.t. $y,w$ and it is regular up to order $q$ in $\theta$, see Subsection \ref{polipo}.
 We make this definition quantitative by denoting a {\em tameness constant} for $G$ by $C_{\vec v,p}(G)$, here $\vec v$ contains all the information relative to the domain of definition of $G$, while $p$ gives us the Sobolev regularity.
  In Appendix \ref{tameprop} we  
 describe some properties of  tame vector fields which we believe are interesting for themselves.
 Finally  our  vector fields are not necessarily bounded, instead they may lose some regularity, namely we allow
 \begin{equation}\label{foffi}
F: \TTT^d_s\times  \big\{ y \in \CCC^{d_1} \, : \, |y |_1 < r^\mathtt s \big\}\times \big\{w\in \ell_{a,p+\nu}\,:\, \|w\|_{a,\gotp_1}<r\big\}\to \CCC^d \times \CCC^{d_1} \times \ell_{a,p}
 \end{equation}
for some fixed $\nu\geq 0$. 
  The properties we require  are very general and are satisfied by large class of PDEs,
  for instance it is well known that these properties 
are satisfied by a large class of composition operators on Sobolev spaces; see \cite{Moser-Pisa-66, KaM}.
 \\
{\bf The  $({\mathbb \NN,\calX,\RR})$ decomposition.} We choose a subspace of polynomials  of maximal degree $\mathtt n$, which we call $\calX$,  containing all the terms we wants to ``cancel out'' from $G$. This space contains the algebra of the changes of variables we shall apply.
Clearly the subspace $\calX$ must be chosen so that a vector field with no terms in $\calX$ possesses an invariant torus. 
In order to identify the part of $F$ belonging to $\calX$ we Taylor-expand it about $y=0$, $w=0$:
{since $F$ is assumed to be a $C^{q}$ vector field, this obviously requires that $q$ is larger than $\mathtt n $ i.e. the maximal degree of the monomials in $\calX$.} 
With some abuse of notation (see Definition \ref{mamm} and comments before it) we denote this operation as a projection $\Pi_{\calX}F$.\\
We  also define a space of polynomial vector fields $\NN$ (which does not intersect $\calX$) such that $N_0\in \NN$.   We allow a lot of freedom on the choice of $\NN$, provided that it satisfies some  rather general hypotheses, in particular all vector fields in $\NN$ should have an invariant torus at zero, and $\NN$ should contain the unperturbed vector field $N_0$; in fact we shall require a stronger condition on $N_0$, i.e. that it is diagonal; see the example in \eqref{poesse} and Definition \ref{norm} for a precise formulation.
\\
Our space of $C^k$-tame vector fields is then decomposed uniquely as $\calX\oplus\NN\oplus\RR$, and we may write 
$$
(\uno-\Pi_\calX)F = N+R
$$ 
where $N=\Pi_\NN F$  while $R=\Pi_\RR F$ is a remainder.


%

\noindent {\bf The Invariant subspace $\cal E$. } We choose  a  space of vector fields $\calE$ (see Definition \ref{nx}) where we want our algorithm to run. Such space  appears naturally in the applications 
where usually one deals with special structures (such as Hamiltonian  or reversible structure) that  one wishes to preserve throughout  the algorithm. 
As one might expect, the choice of the space $\calE$ influences the set $\calO_\infty$.
 In the applications to PDEs the choice of the space $\calE$ is often quite subtle:
 we give some examples in Section \ref{appli}. 
 \\
 {\bf Regular vector fields.}  We choose a subspace of polynomial vector fields, which we denote by {\em regular vector fields}  and endow with a Hilbert norm $|\cdot |_{\vec v,p}$ with the only restriction that they should satisfy a set of properties detailed in Definition \ref{linvec-abs}, for instance we require that all regular vector fields are tame with tameness constant 
 equal to the norm $|\cdot |_{a,p}$ 
 and moreover that for $p=\gotp_1$ this tameness constant is {\em sharp}.
 Throughout our algorithm we shall apply close to identity  changes of variables which preserve $\calE$ and are generated by such vector fields (this is the group $\mathcal G$ of changes of variables).
 This latter condition can probably be weakened but, we believe, not in a substantial way: 
 on the other hand it is very convenient throughout the algorithm. 
  
%


\begin{metate}
We  fix any decomposition $(\NN,\calX,\RR)$, any space $\calE$ and any space of regular vector fields $\calA$ provided that they satisfy Definitions \ref{nomec}, \ref{nx} and \ref{linvec-abs}. \\
 We assume that $F= N_0+G$ belongs to $\calE$, is $C^{\mathtt n +2}$ tame (the value of $\mathtt n$ being fixed by $\calE$) and that $\Pi_\calX F=\Pi_{\calX}G$ is appropriately small while $(\uno-\Pi_{\calX})F$ is ``controlled''.
We look for a change of coordinates $\calH_{\io}$ such that for all $\x\in\calO_{\io}$
one has\footnote{given a diffeomorphism $\Phi$ one defines the push-forward of a vector field $F$
	as $\Phi_{*}F= d\Phi(\leftinv\Phi)[F(\leftinv{\Phi})]$. }
$$
\Pi_{\calX}(\calH_{\io})_{*}F\equiv0\,.
$$
\end{metate}
\noindent At the purely formal level,
a change of coordinates $\Phi$ generated by a bounded  vector field $g$, transforms 
$F$ into
\begin{equation}\label{thorn}
\Phi_*(F)\sim F+[g,F]+ O(g^2)\sim e^{[g,\cdot]}F.
\end{equation}

We find the change of variable we look for via an iterative algorithm; at each step
we need $\Phi$ to be such that $\Pi_\calX (\Phi_*F)=0$ up to negligible terms,
 so we need to find $g$ such that
$$
\Pi_\calX (F+ [g,F])=0\,;
$$
in other words we need to invert the operator $\Pi_\calX[F,\cdot]$. 
Since one expects $g\sim X:= \Pi_{\calX}F$ (which is assumed to be small) then, at least formally, the term 
 $[g,\Pi_{\calX}F]$ is negligible and one needs to solve
 \begin{equation}\label{eqhomolog}
 \Pi_\calX([(1-\Pi_\calX)F,g])-X:=u
 \end{equation}
with $g\sim X$ and $u\sim X^2$. Equation \eqref{eqhomolog} is called \emph{  homological equation} and 
in order to solve it one needs the ``approximate invertibility'' for the operator
\begin{equation}\label{eqhom1}
 \gotA(\cdot):= \Pi_\calX[(1-\Pi_\calX)F,\cdot].
 \end{equation}
 
 Then the iteration is achieved by setting $\Psi_0:=\uno$ and
 \begin{equation}\label{iteraz}
 \Psi_{n}:=\Phi^{1}_{g_{n}}\circ\Psi_{n-1}\,,\quad F_{n}:=(\Psi_{n})_{*}F
 \end{equation}
  where $\Phi^{1}_{g_{n}}$ is the time-$1$ flow map generated by the vector field $g_{n}$ which in turn solves the   homological equation 
 \begin{equation}\label{hom1}
 \Pi_\calX([(1-\Pi_\calX)F_{n-1},g_n])-X_{n-1}:=u_n.
 \end{equation}
   Since we need to preserve the structure, namely we need that at each step $F_{n}\in \calE$, then at each step the change of variables $\Phi^1_{g_{n}}$ should
   preserve $\calE$. In fact  we  require   that $g_n$ is a {\em regular } vector field,

 In order to pass from the formal level to the convergence of the scheme
 we need to prove that 
 $g_n$ and $u_n$ satisfy appropriate  bounds. 
 \\
  {\bf Homological equation.} We  say that a set of parameters $\calO$  satisfies the   homological equation  ( for $(F,K,\vec v^0,\rho)$) if there exist  $g,u$  which satisfy \eqref{eqhomolog} with appropriate bounds (depending on the parameters $K,\vec v^0,\rho$; $\vec v^0$ controls the domain of definition, $\rho$ controls the size of the change of variables and $K$ is an {\em ultravioled cut-off}), see  Definition \ref{pippopuffo2}.
 Since  $F$ is a merely differentiable vector field the bounds are  delicate since expressions like \eqref{thorn} may be meaningless
 in the sense that --apparently-- the new vector field $F_{n+1}$ is less regular than $F_n$, and hence it is not obvious that one can iterate the procedure.
Indeed the commutator loses derivatives and thus one cannot use Lie series expansions in order to describe the change of variables.
However one can  use Lie series expansion formula on  polynomials, such as $\Pi_{\calX}\Phi_{*}F$. 
  We show that, provided that $X$ is small  while $R,N-N_0$ are appropriately 
  bounded, we obtain a converging KAM algorithm. 
  \\
  Note that the smallness of $X$ implies the existence 
  of a sufficiently good approximate solution; on the other hand, we only need very little control on
   $(\uno-\Pi_\calX)F_n$, which results on very weak (but  quite cumbersome) assumptions on $R$ and especially on $N-N_0$; see \eqref{cioffame}
   and Remark \ref{concreto?}.
     \\
{\bf Compatible changes of variables.}
 It is interesting to notice that at each step of  the iterative scheme \eqref{iteraz} we may apply any   change of variables $\calL_n$ with the only condition that it does not modify the bounds, i.e.
 $\Pi_\calX (\calL_n)_\star F_{n-1} \sim \Pi_\calX  F_{n-1}$ (and the same for the other projections).  We formalize this idea in  Definition \ref{compa} where we introduce the changes of variables compatible with $(F,K,\vec v^0,\rho)$. 
 Then  we may set
  \begin{equation}\label{iteraz2}
  \calH_{n}=\Phi^{1}_{g_{n}}\circ\calL_n \circ \calH_{n-1}\,,\quad F_{n}:=(\calH_{n})_{*}F
  \end{equation} 
  and the algorithm is still convergent.
  \\
   This observation is essentially tautological but it might well be possible that in a new set of coordinates
  it is simpler to invert the operator $\gotA_{n}:= \Pi_\calX{\rm ad}((1-\Pi_\calX) F_{n-1})$ (for instance $\gotA_{n}$ may be diagonal up to a negligible remainder, see Subsection \ref{reduco}). 
  Note that since the approximate invertibility of 
   $\gotA_n$ is in principle   independent from the coordinates  (provided the change does not lose regularity), if one knows that a change of 
   variables simplifying $\gotA_n$ exists, there might be no need to apply it in order to deduce bounds on the approximate inverse. 
   This is in fact the strategy of the papers \cite{BBM,BBM1}, where the authors study fully nonlinear equations and prove existence 
   of quasi-periodic solutions with Sobolev regularity. On the other hand one might modify the definition of the subspace $\calX$ in such a 
   way that $\calL_n$ is the time one flow of a  regular bounded vector field in $\calX$. The best strategy clearly depends on the 
   application, so we leave the $\calL_n$ as an extra degree of freedom. 
We can summarize our result as follows, for the notations we refer to the  informal Definitions written above.
 \begin{theorem}\label{thm:kambis00}
 Fix $\nu\ge0$ as in \eqref{foffi}, and 
fix a decomposition $(\NN,\calX,\RR)$ , a subspace $\calE$   and a space of regular vector fields $\calA$. 	
Fix parameters $\e_0,\tR_0,\tG_0,\gotp_2$  satisfying appropriate constraints. Let $N_0$ be a diagonal vector field  and consider a  vector field
 	\begin{equation*}
 	F_0:=N_{0}+G_0 \in \calE
 	\end{equation*} 
 	which is  $C^{\mathtt n+2}$-tame up to order $q=\gotp_2+2$.
 	
 	Fix $\g_0>0$ and assume that
 	\begin{equation}\label{cioffame}
 	\g_{0}^{-1}C_{\vec{v}_0,\gotp_2}(G_0) \le \tG_0\,,\quad
 	\g_0^{-1}C_{\vec{v}_0,\gotp_2}(\Pi_{\NN}^\perp G_0)\le \tR_0\,, \quad
 	\g_0^{-1}|\Pi_{\calX}G_0|_{\vec{v}_0,\gotp_1}\le \e_0\,,
 	\quad  \g_0^{-1}|\Pi_{\calX}G_0|_{\vec{v}_0,\gotp_2}\le \tR_0\,,
 	\end{equation}
 	here $C_{\vec v,p}(G)$ is a tameness constant, while $|\cdot|_{\vec v,p}$ is the norm on regular vector fields.
 	

 	\medskip
 	
  \noindent	For all $n\geq 0$  we define recursively  changes of variables $\calL_n,\Phi_n$  and  compact sets
  $ \calO_n$ as follows.\\
  	Set $\HH_{-1}=\HH_0=\Phi_0=\calL_0=\uno$, and for $0\le j \le n-1$ set recursively
 	$\HH_j= \Phi_j\circ \calL_j\circ \HH_{j-1}$ and  $F_j:=(\HH_j)_{*}F_0:=N_0+G_{j}$.
 	Let $\calL_{n}$ be any {\rm compatible change of variables}    for $(F_{n-1},K_{n-1}, \vec v_{n-1},\rho_{n-1})$ 
 	and  $\calO_n$ be any compact set 
 	\begin{equation*}
 	\calO_{n}\subseteq \calO_{n-1}\,,
 	\end{equation*}  
 	which {\rm satisfies the   homological equation} for $((\calL_n)_*F_{n-1},K_{n-1},\vec v^{\text{\tiny 0}}_{n-1},\rho_{n-1})$, let $g_n$ be the
	solution of the   homological equation and $\Phi_n$ the time-1 flow map generated by $g_n$.
 	\\
	Then the sequence $\calH_n$ converges
 	for all $\x\in\calO_0$ 
 	to some change of variables
 	\begin{equation*}
 	{\calH}_\io={\calH}_\io(\x):  D_{{a_{0}},p}({s_{0}}/{2},{r_{0}}/{2})\longrightarrow D_{\frac{a_{0}}{2},p}({s_{0}},{r_{0}}).
 	\end{equation*}
 	
 	such that defining $F_{\infty}:=(\calH_\io)_{*}F_0$ one has 
 	\begin{equation*}
 	\Pi_\calX F_{\infty}=0
 	\quad \forall \xi \in \calO_\io:=\bigcap_{n\geq0}\calO_n
 	\end{equation*}
 	and 
 	$$
 	\g_0^{-1}C_{\vec{v}_\io,\gotp_1}(\Pi_{\NN}F_{\infty}-N_0)\le 2\tG_0, \quad
 	\g_0^{-1}C_{\vec{v}_\io,\gotp_1}(\Pi_{\RR}F_{\infty})\le 2\tR_0
 	$$
 	with $\vec{v}_\io:=(\g_0/2,\calO_\io,s_{0}/2,a_{0}/2)$. \end{theorem}
  \smallskip
While the scheme is quite general, as a drawback the set of  parameters $\xi$ for which the invariant torus exists is defined in a very 
complicated way, in terms of the approximate invertibility
of the operators  $\gotA_n$. 

In order to get a simpler description of the good parameters we may require that $(\NN,\calX,\RR)$ is a  ``triangular decomposition''
 i.e.  there exists a decomposition $\calX=\oplus_{j=1}^\mathtt b \calX_j$ such
 that for all $N\in \NN$ the action of  the operator $\gotN:=\Pi_{\calX}[N,\cdot]$ is block diagonal  while  for all $R\in \RR $ the action 
 of  $\gotR:=\Pi_{\calX}[R,\cdot]$ is strictly upper triangular, see Definition \ref{triang}. In Proposition \ref{uffffa} we show that under such hypotheses the problem of solving the   homological equation \eqref{hom1}  is reduced to proving the approximate invertibility of  $\gotN$, the so called {\em Melnikov conditions}, which is typically much simpler to analyze in the applications.
 Indeed solving \eqref{hom1} amounts to inverting $\gotA$ but since
 $\gotR$ is upper triangular and $\gotN$ is diagonal, the Neumann series
 $$
 \gotA^{-1}=(\gotN+\gotR)^{-1} = \gotN^{-1}\sum_{j}(-1)^j(\gotN^{-1}\gotR)^j
 $$
is  a finite sum.

Note that in order to produce a triangular decomposition one can associate degrees  to $w,y$ (say resp. $1$ and $\mathtt d>0$)   
in order to induce a degree decomposition (see Remark \ref{sonfica}) which gives to the space of polynomial vector fields a graded Lie algebra structure. 
By Lemma \ref{pillo} triangularity is achieved when  $\NN$ contains only terms with degree $ 0$ while $\RR$ contains only  terms with  
degree $>0$ while  $\calX=\oplus_{j=1}^\mathtt b \calX_j$ is the decomposition of $\calX$ in subspaces of increasing degree. Note that 
this can be done for any choice for $\mathtt d$.

Another natural way to understand the nature of the subspaces $\calX,\NN,\RR$  is through ``rescalings''. This means introducing a special degree
decomposition where the degree of $y$ is the same $\mathtt s$ as the one used in the definition of the domains in the phase space, see \eqref{foffi}. 
 This latter degree decomposition separates terms which behave differently under rescaling of the order on magnitude of the domains $r\to \la r$.
Note that in this way $\calX$ contains terms with negative scaling,  $\NN$ contains the terms with scaling zero and finally $\RR$  contains 
terms with positive scaling. 

Typically the invertibility of $\gotN$ relies on non degeneracy conditions on the eigenvalues which provides a lower bounds on the small divisors.
The size $\g_0$ of such denominators is essentially given by the problem. A vector  field is considered a perturbation if its size is small with respect  to the size of the small divisor. Hence in giving the smallness conditions on $G$ one must find a modulation between the size of the remainder $R$,
which can be made small by a rescaling, and the size of $\calX$ which   grows under the rescaling.
The polynomial vector fields in $N$ are more delicate. Indeed some such terms do not change under rescaling. 
Of course $\Pi_{\NN}G$ should  be much smaller with respect to  $N_0$ but  in fact one does not need that
$\Pi_{\NN}G$ is small with respect to $\g_0$ provided that $\Pi_\calX G$ is sufficiently small.
This  further justifies  why  the smallness conditions on the
vector field $G$  are  imposed separately on each term. 
We refer the reader to Paragraph \ref{riscalo} for more details.

In Section \ref{appli} we discuss various applications and examples 
and we show how our algorithm allows to interpolate between the Nash-Moser algorithm and the classical KAM one.  
Example 1. is the classic Nash-Moser approach, where one fixes the subspace $\calX$ to be as simple as possible. \\
 In Example 2.  we study Hamiltonian vector fields, and exploit the Hamiltonian structure in order to simplify the Melnikov conditions; this is a reinterpretation to our setting of the strategy of \cite{BB3}.  
 We conclude subsection \ref{exHam} by collecting our results into Theorem \ref{esempappli}.
 \\
In Example 3.   we only assume that our vector fiels is reversible and we simplify the Melnikov conditions by making an appropriate choice of the sets $\calX,\NN,\RR$, this is a new result which we believe should enable us to prove existence of quasi-periodic solutions  in various settings,  see \cite{FP2} for the case of the fully nonlinear NLS on the circle.
 We conclude subsection \ref{pensavo} by collecting our results into Theorem \ref{esempappli2}.
\\
In Examples 2. and 3.    our formulation essentially   decouples the dynamics on the approximate invariant torus, 
which is given by the equation for $\theta$ and $y$,
and the dynamics in the normal direction $w$.  
More precisely in these cases the invertibility of $\gotN_n$ follow from the conditions:
\begin{itemize}
\item The frequency vector $\oo_n(\x):= \langle F_n^{(\theta)}(\theta,0,0)\rangle$  needs to be diophantine;
\item The operator $\gotL_n:=\oo_n\cdot\del_{\theta}+d_{w}F^{(w)}(\theta,0,0)$ acting on $H^p(\TTT_s^d, \ell_{a,p})$ must be ``approximatively invertible''. 
\end{itemize}
Note that  $\gotL_n$ has the same form 
of the linearized operator of a forced equation, namely the case where $F^{\theta}= \oo\cdot\del_\theta$ and the frequency vector $\oo$ plays the r\^ole of
an  external parameter.
Moreover if $F$ is a vector field coming from a PDE (possibly after one step of Birkhoff Normal Form),
then $\gotL_n$ differs from the linearization of a composition operator by a finite rank term, this is an essential property in the study of quasi-linear PDEs.\\
 In Example 4 we prove a KAM theorem for a class of Hamiltonian vector fields corresponding to the classical paper \cite{Po2},  but requiring only finite differentiability and imposing milder smallness conditions, comparable  to those of \cite{BB}.
 We conclude subsection \ref{reduco} by stating a result on the existence of reducible invariant tori; this is the finitely differentiable version of \cite{Po2,BB}.

Finally in subsection \ref{exNLS} we discuss how to apply Examples 2 and 4 to an NLS with Fourier multipliers, respectively in the 
case of a Lie group and of $[0,\pi]$ with Dirichlet boundary conditions.

\medskip

\noindent
{\bf Acknowledgements}. We thank L.  Biasco, F. Giuliani, J. Massetti and R. Montalto for useful comments on the manuscript.
This research was  supported by the European Research Council under
FP7 ``Hamiltonian PDEs and small divisor problems: a dynamical systems approach'',
by PRIN2012 ``Variational and perturbative aspects of non linear differential problems" and McMaster University.

\zerarcounters

\section{ Functional setting and  main result }
\label{secNM}

In this paragraph we introduce all the relevant notation and tools we need. In particular we define our phase space, the  vector fields we will deal with and the type of change of variables we need in order to perform our KAM algorithm, as explained in the introdution.

\subsection{The Phase Space}\label{phase}
Our starting point is an infinite dimensional space with a product structure $V_{a,p} := \CCC^d \times \CCC^{d_1} \times \ell_{a, p}$. 
Here $\ell_{a,p}$ is a scale of  separable Hilbert spaces endowed with norms $\|\cdot \|_{a,p}$, 
in particular this means that $\|f\|_{a,p}\leq \|f\|_{a',p'}$ if $(a,p)\leq (a',p')$ lexicographically.
\noindent
\begin{hypo}\label{hyp22}
	The space $\ell_{0,0}$ is  endowed with a bilinear scalar product
	$$f,g\in \ell_{0,0} \mapsto f\cdot g \in \CCC.$$
	The scalar product identifies the dual $\ell_{a,p}^*$ with $\ell_{-a,-p}$ and is 
	such that
	\begin{equation}\label{sonasega}
	\|w\|_{0,0}^2= w\cdot \ol{w}\,,\quad |g \cdot f|\leq   \|g\|_{a,p}\|f\|_{-a,-p}\,\,\quad |g \cdot f|\leq   \|g\|_{0,0}\|f\|_{0,0}\leq \|g\|_{a,p}\|f\|_{0,0}.
	\end{equation}
	We denote the set of variables $\mathtt{V} := \big\{ \theta_1,\ldots, \theta_d, y_1,\ldots, y_{d_1}, w  \big\}$.	Moreover we make the following assumption on the scale  $\ell_{a,p}$.
We assume that there is a non-decreasing family $(\ell_K)_{K \geq 0}$ of closed subspaces of $\ell_{a,p}$ such that
$\cup_{K \geq 0} \ell_K$ is dense in $\ell_{a,p}$ for any $p\geq 0$, and that there are projectors
\begin{equation}\label{duck3}
\Pi_{\ell_{K}} : \ell_{0,0} \to \ell_{K}, \quad \Pi_{\ell_K}^{\perp}:=\uno-\Pi_{\ell_{K}},
\end{equation}
such  for all  $p,\al,\be\geq0$ there  exists a constant $\mathtt C=\mathtt C(a,p,\al,\be)$ such that  one has
\begin{subequations}
\begin{align}
&\|\Pi_{\ell_K}w\|_{a+\al,p+\be}\leq \mathtt C e^{\al K}K^{\be }\|w\|_{a,p}\;\; \forall w\in\ell_{a,p},\label{duck33}\\
&\|\Pi_{\ell_K}^{\perp}w\|_{a,p}\leq \mathtt C e^{-\al K}K^{-\be} \|w\|_{a+\al,p+\be}, \;\;\; \forall w\in \ell_{a+\al,p+\be}.\label{duck34}
\end{align}
\end{subequations}
\end{hypo} 

We shall need two parameters, $\gotp_0<\gotp_1$. Precisely $\gotp_0>d/2$ is needed in order to have the
Sobolev embedding and thus the algebra properties, while $\gotp_1$ will be chosen very large and is
needed in order to define the phase space.

\begin{defi}[Phase space]\label{spazio} Given $\gotp_1$ large enough,
we consider the toroidal domain
\begin{equation}\label{Dsr}
\TTT^d_s \times D_{a,p}(r) :=\TTT^d_s\times  B_{r^2} \times \B_{r,a,p,\gotp_1}\,,
 \subset V_{a,p}
\end{equation}
where 
\begin{equation}\nonumber
\begin{aligned}
&\TTT^d_s := \big\{ \theta \in \CCC^d \, : \, {\rm Re}(\theta)\in\TTT^{d},\ \max_{h=1, \ldots, d} |{\rm Im} \, \theta_h | < s
 \big\} \,, \\
&B_{r^\mathtt s}  := \big\{ y \in \CCC^{d_1} \, : \, |y |_1 < r^\mathtt s \big\}\,, \qquad
\B_{r,a,p,\gotp_1}:=\big\{w\in \ell_{a,p}\,:\, \|w\|_{a,\gotp_1}<r\big\}\,,
 \end{aligned}
 \end{equation}
and we denote by $ \TTT^d := ( \RRR/2\pi \ZZZ)^d $ the $ d $-dimensional torus.
\end{defi}

\begin{rmk}
Note that $\B_{r,a,p,\gotp_1}$ is the intersection of $\ell_{a,p}$ with the ball of radius $r$ in $\ell_{a,\gotp_1}$,
thus our phase space clearly depends on the parameter $\gotp_1$. We drop it in the notations since it will be
fixed once and for all, while $a,p,s,r$ vary throughout the algorithm and we carefully need to keep track of them.
\end{rmk}

Fix some numbers $s_0,a_0\geq 0$ and $r_0>0$. Given $s\leq s_0$, $a,a'\leq a_0$, $r\leq r_0$,
$ p,p'\geq \gotp_0$.
We endow the space $V_{a,p}$ with the following norm. For $u=(u^{(\theta)},u^{(y)},u^{(w)})\in \TTT^d_s \times D_{a,p}(r)$
\begin{equation}\label{normaV}
\|u\|_{V_{a,p}}=\frac{1}{\max\{1,s_0\}}\|u^{(\theta)}\|_{\CCC^{d}}+\frac{1}{r_0^{(\mathtt{s})}}\|u^{(y)}\|_{\CCC^{d_1}}+\frac{1}{r_0}\|u^{(w)}\|_{\ell_{a,p}},
\end{equation}

Now  consider maps
\begin{equation}\label{faglimale}
\begin{aligned}
 f : \TTT^d_s \times D_{a',p'}(r)  & \to V_{a,p} \\
 (\theta,y,w) &\to ( f^{(\theta)}(\theta,y,w), f^{(y)}(\theta,y,w),f^{(w)}(\theta,y,w)) \,,
\end{aligned}
\end{equation}
with
$$
f^{(\mathtt v)}(\theta,y,w) = \sum_{l\in \ZZZ^d} f^{(\mathtt v)}_l(y,w)e^{\ii l \cdot \theta}\,,\quad
\mathtt v\in{\mathtt V} \,, 
$$
where $f^{(\mathtt v)}(\theta,y,w) \in \CCC$ for $\mathtt v=\theta_i,y_i$ while $f^{(w)}(\theta,y,w) \in  \ell_{a,p}$.
We shall  use also the notation $f^{(\theta)}(\theta,y,w)\in \CCC^d$, $ f^{(y)}(\theta,y,w)\in\CCC^{d_1}$.

\begin{rmk}\label{perche}
We think of these maps as families of torus embeddings from $\TTT^d_{s}$ into $V_{a,p}$ depending
parametrically on $y,w\in D_{a',p'}(r)$, and this is the reason behind the choice of the norm; see below.
\end{rmk}

We define a norm (pointwise  on $y,w$)\ by setting
\begin{equation}\label{totalnorm}
\|f\|_{s,a,p}^2:=\|f^{(\theta)}\|^2_{s,p}+\|f^{(y)}\|^2_{s,p}+\|f^{(w)}\|^2_{s,a,p}
\end{equation}
where
\begin{equation}\label{equation}
\|f^{(\theta)}\|_{s,p}:=\left\{\begin{aligned} &\frac{1}{s_{0}}\sup_{i=1,\ldots,d}\|f^{(\theta_{i})}(\cdot,y,w)\|_{H^p(\TTT^d_s)}
\quad s\leq s_0\neq 0\\
&\sup_{i=1,\ldots,d}\|f^{(\theta_{i})}(\cdot,y,w)\|_{H^{p}(\TTT^{d}_{s})}\,,\quad s=s_0=0\end{aligned}\right.
\end{equation}
\begin{equation}\label{equation2}
\|f^{(y)}\|_{s,p}:=\frac{1}{r_0^\mathtt s}\sum_{i=1}^{d_1}\|f^{(y_{i})}(\cdot,y,w)\|_{H^{p}(\TTT^{d}_{s})}
\end{equation}
\begin{equation}\label{equation3}
\|f^{(w)}\|_{s,a,p}
:=\frac{1}{r_0}\left(\|f^{(w)}\|_{H^{p}(\TTT^{d}_{s};\ell_{a,\gotp_0})}
+\|f^{(w)}\|_{{H^{\gotp_0}(\TTT^{d}_{s};\ell_{a,p})}}
\right),
\end{equation}
where $H^{p}(\TTT^{d}_{s})=H^{p}(\TTT^{d}_{s}; \CCC)$ is the standard space of analytic functions in the strip of size $s$ which are Sobolev on the boundary with norm
 \begin{equation}\label{normasob}
 \|u(\cdot)\|_{H^{p}(\TTT^{d}_s)}^{2}:=\sum_{l\in\ZZZ^{d}}|u_l|^{2}e^{2s|l|}\langle l \rangle^{2p}, \qquad \av{ l}:=\max\{1,|l|\}.
 \end{equation}
 If $s=0$ clearly one has that $H^{p}(\TTT_{s}^{d})=H^{p}(\TTT^{d})$ is the standard Sobolev space.
 More in general given a Banach space $E$ we denote by $H^{p}(\TTT^{d}_{s};E)$ the space of 
 analytic functions in the strip of size $s$ which are Sobolev in $\theta$ on the boundary
with values in $E$ endowed with the natural norm.
 Note that trivially $\|\partial_\theta^{p'} u\|_{H^{p}(\TTT^{d}_s)}= \|u\|_{H^{p+p'}(\TTT^{d}_s)}$. 

 \begin{rmk}\label{chede}
 We can interpret \eqref{equation3} as follows. We associate $f^{(w)}$ with a function of $\theta$ defined as
 \begin{equation}\label{effepi}
 {\mathtt f}_p(\theta):= \sum_{l\in\ZZZ^d}\|f^{(w)}_{l}(y,w) \|_{a,p} e^{\ii \theta\cdot l}, 
 \end{equation}
 and then we have
 $$
 \|f^{(w)}\|_{s,a,p} = \|{\mathtt f}_{\gotp_0}\|_{H^p(\TTT^d_s)}+ \|{\mathtt f}_{p}\|_{H^{\gotp_0}(\TTT^d_s)}.
 $$
 \end{rmk}

 \begin{rmk}\label{stace}
If $\ell_{a,p}= H^p(\TTT^r_a)$ then fixing $\gotp_0\geq (d+r)/2$ we have that $\|\cdot \|_{s,a,p}$ in \eqref{equation3} is equivalent to $ \|\cdot \|_{H^p(\TTT^d_s\times \TTT^r_a)}$
 \end{rmk}

By recalling the definition in \eqref{normaV} we have that 
\begin{equation}\label{chede2}
\|f\|_{s,a,p}\sim\|f\|_{H^{\gotp_0}(\TTT^{d};V_{a,p})\cap{H^{p}(\TTT^{d};V_{a,\gotp_0})}}:=\|f\|_{H^{p}(\TTT^{d};V_{a,\gotp_0})}+\|f\|_{H^{\gotp_0}(\TTT^{d};V_{a,p})}
\end{equation}

\begin{rmk}\label{stocazzo}
Formula \eqref{totalnorm} depends on the point $(y,w)$, hence it is {\em not} a norm for vector fields
and this is very natural in the context of {\em Sobolev regularity}.
Indeed in the scale of domains $\TTT^d_s\times D_{a,p}(r)$ one controls only the $\gotp_1$ norm of $w$ (see Definition
\ref{spazio}), and hence there is no reason for which one may have
$$
\sup_{(y,w)\in D_{a,p} (r )}\|f\|_{s,a,p}<\io\,.
$$
Naturally if one fixes $ p=\gotp_1$ one may define as norm of $F$ the quantity
$\sup_{(y,w)\in D_{a,\gotp_1}(r)}\|F\|_{s,a,\gotp_1}$.

The motivation for choosing the norm \eqref{totalnorm}  is the following.
Along the algorithm we need to control commutators of vector fields. In the analytic case,
i.e. if $s_0\neq 0$, one may keep $p$ fixed and control the derivatives via
{\em Cauchy estimates} by reducing the analyticity, so  the phase space can be defined in terms of the fixed
$p$. However, since we do not want to add the hypothesis $s_0\neq 0$,
we have to leave $p$ as a parameter and use {\em tameness} properties of the vector field (see
Definition \ref{tame}) as  in the Sobolev Nash-Moser schemes.
\end{rmk}

 It is clear that any $f$ as in \eqref{faglimale} can be identified with ``unbounded'' vector fields
 by writing
 \begin{equation}\label{vectorfield}
 f\leftrightarrow \sum_{\tv\in{\mathtt V}}f^{(\tv)}(\theta,y,w)\del_{\tv},
 \end{equation}
 where the symbol $f^{(\tv)}(\theta,y,w)\del_{\tv}$ has the obvious meaning for $\tv=\theta_i,y_i$ while
 for $\tv=w$ is defined through its action on differentiable functions $G:\ell_{a,p}\to \CCC$ as
 $$
 f^{(w)}(\theta,y,w)\del_w G := d_w G[ f^{(w)}(\theta,y,w)].
 $$
 
 Similarly, provided that $|f^{(\theta)}(\theta,y,w)|$ is small for all $(\theta,y,w)\in  \TTT^d_s \times D_{a,p}(r)  $
 we may lift $f$ to a map \begin{equation}\label{mar}
\Phi:= (\theta +f^{(\theta)},y+f^{(y)}, w+ f^{(w)}) : \, \TTT^d_s \times D_{a',p'} \to
\TTT^d_{s_1} \times \CCC^{d_1}\times \ell_{a,p}\,,\quad \text{for some} \; s_1\ge s\,,
 \end{equation}
 and if we set $\|\theta\|_{s,a,p} := 1$  we can write
  $$
  \|\Phi^{(\tv)}\|_{s,a,p} = \|\tv\|_{s,a,p}+\|f^{(\tv)}\|_{s,a,p} \,,\; \tv=\theta,y,w\,.
  $$
  Note that 
  $$
  \|y\|_{s,a,p}= r_0^{-\mathtt s}|y|_1\,,\quad  \, \|w\|_{s,a,p}=r_0^{-1}\|w\|_{a,p}.
  $$ 
  
  \begin{rmk}\label{azz}
There exists ${\mathtt c}={\mathtt c}(d)$ such that if
 $\|f\|_{s,a,\gotp_1}\le {\mathtt c}\rho$ one has
$$
\Phi: \TTT^d_s \times D_{a+\rho a_0,p}(r ) \to \TTT^d_{s+\rho s_0}\times D_{a,p}(r +\rho r_0).
$$
  \end{rmk}

 We are interested in vector fields defined on a scale of Hilbert spaces; precisely
 we shall fix $\rho,\nu,q\ge0$ and consider vector fields
     \begin{equation}\label{vettorazzi}
  F:\TTT^d_s\times D_{a+\rho a_0,p+\nu}(r)\times \calO\to V_{a,p}\,,
  \end{equation}
  for some $s<s_0$, $a+\rho a_0\le a_0$, $r\le r_0$ and all $p+\nu \le q$. Moreover we require that $\gotp_1$ in Definition \ref{spazio} satisfies $\gotp_1\ge \gotp_0+\nu+1$.

  \begin{defi}\label{leftinverse}  
Fix $0\le\rho,\db\le1/2$, and consider two  differentiable maps $\Phi= \uno +f$, $\Psi= \uno +g$ as in
\eqref{mar} such that for all
$p\ge \gotp_0$,  $2\rho s_0 \le s \le s_0$, $2\rho r_0 \le r \le r_0$ and $0\leq a\leq a_0(1-2\db )$ one has
\begin{equation}\label{laputtana250}
\Phi,\Psi:\TTT^d_{s-\rho s_0} \times D_{a+\db a_0,p}(r-\rho r_0)\to \TTT^d_{s}\times D_{a,p}(r ).
\end{equation}
If
\begin{equation}\label{sinistra}
\begin{aligned}
\uno=\Psi\circ \Phi: \TTT^d_{s-2\rho s_0}\times D_{a+2\db a_0,p}&(r-2\rho r_0)&\longrightarrow &\
\TTT^d_s\times  D_{a,p}(r)\\
&\ (\theta,y,w)&\longmapsto &\ (\theta,y,w)
\end{aligned}
\end{equation}
we  say that $\Psi$ is a left inverse of $\Phi$ and write $\leftinv{\Phi}:=\Psi$.

Moreover fix  $\nu\geq 0$, $0\le\db'\le1/2$. Then for any 
$F : \TTT^d_{s}\times D_{a+\db' a_0,p+\nu}(r) \to V_{a,p}$, with $0\leq a\leq a_0(1-2\db -\db')$, 
we define the ``pushforward'' of $F$ as
\begin{equation}\label{push}
\Phi_* F:= d\Phi(\leftinv\Phi)[F(\leftinv{\Phi})]: \TTT^d_{s-2\rho s_0} \times D_{a+(2\db+\db') a_0,p+\nu}(r-2\rho r_0)  \to 
 V_{a,p}\,.
\end{equation}
\end{defi}

 We need to introduce  parameters $\x\in\calO_0$ a compact set in $\RRR^\mathfrak d$. Given any compact  $\calO\subseteq\calO_0$ 
 we consider  Lipschitz families of vector fields 
   \begin{equation}\label{vettori}
  F:\TTT^d_s\times D_{a',p'}(r)\times \calO\to V_{a,p}\,,
  \end{equation}
  and say  that they  are \emph{bounded} vector fields when  $p=p'$ and $a=a'$.
  Given a positive number $\g$ we introduce  the weighted Lipschitz norm
 \begin{equation}\label{ancoralip}
 \|F\|_{\vec v,p}=\|F\|_{\g,\calO,s,a,p}:=\sup_{\xi\in \calO}\|F(\x)\|_{s,a,p}
 + \g \sup_{\xi\neq \eta\in \calO}\frac{\|F(\x)-F(\eta)\|_{s,a,p-1}}{|\xi-\eta|}\,.
 \end{equation}
 and we shall drop the labels $\vec v=(\g,\calO,s,a)$ when this does not cause confusion. 
 More in general given $E$ a Banach space we define the Lipschitz norm as
  \begin{equation}\label{ancoralip22}
 \|F\|_{\g,\calO,E}
 :=\sup_{\xi\in \calO}\|F(\x)\|_{E}
 + \g \sup_{\xi\neq \eta\in \calO}\frac{\|F(\x)-F(\eta)\|_{E}}{|\xi-\eta|}\,.
 \end{equation}

 \begin{rmk}\label{ck}
 Note that in some applications one might need to assume a higher regularity in $\xi$. In this case it is 
 convenient to define the weighted $q_1$-norm
 $$
 \|F\|_{\vec v,p}=\|F\|_{\g,\calO,s,a,p}:=\sum_{\substack{h\in \NNN^{\mathfrak d} \\
  |h|\leq q_1}}\g^{|h|} \sup_{\xi\in \calO} \|\partial_\xi^h F(\x)\|_{s,a,p-|h|}.
$$
Where the derivatives are in the sense of Whitney.

 Throughout the paper we shall always use the Lipschitz norm \eqref{ancoralip}, although all the properties 
 hold verbatim also for the $q_1$-norm. 
 \end{rmk}

  \begin{defi}\label{vecv}
  We shall denote by
 $\VV_{\vec v,p}$  with $\vec v= (\g,\calO,s,a,r)$ the space of  vector fields as in \eqref{vettorazzi}  with $\db=0$. 
 By slight abuse of notation we denote the norm $\|\cdot\|_{\g,\calO,s,a,p}=\|\cdot\|_{\vec v,p}$ .
 \end{defi}


\subsection{Polynomial decomposition}\label{polipo}

In $\VV_{\vec v,p}$ we identify the closed {\em monomial} subspaces
\begin{equation}\label{sotto}
\begin{aligned}
&\mathcal V^{(\tv,0)} := \{f\in \VV_{\vec v,p}\,:\; f= f^{(\tv,0)}(\theta) \partial_\tv\} \,, \quad \tv\in \mathtt V \\ & 
\mathcal V^{(\tv,\tv')} := \{f\in \VV_{\vec v,p}\,:\; f= f^{(\tv,\tv')}(\theta)[\tv']  \partial_\tv \}\,,\quad
 \tv\in\mathtt V\,,\quad \tv'= y_1,\ldots,y_{d_1},w
\\&  \mathcal V^{(\tv,\tv_1,\dots,\tv_k)} := \{f\in \VV_{\vec v,p}\,:\; f= f^{(\tv,\tv_1,\dots,\tv_k)}(\theta)[\tv_1,\dots,\tv_k]  
\partial_\tv \}\,,\quad  \tv\in\mathtt V\,,\quad \tv_i= y_1,\ldots,y_{d_1},w\,,
\end{aligned}
\end{equation}
where $\tv_1,\ldots,\tv_k\in{\mathtt U}:=\{y_1,\ldots,y_{d_1},w\}$ are ordered and $f^{(\tv,\tv_1,\ldots,\tv_k)}$ is a multilinear
form symmetric w.r.t. repeated variables.

As said after \eqref{faglimale}  it will be convenient to use also vector notation so that,
for instance
$$
f^{(y,y)}(\theta)y\cdot\del_y \in\VV^{(y,y)}=\bigoplus_{i\le j=1,\ldots,d_1}
\VV^{(y_i,y_j)}
$$
with $f^{(y,y)}(\theta)$ a $d_1\times d_1$ matrix.

Note that the polynomial vector fields of (maximal) degree $k$ are
$$
\calP_k:=
\bigoplus_{\tv\in{\mathtt V}}\,\bigoplus_{j=0}^k\bigoplus_{(\tv_1,\ldots,\tv_j)\in{\mathtt U}}\mathcal V^{(\tv,\tv_1,\dots,\tv_j)}\,,
$$
so that, given a polynomial $F\in\calP_k$ we may define its ``projection'' onto a monomial subspace
$\Pi_{\mathcal V^{(\tv,\tv_1,\dots,\tv_j)}}$ in the natural way.

Since we are not working on spaces of polynomials, but on vector fields with finite regularity,
we need some more notations.
Given a $C^{k+1}$ vector field $F\in\VV_{\vec v,p}$, we introduce the notation
\begin{equation}\label{aiuto}
F^{(\tv,0)}(\theta):=F^{(\tv)}(\theta,0,0), \quad F^{(\tv,\tv')}(\theta)[\cdot]:=d_{\tv'}F^{(\tv)}(\theta,0,0)[\cdot], \quad  \tv\in\mathtt V\,,\quad \tv'= y_1,\ldots,y_{d_1},w
\end{equation}
$$
F^{(\tv,\tv_1,\dots,\tv_k)}(\theta)[\cdot,\dots,\cdot]=\frac{1}{\al(\tv_1,\ldots,\tv_k)!}
\Big(\prod_{i=1}^kd_{\tv_i}\Big)
F^{(\tv)}(\theta,0,0)[\cdot,\dots,\cdot]\,,\quad  \tv\in\mathtt V\,,\quad \tv_i= y_1,\ldots,y_{d_1},w\,,
$$
where we assume that $\tv_1,\ldots \tv_k$ are ordered and the $(d+1)$-dimensional
vector $\al(\tv_1,\ldots,\tv_k)$
denotes the multiplicity of each component.

By Taylor approximation formula any vector field in $\VV_{\vec v,p}$ which is  $C^{k+1}$ in $y,w$
may be written in a unique way as
sum of its Taylor polynomial in $\calP_k$ plus a $C^{k+1}$ (in $y,w$)
vector field with a zero of order at least $k+1$ at $y=0$, $w=0$. 
We think of this as a direct sum of vector spaces and introduce 
the notation
\begin{equation}\label{pro}
\Pi_{\VV^{(\tv,\tv_1,\dots,\tv_k)} }F:= F^{(\tv,\tv_1,\dots,\tv_k)}(\theta)[\tv_1,\dots,\tv_k]\,,
\end{equation}
 we refer to such operators as {\em projections}.
%
%
 
 \begin{defi}\label{mamm}
We identify the vector fields in $\VV_{\vec v,p}$ which are $C^{k+1}$ in $y,w$, with the direct sum
 $$
\calW_{\vec v,p}^{(k)}= \calP_k\oplus \RR_k\,,
 $$
 where $\RR_k$ is the space of $C^{k+1}$ (in $y,w$) vector fields with a zero of order at least $k+1$
 at $y=0$, $w=0$. 
On $\calW_{\vec v,p}^{(k)}$ we induce the natural norm for direct sums, namely for 
$$
f=\sum_{\tv\in{\mathtt V}}\,\sum_{j=0}^k\sum_{(\tv_1,\ldots,\tv_j)\in{\mathtt U}}
f^{(\tv,\tv_1,\dots,\tv_j)}(\theta)[\tv_1,\dots,\tv_j]\del_\tv + f_{\RR_k}\,
\qquad f_{\RR_k}\in\RR_{k}\,,
$$
we set
\begin{equation}\label{supernorma}
\|f\|_{\vec v,p}^{(k)}:=\sum_{\tv\in{\mathtt V}}\,\sum_{j=0}^{k}\sum_{(\tv_1,\ldots,\tv_j)\in{\mathtt U}}
\|f^{(\tv,\tv_1,\dots,\tv_j)}(\cdot)[\tv_1,\dots,\tv_j]\|_{\vec v,p} + \|f_{\RR_k}\|_{\vec v,p}\,.
\end{equation}

\end{defi}
\begin{rmk}
Note that with this definition if $k=\infty$ we are considering analytic maps with the norm
$$
\sum_{\tv\in{\mathtt V}}\,\sum_{j=0}^{\infty}\sum_{(\tv_1,\ldots,\tv_j)\in{\mathtt U}}
\|f^{(\tv,\tv_1,\dots,\tv_j)}(\cdot)[\tv_1,\dots,\tv_j]\|_{\vec v,p} .
$$
\end{rmk}

  We can and shall introduce in the natural way the polynomial subspaces and the norm \eqref{supernorma}
  also for maps 
  $\Phi= (\theta +f^{(\theta)},y+f^{(y)}, w+ f^{(w)})$ with
  $$
  \Phi: \TTT^d_s\times D_{a',p'}(r)\times \calO \to \TTT^d_{s_1}\times D_{a,p}(r_1)\,,
  $$
  since the Taylor formula holds also for functions of this kind.
  
We also denote
\begin{equation}\label{lemedie}
\begin{aligned}
&\av{\VV{(\tv, \tv_{{1}},\cdots, \tv_{{k}})}}:=\!\{f\in \VV^{(\tv,\tv_1,\ldots,\tv_k)}\!:  
f =  \langle f^{(\tv,\tv_1,\ldots,\tv_k)}\rangle\cdot
 \partial_\tv\},\\
&\VV^{(\tv, \tv_{{1}},\cdots, \tv_{{k}})}_0:=\{f\in \VV^{(\tv, \tv_{{1}},\cdots, \tv_{{k}})}\!:  
f =  (f^{(\tv, \tv_{{1}},\cdots, \tv_{{k}})}-\langle
 f^{(\tv, \tv_{{1}},\cdots, \tv_{{k}})}\rangle)\cdot \partial_\tv\}\,,
 \end{aligned}
\end{equation}
where $ \av{f}:=\fint_{\TTT^d}f(\theta)d\theta$.
  \medskip
  
\noindent  {\bf Tame vector fields}.\,
     We now define vector fields  behaving ``tamely'' 
  when composed with maps $\Phi$.  Let us fix a degree $\mathtt n\in \NNN$ and thus a norm 
   \begin{equation}\label{porc}
    \|f\|_{\vec v,p}=\|f\|_{\vec v,p}^{({\mathtt n})}\,.
    \end{equation}
  
   \begin{defi}\label{tame}
   Fix a large $q\geq \gotp_1$, $k\ge 0 $ and  a set $\calO$.
   Consider a $C^{k+\mathtt n+1}$ vector field
     $$
     F\in\calW_{\vec v,p}^{({\mathtt n})},\qquad   \vec v= (\g,\calO,s,a,r)\,.
     $$
    We say that $F$ is $C^k$-\emph{tame} (up to order $q$) if there exists a scale of  constants 
   $C_{\vec{v},p}(F)$,  with $C_{\vec{v},p}(F)\le C_{\vec{v},p_1}(F)$ for $p\le p_1$,
   such that  the following holds.
   
   For all  $\gotp_0\leq p\leq p_1\leq q$ consider any $C^{{\mathtt n}+1}$ map
   $\Phi= (\theta +f^{(\theta)},y+f^{(y)}, w+ f^{(w)})$
   with $\|f\|_{\vec{v}',\gotp_1}<1/2$ and 
   $$
   \Phi:\TTT^d_{s'}\times  D_{a_1,p_1}(r')\times \calO \to \TTT^d_s\times 
   D_{a,p +\nu }(r)\,,\quad \text{ for some } \quad r'\leq r\,, s'\leq s;
   $$   
   and denote
    $ \vec v'= (\g,\calO,s',a,r')$.
   Then for any $m=0,\ldots,k$ and any $m$ vector fields    
   $$h_1,\dots,h_m:\; \TTT^d_{s'}\times  D_{a_1,p_1}(r')\times \calO  \to V_{a,p +\nu}, $$
   one has 
  \begin{equation*}
  \begin{array}{crcl} 
    (\text{T}_m) & 
    \|d^{m}_{\mathtt U}F(\Phi)[h_1,\dots,h_m]\|_{\vec{v}',p} &\leq & 
     \big(C_{\vec{v},p}(F)+C_{\vec{v},\gotp_0}(F)\|\Phi\|_{\vec{v}',{p+\nu}} \big)
   \prod_{j=1}^m\|h_j\|_{\vec{v}',\gotp_0+\nu} \\ [.5em]
   & & &
    +C_{\vec{v},\gotp_0}(F)\sum_{j=1}^m\|{h_j}\|_{\vec{v}',{p+\nu}}\prod_{i\neq j}\|h_i\|_{\vec{v}',{\gotp_0+\nu}}
   \end{array}
   \end{equation*}
  for all $(y,w)\in D_{a_1,p_1}(r')$ and $p\leq q$.
   Here $d_{\mathtt U}F$ is the differential of $F$ w.r.t. the variables ${\mathtt U}:=\{y_1,\ldots,y_{d_1},w\}$
   and the norm is the one defined in \eqref{porc}.

   \noindent
   We say that a \emph{bounded} vector field $F$ is tame  if the conditions  (T$_m$)
   above hold with $\nu=0$.
   We  call $C_{\vec{v},p}(F)$ the $p$-\emph{tameness constants of }$F$.
   \end{defi}
  \begin{rmk}
  Note that in the definition above appear two regularity indices: $k$ being the maximum regularity in $y,w$ and $q$ the one in $\theta$.
  \end{rmk}

  \begin{rmk}\label{tameness}
   
   Definition \ref{tame} is quite natural if one has to deal with functions and vector fields which are
   merely differentiable.
   In order to clarify what we have in mind we consider the following example. Let $L$ be a linear operator
   $$
   L:H^{p}(\TTT^d)\to H^p(\TTT^d)\,.
   $$
   In principle there is no reason for $L$ to satisfy a bound like
   \begin{equation}\label{lop}
   \| L u\|_{p}\lessdot \|L\|_{\calL,p} \|u\|_{\gotp_0}+ \|L\|_{\calL,\gotp_0}\|u\|_p
   \end{equation}
   where $\|\cdot \|_{\calL,p}$ is the $H^p$-operator norm. However if $L=M_a$ is a multiplication operator, i.e.
   $M_au=au$ for some $a\in H^{p}(\TTT^d)$ then it is well known that
   $$
   \| M_a u\|_{p}\leq	\kappa_p(\|a\|_{p} \|u\|_{\gotp_0}+ \|a\|_{\gotp_0}\|u\|_p)
   $$
   which is \eqref{lop} since $\|a\|_p=\|M_a\|_{\calL,p}$. In this case we may set for all $p\leq q$
   $ C_{p}(M_{a})=\kappa_q\|a\|_p\,$, where $q$ is the highest possible regularity.
   This is of course a  trivial  (though very common in the applications) example in which the tameness constants
   and the operator norm coincide; we preferred to introduce Definition \ref{tame}
   since it is the most general class in which we are able to prove our result.
   \end{rmk}
   
   \begin{rmk}\label{tameness100}
   It is trivial to note that, given a sequence $C_{\vec{v},p}(F)$ of tameness constants for the field $F$, then any increasing sequence 
   $\tilde{C}_{\vec{v},p}(F)$ such that  $C_{\vec{v},p}(F)\leq \tilde{C}_{\vec{v},p}(F)$ for any $p$ is a possible choice of tameness constants for $F$.
   This leads to the natural question of finding a sharp sequence which then could be used as norm.
   Throughout the paper we shall write $C_{\vec{v},p}(F)\leq C$ to mean that the tameness constants 
   of $F$ can (and shall) be chosen in order to satisfy the bound.
   
   \end{rmk}

\subsection{Normal form decomposition}\label{normale}
 \begin{defi}[{\bf $(\NN,\calX,\RR)$-decomposition}]\label{nomec}
  Let  $\NN,\calX\subseteq \calP^{(\mathtt n)}$  have  the
  following properties:
    \begin{itemize}
    \item[(i)] if $\NN\cap\mathcal V^{(\tv,\tv_1,\dots,\tv_j)}\neq \emptyset$ then either
     $\NN\cap\mathcal V^{(\tv,\tv_1,\dots,\tv_j)}=\mathcal V^{(\tv,\tv_1,\dots,\tv_j)}$ or $\NN\cap\mathcal V^{(\tv,\tv_1,\dots,\tv_j)}=\av{\mathcal V^{(\tv,\tv_1,\dots,\tv_j)}}$
     for all $j\le {\mathtt n}$;

      \item[(ii)] one has $\VV^{(\mathtt v,0)}\subset \calX$ for $\mathtt v=y,w$.
    \end{itemize}
   
    We then decompose
  $$
  \calW_{\vec v,p}=C^{2{\mathtt n}+3}\cap\calW^{({\mathtt n})}_{\vec v,p}:=\NN\oplus \calX\oplus\RR
  $$
  where $C^{2{\mathtt n}+3}$ is the set of vector fields with $(2{\mathtt n}+3)$-regularity in $y,w$,
   $\RR$ contains all of $\RR_{\mathtt n}$ and all the polynomials generated by monomials not in
  $\NN\oplus \calX$. We shall denote $\Pi_\RR:= \uno-\Pi_\NN-\Pi_\calX$ and
  more generally for $\SSSS= \NN,\calX,\RR$ we shall denote  $ \Pi^\perp_{\SSSS}:= \uno-\Pi_\SSSS$.
   \end{defi}

 The following Definition is rather involved since we are trying to make our result as general as possible.
 However in the applications we have in mind, it turns out that one can choose $\calA_{s,a,p}$ satisfying the
 properties of the Definition below in an  explicit and natural way; see Section \ref{appli}.

\begin{defi}
[{\bf Regular vector fields}]\label{linvec-abs}
Given a subset $\calA_{s,a,p}\subset \calP^{(\mathtt n)}$ of polynomial  vector fields  $f: \TTT^d_s\times D_{a,p+\nu}(r )\to V_{a,p}$  we say that $\calA$ is a space of regular vector fields if the following holds.

Given a compact set $\calO\in \calO_0$ we  denote by $\calA_{\vec v,p}$ with $\vec v= (\g,\calO,s,a,r)$ the set of Lipschitz families $\calO\to \calA_{s,a,p}$.
We require that $\calA_{s,a,p}$ is a scale of Hilbert spaces w.r.t a  norm  $|\cdot|_{s,a,p}=|\cdot|_{s,a,p,\nu} $ and 
   we denote by $|\cdot|_{\vec v,p}$ the corresponding $\g$-weighted Lipschitz norm.

\begin{enumerate}
\item  $\VV^{(\mathtt v,0)}\in \calA_{\vec v,p}$ for  $\mathtt v= y,w$ 
while either $\calA_{\vec v,p}^{(\theta)}=\emptyset$ or
$\calA_{\vec v,p}^{(\theta)}=\VV^{(\theta,0)}_0$.

\item All $f\in \calA_{\vec v,q}$ are  $C^k$ tame up to order $q$ for all $k$, with tameness constants
 \begin{equation}\label{tameconst2}
  C_{\vec{v},p}(f)= \mathtt C \,|f|_{\vec{v},p}.
\end{equation}
for some $C$ depending on $\gotp_0$ on the dimensions $d,d_1$ and on the maximal regularity $q$.
Moreover $|\cdot|_{\gotp_1}$ is a {\em sharp} tameness constant, namely there exists a $\mathtt{c}$  depending  on $\gotp_0,\gotp_1,d,d_1$ such that
 \begin{equation}\label{tameconst3}
   |f|_{\vec{v},\gotp_1} \le \mathtt{c} \, C_{\vec{v},\gotp_1}(f)
\end{equation}
for any $f$ and any tameness constant $ C_{\vec{v},\gotp_1}(f)$.
\item For $K>1$ there exists  smoothing projection operators $\Pi_K: \calA_{\vec v,p}\to \calA_{\vec v,p} $ such that $\Pi_K^2=\Pi_K$, 
%
for $p_1\ge 0$, one has 
\begin{align}
& \qquad| \Pi_{K}F |_{\vec v,p+p_1} \le  \mathtt C K^{{p_1 }}| F|_{\vec v,p} \label{P1}\\
& \qquad| F-\Pi_{K}F |_{\vec v,p} \le  \mathtt C K^{{-p_1 }} | F|_{\vec v,p+p_1}\label{P2}
\end{align}
finally if $C_{\vec v,p}(F)$ is any tameness constant for $F$ then we may choose a tameness constant such that
\begin{equation}
\qquad C_{\vec v,p+p_1}(\Pi_{K}F )\le \mathtt C K^{{p_1 }} C_{\vec v,p}(F)\qquad\label{P3}
\end{equation}

%
We denote  by $E^{(K)}$ the subspace where $\Pi_K E^{(K)}=E^{(K)}$.

\item Let $\BB$ be the set of bounded vector fields in $\calA_{s,a,p}\ni f: \TTT^d_s\times D_{a,p}(r )\to V_{a,p}$ with the corresponding norm $|\cdot |_{s,a,p,0}$.
For all $f\in\BB$ such that
 \begin{equation}\label{pippo2}
 |f|_{\vec{v},\gotp_1}\leq {\mathtt c}{\rho},
 \end{equation}
with  $\rho>0 $  small enough,
 the following holds:
 
  \noindent
 (i) The map $\Phi:=\uno+f$ is such that
 \begin{equation}\label{speriamobene2}
 \Phi: \TTT^{d}_{s}\times D_{a,p}(r)\times \calO\longrightarrow \TTT^{d}_{s+\rho s_0}\times D_{a,p}(r+\rho r_0).
 \end{equation}

 \noindent
 (ii) There exists a vector field $h\in \BB$ such that
 \begin{itemize}
 \item $|h|_{\vec{v}_1,p}\leq 2|f|_{\vec{v},p}$, 
 the map $\Psi:=\uno+h$ is such that
 \begin{equation}\label{nome002}
 \Psi : \TTT^{d}_{s-\rho s_0}\times D_{a,p}(r-\rho r_0)\times \calO\to \TTT^{d}_{s}\times  D_{a,p}(r ).
 \end{equation}
 
 \item  for all $(\theta,y,w)\in \TTT^d_{s-2\rho s_0}\times D_{a,\gotp_1}(r-2\rho r_0)$ one has
 \begin{equation}\label{nome2}
 \Psi\circ\Phi(\theta,y,w)=(\theta,y,w).
 \end{equation}
 \end{itemize}   
  \item Given any regular bounded vector field $g\in \mathcal B$, $p\geq \gotp_1$  with $| g|_{\vec{v},\gotp_1}\leq {\mathtt c} \rho$
   then  for $0\leq t\leq1$ there exists $f_{t}\in \mathcal B$ such that the time$-t$ map of the flow of $g$ is of the form 
    $\uno+ f_{t}$ moreover we have $|f_{t}|_{\vec{v},p} \leq  2| g|_{\vec{v}_1,p}$ where $\vec v_1= (\lambda,\calO,s-\rho s_0,a,r)$.
\end{enumerate}
\end{defi}

\begin{defi}\label{nx}
Consider $\mathcal E$ a subspace\footnote{For instance
 $\mathcal E$ may be the subspace of Hamiltonian vector fields.}
 of $\VV_{\vec{v},p}$.
  We say that $\calE$ is compatible with the $(\NN,\calX,\RR)$-decomposition if
 
 \begin{itemize}

 \item[(i)]
  any $F\in \calE\cap \calX$ is a regular vector field;

  \item[(ii)]
  for any $F\in \calE\cap\calP_{\mathtt n}$ one has
$\Pi_{\calU}F\in \calE$ for $\calU=\NN,\calX,E^{(K)}$;

 \item[(iii)]  
denoting 
\begin{equation}\label{subspazio}
{\BB}_\calE:=\Big\{f\in\calX\cap \BB\, :\;   \Phi_{f}^t \mbox{ is }\;\mathcal E\; {\mbox{ preserving for all }} t\in[0,1]
\Big\}\subset \calX\cap \BB\,,
\end{equation}
one has
  \begin{equation}\label{proloco}
\forall g\in \BB_\calE,F\in \calE:\quad  [g,F]\in \calE\,,\quad \forall g,h\in \BB_\calE:\quad  \Pi_\calX[g,h]\in \BB_\calE\,.
\end{equation}

\end{itemize}
\end{defi}

\begin{defi}[{\bf Normal form}]\label{norm}
 We say that $N_{0}\in \NN\cap \calE$ is a {\em diagonal} vector field if for all $K>1$
\begin{equation}\label{norm1}
 {\rm ad}(N_0)\Pi_{E^{(K)} } \Pi_\calX =\Pi_{E^{(K)} } \Pi_\calX  {\rm ad} (N_0)\,,\quad\mbox{on $\BB_\calE$.}
\end{equation}

 \end{defi}

 \subsection{Main result}\label{main}

Let us fix once and for all a $(\NN,\calX,\RR)$-decomposition.
Before stating the result we need to 
introduce parameters (which shall depend on the application) fulfilling the following constraints.
\begin{const}[{\bf The exponents}]\label{sceltapar}
We fix parameters $\e_0,\tR_0,\tG_0,\mu,\nu,\eta,\chi,\al,\ka_1,\ka_2,\ka_3,\gotp_2$ such that the following holds.
\begin{itemize}
\item $0<\e_0\le\tR_0\le \tG_0$ with $\e_0\tG_0^3,\e_0 \tG_0^2 \tR_0^{-1}<1$.

\item $\mu,\nu,\ka_3\ge 0$, $\gotp_2>\gotp_1$, $0\le \al <1$, $1<\chi<2$  such that $\al\chi<1$.

\item Setting $\ka_0:= \mu+\nu+4$  and $\Delta\gotp:= \gotp_2-\gotp_1$  one has
\begin{subequations}\label{exp}
\begin{align}
 \ka_1 &> \max(\frac{\ka_0+\ka_3}{\chi}, \frac{\ka_0}{\chi-1})\,,
 \label{exp1} \\
 \ka_2&> \max\big(\frac{2\ka_0}{2-\chi}, \frac{1}{1-\al\chi}((1+\al)\ka_0 +2  \max(\ka_1,\ka_3)-\chi \ka_1)\big) \,,
 \label{exp2}\\
 \eta &> \mu+(\chi-1) \ka_2+1\,,
 \label{exp3} \\
 \Delta\gotp&>\max\big(\ka_0+\chi \ka_2+\max(\ka_1,\ka_3),\frac{1}{1-\al}(\ka_0+(\chi-1) \ka_2+\max(\ka_1,\ka_3))\big)\,,
 \label{exp4}\\
 \al \Delta\gotp &\le  \ka_2+\chi \ka_1-\ka_0-\max(\ka_1,\ka_3)\,.
 \label{exp5}
\end{align} 
\label{eq106}
\end{subequations}
\item  there exists $K_0>1$   such that
\begin{equation}\label{expexpexp}
 \log K_0\geq \frac{1}{\log\chi}C,
\end{equation}
with $C$ a given function of $\mu,\nu,\eta,\al,\ka_1,\ka_2,\ka_3,\gotp_2$  
and moreover
\begin{subequations}\label{expexp}
\begin{align}
&\tG_0^2\tR_0^{-1}\e_0 K_0^{\ka_0}\max (1, 	\tR_0 \tG_0 K_0^{\ka_0+(\chi-1)\ka_2})<1\,,
\label{1s1}\\
&\max( K_0^{\ka_1},\e_0 K_0^{\ka_3})  K_0^{\ka_0-\Delta\gotp +(\chi-1) \ka_2 }\tG_0\e_0^{-1} \max\Big ( 1,\tR_0 ,
 \e_0\tG_0 K_0^{\al\Delta\gotp} \Big)\le 1\,,
 \label{4s2}\\
& \max( K_0^{\ka_1},\e_0 K_0^{\ka_3})K_0^{\ka_0-\chi \ka_1}\tG_0 \tR_0^{-1}\max
 \Big( \tR_0 , \e_0 \tG_0  K_0^{\al\Delta\gotp} \Big)\le 1\,.
 \label{6s2}
 \end{align}
 \label{small}
 \end{subequations}
\end{itemize}
\end{const}

\begin{rmk}\label{concreto?} 
 In the applications  the
 constants $\tG_0,\tR_0,\e_0$ in Constraint \ref{sceltapar} are given by the problem under
 study and typically they depend parametrically on ${\rm diam}(\calO_0)\sim \gamma$;
 then one wishes to show that for $\g$ small enough it is possible to choose all other
 parameters in order to
 fulfill Constraint \ref{sceltapar}. Often this implies requiring that $K_0\to \infty$ as $\g\to 0$.
 In order to highlight this dependence one often uses $\e_0$ as parameter and
introduces $\tg,\tr$ such that
\begin{equation}\label{porcammerda}
 \tG_0 \sim \e_0^\tg \,,\quad \tR_0\sim \e_0^\tr \,,\quad \text{ with } \quad \tg\le\tr\le 1\,,\quad 
 \min\{1+3\tg,1+2\tg-\tr\}> 0\,. 
\end{equation}
Then given $\ka_0,\ka_3$ one looks for $\al,\chi,\ka_1,\ka_2,\gotp_2$ satisfying \eqref{eq106} and,
setting
$K_0= \e_0^{-\ta}$, the constraints \eqref{small} become constraints on $\ta$.
Another typical procedure is to write $\tG_0,\tR_0,\e_0$ as powers of $K_0$, see paragraph \ref{riscalo}.
Note that in \eqref{porcammerda} we need $1+3\tg>0$ but in principle we allow $\tg<0$; same for $\tr$. This means in particular that $\tG_0$ and 
$\tR_0$ might be very large.
\end{rmk}

\begin{defi}[{\bf Homological equation}]\label{pippopuffo2} 
Let $\gamma> 0$, $K\ge K_0$,
consider a compact set $\calO \subset \calO_0$ and set $\vec{v}=(\g,\calO,s,a,r)$ and $\vec{v}^{\text{\tiny 0}}=(\g,\calO_0,s,a,r)$.
Consider a  vector field $F\in \calW_{\vec v^{\text{\tiny 0}},p}$  i.e.
$$
F= N_0+G: \calO_0\times    D_{a,p+\nu}(r)\times\TTT^{d}_{s}\to V_{a,p}\,, 
$$
 which is  $C^{\mathtt n+2}$-tame up to order $q=\gotp_2+2$.
   We say $\mathcal O$ satisfies the  \emph{  homological equation}, for
   $(F,K,\vec{v}^{\text{\tiny 0}},\rho)$ if  
    the following holds.

 \noindent
1.\, For all $\xi\in \mathcal O$   one has  $F(\xi)\in \calE$ and $|\Pi_{\calX}G|_{\vec{v},\gotp_{2}-1}\leq \mathtt{C}\, C_{\vec{v},\gotp_{2}}(\Pi_{\NN}^{\perp}G)$.
 
 \noindent
2.\, there exist   a bounded regular vector field $g\in \calW_{\vec{v}^{\text{\tiny 0}},p}\cap E^{(K)}$ such that 

 \begin{itemize}
 \item[(a)]  $g\in \BB_\calE$ for all $\xi\in \calO$,
  \item[(b)] one has $|g|_{\vec{v}^{\text{\tiny 0}},\gotp_1}\le \mathtt C |g|_{\vec v,\gotp_1}\le {\mathtt c}\rho$
  and for $\gotp_1\le p\le \gotp_2$
 \begin{equation}\label{buoni22}
|g|_{\vec v,p}\leq
\gamma^{-1}K^{\mu}(|\Pi_K\Pi_{\calX}G|_{\vec v,p}
+K^{\al(p-\gotp_1)}|\Pi_K\Pi_{\calX}G|_{\vec v,\gotp_1}
\g^{-1}C_{{\vec v,p}}(G))\,,
 \end{equation}
 $$
 |\Pi_\calX[\Pi_{\calX}^{\perp} G, g]|_{\vec v,p-1} \le C_{\vec v,p+1}(G)|g|_{\vec v,\gotp_1}+  C_{\vec v,\gotp_1}(G)|g|_{\vec v,p+\nu+1}
 $$
 \item[(c)] setting 
 $
 u:=\Pi_{K}\Pi_{\calX}({\rm ad}(\Pi_{\calX}^{\perp} F)[g]-F),
 $
one has 
  \phantom{assssasfffff}
  \begin{equation}\label{cribbio42}
  \begin{aligned}
    |u|_{\vec{v},\gotp_1}&\leq \e_0\g^{-1} K^{-\eta+\mu} C_{\vec{v},\gotp_1}(G)|\Pi_{K}\Pi_{\calX}G|_{\vec{v},\gotp_1},\\
    |u|_{\vec{v},\gotp_2}&\leq \g^{-1} K^{\mu}\left(
    |\Pi_{K}\Pi_{\calX}G|_{\vec{v},\gotp_2}C_{\vec{v},\gotp_1}(G)+K^{\al(\gotp_2-\gotp_1)}|\Pi_{K}\Pi_{\calX}G|_{\vec{v},\gotp_1}C_{\vec{v},\gotp_2}(G)
    \right);
\end{aligned}
  \end{equation}
  \item[(d)] setting $\vec{v}'=(\g,\calO,s-\rho s_0,a,r-\rho r_0)$, and let $\Phi$ the change of variables 
  generated by $g$, 
  one has that
  \begin{equation}\label{cribbio420}
  \begin{aligned}
  |\Pi_{\calX}\Phi_{*}F|_{\vec{v}',\gotp_2-1}\leq C_{\vec{v},\gotp_2}(\Pi_{\NN}^{\perp}G)+
  \mathtt{C}\left( |g|_{\vec{v},\gotp_2}C_{\vec{v},\gotp_1}(G)+
|g|_{\vec{v},\gotp_1}C_{\vec{v},\gotp_2}(G)  
  \right)
  \end{aligned}
  \end{equation}
  \end{itemize}
\end{defi}

\begin{defi}[{\bf Compatible changes of variables}]\label{compa} Let the parameters in Constraint
\ref{sceltapar} be fixed.
Fix also $\vec v= (\g,\calO,s,a,r)$, $\vec{v}^{\text{\tiny 0}}= (\g,\calO_0,s,a,r)$ with $\calO\subseteq\calO_0$
a compact set,  parameters
$K\ge K_0,\rho<1$. Consider a
 vector field $F= N_0+G\in \calW_{\vec{v}^{\text{\tiny 0}},p}$  which is  $C^{\mathtt n+2}$-tame up to
 order $q=\gotp_2+2$ and such that, $$F\in \calE\quad \forall\x\in \calO\,,\qquad |\Pi_{\calX} G|_{\vec{v},\gotp_2-1}\leq \mathtt{C} C_{\vec{v},\gotp_2}(\Pi_{\NN}^\perp G) .$$
  We say that a 
left invertible $\calE$-preserving change of variables 
$$
\calL, \calL^{-1}: \TTT^d_{s}\times D_{a,\gotp_1}(r)\times \calO_0 \to
 \TTT^d_{s+\rho s_0}\times D_{a-\rho a_0,\gotp_1}(r+\rho r_{0})
 $$
 is {\em compatible} with $(F,K,\vec v,\rho)$ if the following holds:
\begin{itemize}
\item[(i)]  $\calL$ is ``close to identity'', i.e.
denoting $\vec{v}^{\text{\tiny 0}}_1:=(\g,\calO_0,s-\rho s_0,a-\rho a_0,r-\rho r_0)$ one has 
\begin{equation}\label{satana}
\begin{aligned}
&\|(\calL-\uno)h\|_{\vec{v}^{\text{\tiny 0}}_1,\gotp_1}\leq \mathtt C  \e_0 K^{-1}
\|h\|_{\vec{v}^{\text{\tiny 0}},\gotp_1}\,.\\
\end{aligned}
\end{equation}

\item[(ii)]  $\calL_*$ conjugates the $C^{\mathtt n+2}$-tame vector field $F$ to 
the vector field $\hat{F}:= 
\calL_{*}F=  N_0+ \hat G$
which is $C^{\mathtt n+2}$-{tame}; moreover denoting
 $\vec v_2:=(\g,\calO,s-2\rho s_0,a-2\rho a_0,r-2\rho r_0)$
one may choose the tameness constants of $\hat G$ so that
\begin{equation}\label{odio}
\begin{aligned}
&C_{\vec v_2, \gotp_1}(\hat{G})\le C_{\vec v,\gotp_1}(G)(1+\e_0 K^{-1})\,,\\
&C_{\vec v_2,\gotp_{2}}(\hat{G})\le \mathtt C \big(C_{\vec v,\gotp_{2}}(G)+
\e_0 K^{\ka_3}C_{\vec v,\gotp_{1}}(G)\big) \,\\
&
|\Pi_{\calX}\hat{G}|_{\vec{v}_2,\gotp_2-1}
\le \mathtt C \big(C_{\vec v,\gotp_{2}}(\Pi_{\NN}^{\perp}G)+
\e_0 K^{\ka_3}C_{\vec v,\gotp_{1}}(\Pi_{\NN}^{\perp}G)\big) \,.
\end{aligned}
\end{equation}
\item[(iii)] $\calL_*$ ``preserves the $(\NN,\calX,\RR)$-decomposition'', namely one has
\begin{equation}\label{satana2}
\Pi_\NN^\perp (\calL_* \Pi_\NN F) = 0\,, \quad \qquad \Pi_{\calX}(\calL_{*}\Pi_{\calX}^\perp F)=0\,.
\end{equation}
\end{itemize}
\end{defi}

%

Given $\g_0>0$ we set for $n\ge 0$
\begin{equation}\label{numeretti}
\begin{aligned}
& 
\quad
\mathtt G_n=\mathtt G_0(1+\sum_{j=1}^n 2^{-j}),\quad \mathtt R_n=\mathtt R_0(1+\sum_{j=1}^n 2^{-j}),\quad K_n= (K_0)^{\chi^n},\, \g_n= \g_{n-1}(1-\frac{1}{2^{n+2}}), \\ & a_n= a_0(1-\frac{1}{2}\sum_{j=1}^n2^{-j})\,,\quad
 r_n= r_0(1-\frac{1}{2}\sum_{j=1}^n2^{-j}), 
\quad s_n= s_0(1-\frac{1}{2}\sum_{j=1}^n2^{-j})\,, \\ &
 \Pi_n := \Pi^{(K_n)}\,,\quad
\Pi_n^{\perp}:=\uno-\Pi_n\,, E_n= E^{(K_n)}, \quad \rho_{n}:=\frac{1}{2^{n+5}} 
\end{aligned}
\end{equation}
Finally, for all $n\ge0$ we denote
$\vec v_n=(\g_{n},\calO_{n},s_{n},a_{n},r_{n})$, $\vec{v}^{\text{\tiny 0}}_n= (\g_{n},\calO_{0},s_{n},a_{n},r_{n})$.

\medskip

We reformulate our main result, stated  in the Introduction, in a more precise way. This is useful  for applications, where  one needs to have information of the sequence of vector fields $F_n$ and on the changes of variables $\calH_n$ in order to prove that the set $\calO_\infty$ is not empty.

 \begin{theo}[{\bf Abstract KAM}]\label{thm:kambis}
Fix  a decomposition and a subspace $\calE$ as in Definitions \ref{nomec} and \ref{nx}. Fix parameters $\e_0,\tR_0,\tG_0,\mu,\nu,\eta,\chi,\al,\ka_1,\ka_2,\ka_3,\gotp_2$  satisfying Constraint \ref{sceltapar}. Let $N_0$ be a diagonal vector field as in Definition \ref{norm} and consider a  vector field
 \begin{equation}\label{kam1}
 F_0:=N_{0}+G_0 \in \calE\cap\calW_{\vec{v}_{0},p}
 \end{equation} 
 which is  $C^{\mathtt n+2}$-tame up to order $q=\gotp_2+2$.

Fix $\g_0>0$ and assume that
 \begin{equation}\label{sizes}
\g_{0}^{-1}C_{\vec{v}_0,\gotp_2}(G_0) \le \tG_0\,,\quad
 \g_0^{-1}C_{\vec{v}_0,\gotp_2}(\Pi_{\NN}^\perp G_0)\le \tR_0\,, \quad
 \g_0^{-1}|\Pi_{\calX}G_0|_{\vec{v}_0,\gotp_1}\le \e_0\,,
 \quad  \g_0^{-1}|\Pi_{\calX}G_0|_{\vec{v}_0,\gotp_2}\le \tR_0\,.
 \end{equation}

For all $n\geq 0$  we define recursively  changes of variables $\calL_n,\Phi_n$  and  compact sets
$ \calO_n$ as follows.

\smallskip

  Set $\HH_{-1}=\HH_0=\Phi_0=\calL_0=\uno$, and for $0\le j \le n-1$ set recursively
  $\HH_j= \Phi_j\circ \calL_j\circ \HH_{j-1}$ and  $F_j:=(\HH_j)_{*}F_0:=N_0+G_{j}$.
  Let $\calL_{n}$ be any change of variables   compatible with $(F_{n-1},K_{n-1}, \vec v_{n-1},\rho_{n-1})$,
    and  $\calO_n$ be any compact set 
   \begin{equation}\label{oscurosignore}
   \calO_{n}\subseteq \calO_{n-1}\,,
   \end{equation}  
  which satisfies the   homological equation for $((\calL_n)_*F_{n-1},K_{n-1},\vec v^{\text{\tiny 0}}_{n-1},\rho_{n-1})$.
   For $n>0$ let $g_n$ be the regular vector field defined in item (2) of Definition \ref{pippopuffo2} and set
    $\Phi_n$ the time-1 flow map generated by $g_n$. 
    
    Then $\Phi_n$ is left invertible and 
     $F_n:= (\Phi_n\circ\calL_n)_*F_{n-1}\in \calW_{\vec v^{\text{\tiny 0}}_n,p}$ is  $C^{\mathtt n+2}$-tame up to order $q=\gotp_2+2$.
Moreover the following holds.

      \begin{itemize}
     
  \item[{\bf (i)}]
  Setting $G_n=F_n-N_0$ then
 \begin{equation}\label{lamorte}
 \begin{aligned}
  \Gamma_{n,\gotp_1}&:=\g_{n}^{-1} C_{\vec{v}_n,\gotp_1}(G_n)\leq \tG_n, \quad 
  \Gamma_{n,\gotp_2} :=\g_{n}^{-1}
  C_{\vec{v}_n,\gotp_2}(G_n)\leq  \tG_0 K_n^{\ka_1},\\
\Theta_{n,\gotp_1}&:= \g_{n}^{-1}C_{\vec{v}_n,\gotp_1}(\Pi_{\NN}^\perp G_n)\leq \tR_n,
 \quad  \Theta_{n,\gotp_2}:=
 \g_{n}^{-1}C_{\vec{v}_n,\gotp_2}(\Pi_{\NN}^\perp G_n)\leq  \tR_0  K_{n}^{\ka_1}
 \\
\de_n&:= \g_{n}^{-1} |\Pi_{\calX}G_n|_{\vec{v}_n,\gotp_1}\leq K_0^{\ka_2} \e_0 K_{n}^{-\ka_2},
\quad \g_{n}^{-1} |\Pi_{\calX}G_n|_{\vec{v}_n,\gotp_2}\leq \tR_0 K_n^{\ka_1} 
\\
|g_{n}|_{\vec u_n,\gotp_1}&
     \leq  K_0^{\ka_2} \e_0 \tG_0K_{n-1}^{-\ka_2+\mu+1},
     \quad 
     |g_{n}|_{\vec{u}_{n},\gotp_2}
     \leq 
     \tR_0 \tG_0^{-1} K_{n-1}^{-\nu-1+\chi \ka_1} 
 \end{aligned}
 \end{equation}
where 
$\vec u_{n}=(\g_{n},\calO_{n},s_{n}+12\rho_n s_0 ,a_{n}+12\rho_{n} a_0,r_{n}+12\rho_n r_0)$.
 \item[{\bf (ii)}]
 The sequence $\calH_n$ converges
     for all $\x\in\calO_0$ 
  to some change of variables
  \begin{equation}\label{dominio}
  {\calH}_\io={\calH}_\io(\x):  D_{{a_{0}},p}({s_{0}}/{2},{r_{0}}/{2})\longrightarrow D_{\frac{a_{0}}{2},p}({s_{0}},{r_{0}}).
  \end{equation}
  
 \item[{\bf (iii)}]
 Defining $F_{\infty}:=(\calH_\io)_{*}F_0$ one has 
 \begin{equation}\label{fine}
 \Pi_\calX F_{\infty}=0
 \quad \forall \xi \in \calO_\io:=\bigcap_{n\geq0}\calO_n
 \end{equation}
 and 
 $$
 \g_0^{-1}C_{\vec{v}_\io,\gotp_1}(\Pi_{\NN}F_{\infty}-N_0)\le 2\tG_0, \quad
  \g_0^{-1}C_{\vec{v}_\io,\gotp_1}(\Pi_{\RR}F_{\infty})\le 2\tR_0
 $$
 with $\vec{v}_\io:=(\g_0/2,\calO_\io,s_{0}/2,a_{0}/2)$. 
 
 \end{itemize}
\end{theo}
\begin{proof}
The proof of this result is deferred to Section \ref{itscheme}.
\end{proof}

\begin{rmk}
Note that if one makes the further assumption that $s_0>0$, the smallness conditions as well as the definition of the set of  parameters
in Definition \ref{pippopuffo2} simplify drastically: in particular one may choose $\gotp_2=\gotp_1$.
We are not makeing this assumption because our aim was to have a unified proof; however we discuss the time analytic case in
Appendix \ref{solaminchia} for completeness.
\end{rmk}

\zerarcounters
\section{Triangular decomposition and Mel'nikov conditions}\label{brooks}

In most applications one may redefine the sets on which one can solve the
  homological equation in a more direct way, by introducing the so-called {\em Mel'nikov conditions}.

We start by introducing some notation.
\begin{defi}[{\bf Triangular decomposition}]\label{triang}
We say that a decomposition $(\NN,\calX,\RR)$ is triangular if
 $\calX$ admits a block decomposition
\begin{equation}\label{blodia}
 \calX=\bigoplus_{j=1}^\mathtt b \calX_j
\end{equation}
such that  for all $N\in\NN,R\in\RR$ setting
\begin{equation}\label{sodia}
\gotN:=\Pi_{\calX}{\rm ad}(N)\,,\quad  \gotR:=\Pi_{\calX}{\rm ad}(R)
\end{equation}
then $\gotN$ is block diagonal and $\gotR$ is strictly upper triangular, i.e. 
$$
\gotN:\; \calX_i\to \calX_i\,,\quad \gotR:\; \calX_i\to \bigoplus_{j>i}\calX_j\,.
$$
\end{defi}
\begin{rmk}\label{sonfica}
In order to construct a triangular decomposition one generally associates some degree to the variables\footnote{clearly one must give $\theta$ degree zero  since we do not Taylor expand on it, then by convention we decide to give degree one to $w$.}:
$${\rm deg}(\theta)=0\,,\quad {\rm deg}(w)=1\,,\quad {\rm deg}(y)= \td, $$
this automatically fixes the degree of a  monomial vector field as 
$$
{\rm deg}(y^j e^{\ii \theta\cdot \ell} w^\al \partial_{\tv})=  j \td  +|\al| - {\rm deg}(\tv),
$$
moreover one verifies that if $g$ has degree $d_1$ and $f$ degree $d_2$ then $[f,g]$ has degree $d_1+d_2$. Finally 
we remark that $\VV^{(\tv,0)}$ has negative degree for $\tv=y,w$ and degree equal to zero for $\tv=\theta$, in the same way $\VV^{(\tv,\tv)}$ has always degree zero for $\tv=y,w$ while $\VV^{\theta,\tv}$ has positive degree.
Then to a polynomial we may associate its minimal and maximal degree. In the same way if a $C^{k+1}$ function has zero projection on all spaces $\VV^{\tv,\tv_1,\cdots,\tv_h}$, with $h\le k$ and degree $\le d$ we call $d$  its minimal degree.
In many applications it is convenient to place all monomials of degree $\le 0$ in $\NN\cap\calX$ and all those of positive degree in $\RR$.

\end{rmk}
\begin{lemma}\label{pillo}
Any decomposition such that $\NN$ contains only polyomials of degree zero and $\RR$  contains only terms of minimal degree $>0$ is triangular w.r.t. the degree decomposition of $\calX= \oplus_j \calX_j$ where the $\calX_j$ are spaces of homogeneous polynomials with increasing degree $\td_j$.
\end{lemma}
\begin{proof}
Given any  polynomial $P$ of positive minimal degree the operator ${\rm ad}P$ has positive degree, namely its action on any polynomial increases its minimal degree.  Now for any $N\in \NN$ $\Pi_\calX {\rm ad} N$ preserves the degree and thus maps each $\calX_j$ into itself.   By definition the maximal degree in $\calX$ is given by $\td_b$. Then  if we have a tame function $f$ with minimal degree $>\td_b -\max(1,\td)$ the operator
$\Pi_\calX {\rm ad}f$ on $\calX$ is equal to zero, and this is just a property of polynomial subspaces in $\RR$. Since all $R\in \RR$ have positive degree, then  obviously $\Pi_\calX {\rm ad}R$ is upper triangular.
\end{proof}

Once we have a triangular decomposition we introduce the following notion
\begin{defi}[{\bf Mel'nikov conditions}]\label{pippopuffo3} 
Let $\gamma,\mu_1> 0$, $K\ge K_0$, consider a compact set $\calO \subset \calO_0$ and set
$\vec{v}=(\g,\calO,s,a,r)$ and $\vec{v}^{\text{\tiny 0}}=(\g,\calO_0,s,a,r)$.
Consider a  vector field $F\in \calW_{\vec v^{\text{\tiny 0}},p}$  i.e.
$$
F= N_0+G: \calO_0\times    D_{a,p+\nu}(r)\times\TTT^{d}_{s}\to V_{a,p}\,, 
$$
 which is  $C^{\mathtt n+2}$-tame up to order $q=\gotp_2+2$.
   We say $\mathcal O$  satisfies the Mel'nikov conditions  for
   $(F,K,\vec{v}^{\text{\tiny 0}})$ if  
    the following holds.

 \noindent
1.\, For all $\xi\in \mathcal O$   one has  $F(\xi)\in \calE$  and $|\Pi_{\calX}G|_{\vec{v},\gotp_{2}-1}\leq \mathtt{C}\, C_{\vec{v},\gotp_{2}}(\Pi_{\NN}^{\perp}G)$.
 
 \noindent
2.\, Setting $\gotN:= \Pi_K\Pi_\calX {\rm ad }(\Pi_\NN F)$ for all $\xi\in \calO$ there exists a block-diagonal  operator 
$\gotW: E^{(K)}\cap\calX\cap\calE  \to E^{(K)}\cap \BB_\calE $ such that 
 for any vector field $X\in E^{(K)}\cap\calX\cap\calE $

 \begin{itemize}

  \item[(a)] one has
 \begin{equation}\label{buoni}
 \begin{aligned}
|\gotW X|_{\vec v,p}&\le
\gamma^{-1}K^{\mu_1}(|X|_{\vec v,p}
+K^{\al(p-\gotp_1)}|X|_{\vec v,\gotp_1}
\g^{-1}C_{{\vec v,p}}(G)).
\end{aligned}
 \end{equation}
 
 \item[(b)] setting
 $
u:=(\Pi_{K}{\rm ad}(\Pi_{\NN}F)[\gotW X]-X)
 $ 
one has 
  \phantom{assssasfffff}
  \begin{equation}\label{cribbio4}
  \begin{aligned}
  |u|_{\vec{v},\gotp_1}&\leq \e_0 \g^{-1} K^{-\h+\mu_1} C_{\vec{v},\gotp_1}(G)|X|_{\vec{v},\gotp_1},\\
    |u|_{\vec{v},\gotp_2}&\leq \g^{-1} K^{\mu_1}\left(
    |X|_{\vec{v},\gotp_2}C_{\vec{v},\gotp_1}(G)+K^{\al(\gotp_2-\gotp_1)}|X|_{\vec{v},\gotp_1}C_{\vec{v},\gotp_2}(G)
    \right).
\end{aligned}
  \end{equation}
  \end{itemize}
\end{defi}

Then we have the following result.

\begin{proposition}[{\bf Homological equation}]\label{uffffa}
Let $\gamma> 0$, $K\ge K_0$,
consider a compact set $\calO \subset \calO_0$ and set $\vec{v}=(\g,\calO,s,a,r)$ and $\vec{v}^{\text{\tiny 0}}=(\g,\calO_0,s,a,r)$.
Consider a  vector field $F\in \calW_{\vec v^{\text{\tiny 0}},p}$  i.e.
$$
F= N_0+G: \calO_0\times    D_{a,p+\nu}(r)\times\TTT^{d}_{s}\to V_{a,p}\,, 
$$
 which is  $C^{\mathtt n+2}$-tame up to order $q=\gotp_2+2$. 
Assume that $\g\sim {\rm diam} \calO_0$ and set
 \begin{equation}\label{sizes2b}
\Gamma_{p}:=\g^{-1}C_{\vec{v},p}(G), \quad 
 \Theta_{p}:=\g^{-1}C_{\vec{v},p}(\Pi_{\NN}^\perp G).
 \end{equation}
Assume finally that for any $f\in \BB_\calE$ 
\begin{equation}\label{vino2}
\begin{aligned}
&|\Pi_\calX  [\Pi_\RR G,f]|_{\vec v,p-1}\le C_{\vec v,p+1}(\Pi_{\NN}^{\perp}G)|f|_{\vec v,\gotp_1}+ C_{\vec v,\gotp_1}(\Pi_{\NN}^{\perp}G)|f|_{\vec v,p+\nu+1},\\
&|\Pi_\calX  [\Pi_\NN G,f]|_{\vec v,p-1}\le C_{\vec v,p+1}(G)|f|_{\vec v,\gotp_1}+ C_{\vec v,\gotp_1}(G)|f|_{\vec v,p+\nu+1},\\
\end{aligned}
\end{equation}
If $\calO$ satisfies the Mel'nikov conditions of Definition \ref{pippopuffo3} for $(F,K,\vec{v}^{\text{\tiny 0}})$  then 
$\calO$ satisfies items 1.  and 2.a-b-c  of Definition \ref{pippopuffo2} provided that we fix 
\begin{equation}\label{migo}
\mu= (b+1)(\mu_1+\nu+1 )+ \mathtt t
\end{equation}where $\mathtt t>0$ is such that $$ (1+ \Theta_{\gotp_1}(1+\Gamma_{\gotp_1}))^{b}(1+\Gamma_{\gotp_1})\le K_0^{\mathtt t}.
$$
\end{proposition}
\begin{proof}
The proof is deferred to Appendix \ref{app:homo}.
\end{proof}

\zerarcounters
\section{Applications}\label{appli}
In order to use the {\em Mel'nikov conditions} in Definition \ref{pippopuffo3} instead of the {\em  homological equation} in
 Definition \ref{pippopuffo2} in Theorem \ref{thm:kambis} we need to prove that also item $2.d$ of Definition \ref{pippopuffo2} holds. 
This latter point depends strongly on the application so we discuss it in various examples. 

Clearly the simplest possible case is  $\ell_{a,p}=0$ or a finite dimensional space. In any case we need to work in some subspace $\calE$ endowed with some
structure (say reversible or Hamiltonian). To this purpose we restrict $\ell_{a,p}$ as follows.
\begin{defi}\label{prodotto}
	We assume that $\ell_{a,p}$ has a product structure  $\ell_{a,p}= h_{a,p}\times h_{a,p}$ with $w=(z^+,z^-)$ and $h_{a,p}$ is a scale of Hilbert spaces w.r.t. a norm $\|\cdot\|_{a,p}$ satisfying \eqref{sonasega}.  Moreover we assume that the subspaces $\ell_K$ have a product structure  as well $\ell_K= h_K\times h_K$ with the $h_K$ satisfying Hypothesis \ref{hyp22}.
\end{defi}

%
%
%
 \subsection{Example 1: Reversible Nash-Moser.}\label{exNM}
 Let us first discuss the ``minimal choice'', i.e. where in all the definitions we make the simplest possible choices.
 
 Clearly the minimal choice for $\calX$ is
 \begin{equation}\label{Xmin}
  \calX:=  \VV^{(y,0)}\oplus \VV^{(w,0)},
 \end{equation}
 whereas for $\NN$ one can make for instance the classical choice
  \begin{equation}\label{sottoMIN}
 \NN:= \VV^{(\theta,0)}\oplus \VV^{(w,w)}\oplus \VV^{(y,y)}\oplus \VV^{(y,w)}\oplus
  \VV^{(w,y)}. 
 \end{equation}

The decomposition \eqref{Xmin} and \eqref{sottoMIN} is
{\em trivially} triangular, see Definition \ref{triang}, provided that we set  $\mathtt b=1$ 
  since for any $R\in \RR, X\in \calX$ one has 
  \begin{equation}\label{oh}
  \Pi_\calX[R,X] =0\,.
  \end{equation}
  
   Note that  it is a degree decomposition
  with $\mathtt d(y)=1$, where   $\NN$ is generated by all the monomials of degree zero and  $\calX$ is generated by all
  those of negative degree.

 We choose the regular vector fields   as $\calA= \calX$, by setting for $f=\big(0,f^{(y)}(\theta),f^{(w)}(\theta)\big)$, 
 $$
 |f|_{\vec v,p}:= \|f\|^{(1)}_{\vec v,p}= \|f\|_{\vec v,p},
 $$ 
 with the projectors $\Pi_K$ defined as 
%
\begin{equation}\label{tronca}
\begin{aligned}
&(\Pi_{K}f^{(y,0)})(\theta):=\sum_{|\ell|\leq K}f^{(y,0)}_{\ell}e^{\ii \ell\cdot \theta},\\
&(\Pi_{K}f^{(w,0)})(\theta):=\sum_{|\ell|\leq K}\Pi_{\ell_{K}}f^{(w,0)}_{\ell}e^{\ii\ell\cdot \theta}.
\end{aligned}
\end{equation}

 \begin{lemma}\label{esempio1}
 The regular vector fields defined above satisfy all the properties of Definition \ref{linvec-abs};
 moreover the norm $|f|_{\vec v,p}$ is a {\em sharp} tameness constant for all $p\ge \gotp_1$.
 \end{lemma}
 \begin{proof}
Item $(1)$ is trivial and the bound \eqref{tameconst2} follows essentially by an explicit computation (see the proof of Lemma \ref{stolemma} for more details).
 Now by 
 Definition \ref{tame} we have that the bounds $(T_{m})$ hold for any change of variables $\Phi$ and any $y,w$.
 Hence for $\Phi\equiv\uno$ and $y=0=w$, for any $p\geq\gotp_0$, one has
 \begin{equation}\label{ese100}
 |f|_{\vec{v},p}=\|f\|_{\vec{v},p}=\|f(\Phi)\|_{\vec{v},p}\leq C_{\vec{v},p}(f)+C_{\vec{v},\gotp_0}(f)\|\Phi\|_{\vec{v},p}\leq \mathtt{c}C_{\vec{v},p}(f),
 \end{equation}
  where the last inequality holds since $\|\Phi\|_{\vec{v},p}\equiv1$ independently of $p$
  (recall that the map $\Phi$ is evaluated at $w=y=0$).
 This means that  $|f|_{\vec v,p}$ is a { sharp} tameness constant:
   this  trivially implies that item  $(2)$ in Definition \ref{linvec-abs}   hold.
  Let us check item $(3)$. Recall the definition of the projectors in \eqref{tronca} and of the norm in \eqref{totalnorm}; for $\tv=\theta,y$ one has
  \begin{equation}\label{ese101} 
  \begin{aligned}
  \|\Pi_{K}f\|^{2}_{s+s_1,a,p+p'}&=\sum_{|l|\leq K}|f^{(\tv,0)}_l|^2\langle l\rangle^{2(p+p')}e^{2|l|(s+s_1)}\\
  &\leq C 
  K^{2p'}e^{2Ks_1}\sum_{l\in\ZZZ}|f^{(\tv,0)}_l|^2\langle l\rangle^{2p}e^{2|l|s}=  K^{2p'}e^{2Ks_1}  \|\Pi_{K}f\|^{2}_{s,a,p}.
  \end{aligned}
  \end{equation}
  The latter bounds holds also for the norm \eqref{ancoralip}, hence \eqref{P1} holds.
  Similarly the estimate \eqref{P1} holds also for $\tv=w$. Moreover for $\tv=w$ we can write 
  $$
  (\uno-\Pi_{K})f^{(w,0)}(\theta)=\sum_{|l|>K}e^{\ii \ell\cdot \theta}f_{l}^{(w,0)}+\sum_{|l|\leq K}(\uno-\Pi_{\ell_K})f_{l}^{(w,0)}e^{\ii\ell\cdot\theta},
  $$
  hence one has for $p,p'\in\NNN$ 
  \begin{equation}\label{ese102}
  \begin{aligned}
  \| (\uno-\Pi_{K})f^{(w,0)}(\theta)\|^{2}_{s,a,p}&\leq \sum_{|l|>K}\langle l\rangle^{2p}\|f_l^{(w,0)}\|^{2}_{a,\gotp_0}e^{2s|l|}
  +\sum_{|l|>K}\langle l\rangle^{2p}e^{2s|l|}\|(\uno-\Pi_{\ell_K})f^{(w,0)}_l\|_{a,\gotp_0}^{2}\\
  & + \sum_{|l|>K}\langle l\rangle^{2\gotp_0}\|(\uno-\Pi_{\ell_K})f_l^{(w,0)}\|^{2}_{a,p}e^{2s|l|}
   + \sum_{l\in\ZZZ}\langle l\rangle^{2\gotp_0}e^{2s|l|}\|(\uno-\Pi_{\ell_K})f^{(w,0)}_l\|_{a,p}^{2}\\
  &  + \sum_{|l|>K}\langle l\rangle^{2\gotp_0}\|\Pi_{\ell_K}f_l^{(w,0)}\|^{2}_{a,p}e^{2s|l|}
  +\sum_{|l|\leq K}\langle l\rangle^{2p}e^{2s|l|}\|(\uno-\Pi_{\ell_K})f^{(w,0)}_l\|_{a,\gotp_0}^{2}\\
  &\leq  2K^{-2p'}\sum_{l\in\ZZZ}\langle l\rangle^{2(p+p')}\|f_l^{(w,0)}\|^{2}_{a,\gotp_0}e^{2s|l|}\\
&+ 2K^{-2p'}\sum_{l\in\ZZZ}\langle l\rangle^{2\gotp_0}\|f_l^{(w,0)}\|^{2}_{a,p+p'}e^{2s|l|}\\
&+ \mathtt{c}\sum_{l\in\ZZZ}\langle l\rangle^{2(p+p')}K^{-2(p'+p)}K^{2\gotp_0}K^{2(p-\gotp_0)}\|f_l^{(w,0)}\|^{2}_{a,\gotp_0}e^{2s|l|}\\
&+\mathtt{c}\sum_{l\in\ZZZ}\langle l\rangle^{2\gotp_0}e^{2s|l|}K^{2(p-\gotp_0)}K^{-2(p+p'-\gotp_0)}\|f^{(w,0)}_l\|_{a,p+p'}^{2}\\
&\leq C K^{-2p'}\|f^{(w,0)}\|^{2}_{s,a,p+p'},
  \end{aligned}
  \end{equation}
    and the latter bounds holds also for the norm \eqref{ancoralip}. Similar bounds holds also for $\tv=\theta,y$, hence \eqref{P2} holds.
    Condition \eqref{P3} is trivial. Finally, items 
    $4$ and $5$ can be checked easily since the map generated by vector fields in $\calA$ are simply translations.
  \end{proof}

  By looking at the homological equation it is clear that the minimal requirement for the vector field is that
  $F^{(y)}(\theta,y,0)=-F^{(y)}(-\theta,y,0)$, otherwise even when $\ell_{a,p}=\emptyset$ one can easily produce
examples in which invariant tori do not exist\footnote{Consider for instance $\dot y=1$.}.

Following  Sevryuk (see for instance \cite{Sev} and references therein) one expects to require that the vector field  satisfies some
appropriate symmetry: this can be stated by saying that the vector field is reversible w.r.t. some involution.
  Naturally one needs also the ``unperturbed vector field'' $N_0$ to be reversible w.r.t. the chosen involution, being the vector field that identifies the approximately
  invariant torus. In the applictaions to PDEs, one typically deals with $N_0$ of the form
  \begin{equation}\label{enneo}
N_0= \omega^{(0)}\cdot\partial_{\theta} + \ii \Lambda^{(0)} w \partial_w = \omega^{(0)}\cdot\partial_{\theta} + \ii \Omega^{(0)} z^+ \partial_{z^+} - \ii \Omega^{(0)} z^- \partial_{z^-}
\end{equation} 
with  $\Omega^{(0)}$ a linear operator which is $\theta$-independent and   block-diagonal w.r.t. all the $h_{K}$. Thus $N_0$ is a diagonal operator as in Definition
\ref{norm}. 
Unfortunately such $N_0$ is not reversible w.r.t. the ``simple'' involution $(\theta,y,w)\to(-\theta,y,w)$, but it is reversible
 w.r.t. the involution $S:(\theta,y,(z^+,z^-))\to (-\theta,y,(z^-,z^+))$.
  Therefore,
 as for the subspace $\calE $ we choose 
 \begin{equation}\label{minimarev}
  \calE=\calE^{(0)}:= \left\{F\in 
  \VV_{\vec v,p}:\quad \begin{pmatrix}
  F^{(\theta)}(-\theta,y,(z^-,z^+))\\F^{(y)}(-\theta,y,(z^-,z^+))\\
  F^{(z^+)}(-\theta,y,(z^-,z^+))\\F^{(z^-)}(-\theta,y,(z^-,z^+))
  \end{pmatrix}
  = - \begin{pmatrix}
  -  F^{(\theta)}(\theta,y,(z^+,z^-))\\F^{(y)}(\theta,y,(z^+,z^-))\\
    F^{(z^-)}(\theta,y,(z^+,z^-))\\F^{(z^+)}(\theta,y,(z^+,z^-))
    \end{pmatrix} \right\}\,,
 \end{equation}
i.e. the vector fields which are reversible w.r.t. the involution $S$.
  
The conditions of Definitions \ref{nomec} and \ref{nx} are trivially fulfilled with $\mathtt n=1$ and 
\begin{equation}\label{beby}
\BB_\calE:=\{g= (0,g^{(y)}(\theta),g^{(w)}(\theta))\,: \qquad  g^{(y)}(-\theta)=g^{(y)}(\theta)\,, \quad g^{(z^+)}(-\theta)= g^{(z^-)}(\theta) \}\,.
\end{equation}

Now consider a vector field of the form
$\calE\ni F= N_0+G$ and 
our aim is to apply Theorem \ref{thm:kambis} to $F$
 provided that $G$ is sufficiently small. 
  The simplest possible choice of compatible change of variables  is, $\calL_n:=\uno$ for all $n$.
   With such choices, our scheme is the standard a Nash-Moser algorithm to find solutions of the torus embedding equation \eqref{gigi}. 
   Indeed each $\Phi_n$ is a
   traslation in the $y,w$ direction 
   $$
   y\to y+g_n^{(y)}(\theta)\,,\quad w\to w+g_n^{(w)}(\theta)
   $$  
  so that 
  \begin{equation}\label{traslazione}
  \calH_n: \; y\to y+h_n^{(y)}(\theta)\,,\quad w\to w+h_n^{(w)}(\theta)\,,\quad h_n= \sum_{j=0}^ng_j \,, 
  \end{equation}
  $$
   F_n= F( \theta,y+h_n^{(y)},w+h_n^{(w)} )- \partial_\theta h_n \cdot  F^{(\theta)}(\theta,y+h_n^{(y)},w+h_n^{(w)}).
   $$ 
   Note that $h_n$ is simply an {\em approximate solution}  for the torus embedding equation  \eqref{gigi}, indeed one has that
   $F_n^{(y)}(\theta,0,0), F_n^{(w)}(\theta,0,0)\to 0$ as $n\to \infty$.
    
  The difference with the standard Nash-Moser algorithm is therefore only in the point of view: instead of looking for a torus embedding, we are looking for
  a translation in the $y,w$ variables which puts the embedding to zero. 
  
  With the above assumptions and assuming also that the smallness conditions \eqref{sizes} are satisfied, then we can apply
  Theorem \ref{thm:kambis}. Now we show that in this case the set of parameters satisfying the the Mel'nikov conditions
  of Definition  \ref{pippopuffo3} also satisfies the   homological equation of Definition \ref{pippopuffo2}.
 
  \smallskip 
%
%
%
%

 \smallskip

 \begin{proposition}\label{uffffa1.1}
 Let $\gamma> 0$, $K\ge K_0$,
 consider a compact set $\calO \subset \calO_0$ and set $\vec{v}=(\g,\calO,s,a,r)$ and $\vec{v}^{\text{\tiny 0}}=(\g,\calO_0,s,a,r)$.
 Consider the  vector field $F\in \calW_{\vec v^{\text{\tiny 0}},p}$ with (see \eqref{enneo})
 $$
 F= N_0+G: \calO_0\times    D_{a,p+\nu}(r)\times\TTT^{d}_{s}\to V_{a,p}\,, 
 $$
  which is  $C^{3}$-tame up to order $q=\gotp_2+2$.  
 Assume that $\g\sim {\rm diam} \calO_0$ and 
 $F\in \calE$ defined in \eqref{minimarev}.
 If $\calO$ satisfies the Mel'nikov conditions of Definition \ref{pippopuffo3} for $(F,K,\vec{v}^{\text{\tiny 0}})$  then 
 $\calO$ satisfies the Homological equation of Definition \ref{pippopuffo2} provided that we fix 
 parameters $\mu=\mu_1$.
 \end{proposition}
\begin{proof}
We note that \eqref{oh} implies $\Pi_\calX {\rm ad} (\Pi_\calX^\perp G)$= $ \Pi_\calX {\rm ad} (\Pi_\NN G)$ so we may set
$ g= \gotW \, \Pi_K\Pi_\calX G\in \BB_\calE$ for $\xi\in \calO$. It is easily seen that the first of \eqref{buoni22} follows from \eqref{buoni}. As for the second equation we use the sharpness of $|\cdot|_{\vec v,p}$. Indeed 
by Lemma \ref{commu2} we know that $C_{\vec v,p+1}(G)|g|_{\vec v,\gotp_1} +  C_{\vec v,\gotp_1}(G)|g|_{\vec v,p+\nu+1}$ is a tameness constant for $[\Pi_{X}^\perp G,g]$. Then the bound follows by Lemma \ref{esempio1}. 
Regarding 2.c, one simply notes that formula \eqref{cribbio4} implies \eqref{cribbio42}. \\
Finally in order to prove 2.d, we start by showing that  the inequality \eqref{cribbio420} holds by substituting the l.h.s.
with a tameness constant $C_{\vec{v},\gotp_2-1}(\Pi_{\calX}\Phi_{*}F)$; this follows from
 Lemmata \ref{conj}, \ref{derivate}, \ref{normnorm} and Remark \ref{monomi}, see the proof of estimate \eqref{kam9} for more details.  Therefore, the bound \eqref{cribbio420} follows  from the sharpness of $|\cdot|_{\vec{v},p}$ for any $p$.
\end{proof} 

By Proposition \ref{uffffa1.1} the set $\calO_{\io}$ of Theorem \ref{thm:kambis} contains the intersection over $n$ of the
sets in which the Mel'nikov conditions are satisfied for $(F_n,K_n,\vec{v}_n^{\text{\tiny 0}})$, therefore we now
analyze the Mel'nikov conditions.
The operator $\Pi_\calX {\rm ad}(\Pi_\NN F_n)$ has the form
\begin{equation}\label{merdina1}
\Pi_\calX {\rm ad}(\Pi_\NN F_n)=\big( F^{(\theta)}_n(\theta,0,0)\cdot \partial_\theta\big)\,\uno  +\begin{pmatrix}
 F^{(y,y)}_n(\theta) &  F^{(y,w)}_n(\theta)\\
 F^{(w,y)}_n(\theta)  & F^{(w,w)}_n(\theta)
\end{pmatrix}\,.
\end{equation}
Recall that  $F^{(\tv_i,\tv_j)}$ are defined in \eqref{aiuto}. Note that this operator  maps $\BB_\calE$ in $\calX\cap\calE$. 
Finding $\gotW$ satisfying \eqref{buoni} and \eqref{cribbio4}  is now equivalent to finding an approximate
inverse for a $K_n$-truncation of \eqref{merdina1}, which unfortunately seems a quite delicate question.

  A possible simplification occurs  if instead of \eqref{sottoMIN} we consider  the decomposition (recall \eqref{lemedie})
  $$
  \NN:= \langle\VV^{(\theta,0)}\rangle \oplus \VV^{(w,w)}
  \oplus \VV^{(y,y)}\oplus \VV^{(y,w)}\oplus
    \VV^{(w,y)}\
   \quad \calX=\calA:= \VV^{(y,0)}\oplus \VV^{(w,0)}\oplus \VV^{(\theta,0)}_0\,,
   $$   
   and leave $\calE$ unchanged; it is easily seen that the equivalent of Lemma \ref{esempio1} holds, and that
   \begin{equation}\label{beby2}
   \BB_\calE:=\{g= (g^{(\theta)}(\theta),g^{(y)}(\theta),g^{(w)}(\theta))\,: \; g^{(\theta)}(-\theta)=-g^{(\theta)}(\theta)\,,\;  g^{(y)}(-\theta)=g^{(y)}(\theta)\,, \; g^{(z^+)}(-\theta)= g^{(z^-)}(\theta) \}\,.
   \end{equation}  
   Note that in this case we would obtain a stronger result, since the dynamics on the model torus would be linear.
   
 We divide $\calX$ by degree decomposition as in  \eqref{blodia}, with $\mathtt b=2$  and $\calX_1= \VV^{(y,0)}\oplus \VV^{(w,0)}$,
  $\calX_2= \VV^{(\theta,0)}_0$; this decomposition is triangular by Remark \ref{sonfica} and the equivalent of Proposition \ref{uffffa1.1} holds.
  
   As before we
  fix $\calL_n=\uno$ for all $n$. Now the maps $\Phi_n$ are a translation in the $y,w$ direction  composed with
  a torus diffeomorphism . They are hence
  of the form
    \begin{equation}\label{criceto200}
  \theta\to \theta+h_n^{(\theta)}(\theta), \quad y\to y+h_n^{(y)}(\theta)\,,\quad w\to w+h_n^{(w)}(\theta)
  \end{equation}
  defined in such a way that
  $F_n^{(\theta,0)}(\theta)= \omega^{(n)} + O(|g_n|)$ (here $\omega^{(n)}$ is the average of
  $F_n^{(\theta,0)}(\theta)$ w.r.t. $\theta$). 
  
  Regarding the Mel'nikov conditions we have that, by definition, $\Pi_\calX {\rm ad}(\Pi_\NN F_n)$
  is block-diagonal on $\calX=\calX_1\oplus \calX_2$ and its action on $\calX_1$ is of the form
  \begin{equation}\label{merdina2}
 \big( \omega^{(n)}\cdot \partial_\theta\big)\,\uno  +\begin{pmatrix}
   F^{(y,y)}_n(\theta) &  F^{(y,w)}_n(\theta)\\
   F^{(w,y)}_n(\theta) & F^{(w,w)}_n(\theta)
  \end{pmatrix}\,,
  \end{equation}
  while the action on $\calX_2$ is simply $\omega^{(n)}\cdot \partial_\theta$. Thus the Mel'nikov conditons \eqref{buoni},\eqref{cribbio4} on the component $\calX_2$ amount to requiring that $\omega^{(n)}$ is $\gamma,\tau$ diophantine up to order $K_n$. All the difficulty is now reduced to inverting \eqref{merdina2}.
  
  Note that, under the same Diophantine hypotheses on  $\omega^{(n)}$, the operator \eqref{merdina1} can be reduced to the form \eqref{merdina2} by choosing
  at each step $n$, the change of variables  $\calL_n$ to be the torus diffeomorphism 
  which reduces $F_n^{(\theta)}(\theta,0,0)$ to its mean value. Of course one needs to verify that
  the $\calL_n$ are in fact a sequence of compatible changes of variables as in Defintion \ref{compa}. 
  
  If we assume that the subspaces $\ell_K$ of Hypothesis \ref{hyp22} are finite dimensional then the invertibility of 
  $ \Pi_{K_n} \Pi_\calX {\rm ad}(\Pi_\NN F_n) \Pi_{K_n}$ can be imposed by requiring that its eigenvalues are non-zero (the so-called first Mel'nikov condition);
   however, unless $\ell_K$ is uniformly bounded (i.e. when $\ell_{a,p}$ is a finite dimensional space)
   it is not at all trivial to obtain from such condition the bounds \eqref{buoni} and \eqref{cribbio4}.
   
   To the best of our knowledge the only examples in which one has enough control on \eqref{merdina2} as it is, are the forced cases,
   i.e. when $F^{(\theta)}=\om_0$ and there are no $y$ variables, that is $d_1=0$. In this case one can use the so-called multiscale approach; see for instance
   \cite{Bo4,BB1,BCP}. Otherwise one needs a more refined decomposition; see below.
\medskip
%
%
%

   \subsection{Example 2: Hamiltonian KAM/Nash-Moser.}\label{exHam}
  
 The following section is essentially a reformulation in our notations of the approach proposed in \cite{BB3}.  We start by remarking that when the vector field \eqref{kam1} is Hamiltonian, it is natural to apply to it only symplectic changes of variables: this amounts to 
completing  the maps introduced in \eqref{criceto200} to symplectic ones.
 We now describe our procedure and at the end of the subsection we state the Theorem with an application to the NLS equation.

\begin{defi}[{\bf  Symplectic structure}]\label{struttHAM}
Recall that we assumed $\ell_{a,p}$ to have the  product structure of Definition \ref{prodotto}.
We  endow  the phase space with the symplectic structure  $d \theta\wedge dy +\ii d z^+\wedge d z^-$. 
\end{defi}
We consider the decomposition
\begin{equation}\label{sottoHAM}
\NN:= \av{\VV^{(\theta,0)}}\oplus \VV^{(w,w)}\oplus \VV^{(y,w,w)}\,,
\quad \calX:=  \VV_0^{(\theta,0)} \oplus \VV^{(y,0)}\oplus \VV^{(y,y)} \oplus
 \VV^{(y,w)}\oplus \VV^{(w,0)} \,.
\end{equation}
This decomposition satisfies \eqref{nomec} with $\mathtt n=2$ and it is a degree decomposition with ${\rm deg}(y)=2$.
Using the notation of \eqref{blodia} we have $\tb= 3 $ and
$$
\calX_1=\VV^{(y,0)}\,,\quad \calX_2= \VV^{(y,w)}\oplus \VV^{(w,0)}\,,\quad \calX_3=  \VV_0^{(\theta,0)} \oplus \VV^{(y,y)}.
$$

We remark that if one wants to solve the torus embedding equation taking advantage of the Hamiltonian structure, then the decompostion \eqref{sottoHAM} appears naturally since it is the minimal decomposition containing \eqref{sottoMIN} and preserving the Hamiltonian structure. More precisely given a  change of
coordinates as in \eqref{criceto200}, completing it to a symplectic one produces an element of $\calX$ in \eqref{sottoHAM}. 

Note that by (i) of Definition \ref{nx}, the bigger is the set $\calX$, the more delicate is the choice of $\calA$.

     \begin{defi}[{\bf Finite rank  vector fields}]\label{linvec}
 	We consider vector fields $f: \TTT^d_s\times D_{a,p+\nu}(r )\to V_{a,p}$ of the form
 	\begin{equation}\label{maremma}
 	\begin{aligned}
 	&f = \sum_{v\in \mathtt V} f^{(v,0)}\del_v + (f^{(y,y)} y +f^{(y,w)}\cdot w )\cdot\del_y \,,\\ 
 	&f^{(y_i,w)}\in H^p(\TTT^d_s;\ell_{-a,-\gotp_0-\nu})\cap H^{\gotp_0}(\TTT^d_s;\ell_{-a,p-\gotp_1-\gotp_0-\nu}) \,,\quad \langle f^{(\theta,0)}\rangle =0
 	\end{aligned}
 	\end{equation}
 	and we set for $p\geq\gotp_1$
 	\begin{equation}\label{linnorm}
	\begin{aligned}
 	|f|_{s,a,p}:&=\sum_{u=\theta,y,w} \|f^{(u,0)}\|_{s,a,p} +
 	\max_{i,j=1,\dots,d_1}\|f^{(y_i,y_j)}\|_{s,a,p}+\\
	+\frac{1}{r_0^{\mathtt{s}}} &\max_{i=1,\ldots,d_1}\left(
	\|f^{(y_i,w)}\|_{H^{p}(\TTT^{d}_{s};\ell_{-a,-\gotp_0-\nu})}+\|f^{(y_{i},w)}\|_{H^{\gotp_0}(\TTT^{d}_{s};\ell_{-a,p-\gotp_1-\gotp_0-\nu})}\right)
	\end{aligned}
 	\end{equation}
 	We say that $f$ is 	 of \emph{finite rank}  if $	|f|_{s,a,p}<\infty$. We denote by $\calA_{s,a,p}$ the space of finite rank vector fields. \\
 	Given a compact set $\calO\subseteq \calO_0$ we  denote by $\calA_{\vec v,p}$ with $\vec v= (\g,\calO,s,a,r)$ the set of Lipschitz families $\calO\to \calA_{s,a,p}$ with the  corresponding $\g$-weighted Lipschitz norm which we denote by $|\cdot|_{\vec v,p}$.
 \end{defi}
 \begin{rmk}
 	Note that $\VV^{(y,w)}$ is not contained in the set of finite rank vector fields; indeed in general by the identification of $\ell_{a,p}^*$ with $\ell_{-a,-p}$ one has that
	and  $g\in \VV^{(y_i,w)}$
 	can be written as  $g^{(y_i,w)}(\theta)\cdot w \,\del_{y_i}$ where
	$$
	g^{(y_i,w)}\in  H^{p}(\TTT^d_s,\ell_{-a,-\gotp_0-\nu})\cap H^{\gotp_0}(\TTT^d_s,\ell_{-a,-p-\nu})\,.
	$$
	On the other hand \eqref{maremma} is a stronger condition.
 	Our -- notationally quite unpleasant -- choice of $\ell_{-a,p-\gotp_1-\gotp_0-\nu}$ is needed in order to verify condition \eqref{tameconst3} in Definition \ref{linvec-abs}. 
 \end{rmk}
 \begin{defi}\label{projector}
 Given $K>0$ and a vector field  $f\in\calA$ 
   we define the projection $\Pi_K f$  as
   \begin{equation}\label{duck33333}
 \begin{aligned}
 &(\Pi_{K}f^{(\tv,0)})(\theta):=\sum_{|\ell|\leq K}f^{(\tv,0)}_{\ell}e^{\ii \ell\cdot \theta}, \qquad \tv=\theta,y\,,\\
 &(\Pi_{K}f^{(w,0)})(\theta):=\sum_{|\ell|\leq K}\Pi_{\ell_{K}}f^{(w,0)}_{\ell}e^{\ii\ell\cdot \theta},
 \quad (\Pi_{K})f^{(y_i,y_j)}(\theta):=\sum_{|\ell|\leq K}f^{(y_i,y_j)}_{\ell}e^{\ii \ell\cdot\theta},i,j=1,\ldots,d\,,\\
 &(\Pi_{K}f^{(y_i,w)})(\theta):=\sum_{|\ell|\leq K}\Pi_{\ell_{K}}f^{(y_i,w)}_\ell e^{\ii \ell\cdot\theta}\,,
 \end{aligned}
 \end{equation}
 %
 %
 %
 and we define $E^{(K)}$ as the subspace of $\calA_{\vec v,p}$ where $\Pi_K$ acts as the identity.
 \end{defi}  
 \begin{lemma}\label{grrr}
The finite rank vector fields of Definition \ref{linvec} satisfy all the conditions  of Definition \ref{linvec-abs}. 
\end{lemma}

\begin{proof}
The Hilbert structure comes from the fact that \eqref{linnorm} is defined by using the norm of an Hilbert space on each component.
Item $1$ follows by the definition while
item $2$ formula \eqref{tameconst2} is proved in \ref{stolemma}. 
To prove bound \eqref{tameconst3} in item $2$ we reason as follow.
Let us study the $(y,w)$ component since the other are trivial.
For $\Phi=\uno$ one has that
\begin{equation}\label{tatuo}
\|f^{(y,w)}\circ\Phi\|_{s,a,\gotp_1}=\max_{i=1,\ldots,d}\frac{1}{r_0^{\mathtt{s}}}\sum_{l\in\ZZZ}\langle l\rangle^{2\gotp_1}|f_{l}^{y_i,w}\cdot w |^{2}\leq 
\max_{i=1,\ldots,d}\frac{1}{r_0^{\mathtt{s}}}\sum_{l\in\ZZZ}\langle l\rangle^{2\gotp_1}\|f^{(y_i,w)}_{l}\|^{2}_{-a,-\gotp_0-\nu}\|w\|^{2}_{a,\gotp_0+\nu}
\end{equation}
using the Cauchy-Schwartz inequality. By the sharpness of the latter inequality we deduce that any tameness constant must be larger that the right hand side of \eqref{tatuo}, which in turn is bounded from below by  $\frac12 |f^{(y,w)}|_{s,a,\gotp_1}$.

Items   $4,5$ are proved in  Lemmata \ref{cometichiami} and \ref{flusso} in the Appendix.

Finally we need to show that item $3$ holds. 
Now given a vector field $f\in \calA_{\vec{v},p}$, the components $f^{(\tv,0)}$ are discussed in \eqref{ese101} and \eqref{ese102},
and the components $f^{(\tv,0)}$ and $f^{(y,y)}$ can be treated in the same way.
By the definition of the projector \ref{projector} one has that the last component $(y,w)$ behaves essentially
as the component $(w,0)$. Hence again the smoothing bounds hold by reasoning as done in \eqref{ese101} and \eqref{ese102}.
\end{proof}

We choose  ($J$ is the standard symplectic matrix)
\begin{equation}\label{spazioham}
\calE=\calE^{(0)}_{\rm Ham}:=\Big\{ F\in \VV_{\vec v,p}:\;  F=(\partial_y H,-\partial_\theta H,\ii J \partial_w H 
 )\,,\quad H(\theta,y,z^+,\ol {z^+})\in \RRR\Big\},
\end{equation}
while the regular vector fields are given by Definition \ref{linvec}. 
Note that by construction $\calE^{(0)}_{\rm Ham}\cap \calX\equiv \calA$, indeed the condition $J\partial_{w} H(\theta,y,w)=F^{(w)}(\theta,y,w)\in \ell_{a,p}$ implies that $\partial_w F^{(y)}:= -\partial_w\partial_\theta H(\theta,y,w) \in \ell_{a,p}$ as well. 
Then $\BB_\calE$ is the space of regular Hamiltonian vector fields.
The conditions of Definitions \ref{nomec} and \ref{nx} are trivially fulfilled. Note that the degree decomposition preserves the Hamiltonian structure.
\begin{lemma}\label{sharpa}
Consider a tame vector field $f\in \calA_{\vec{v},p}\cap \calE$ (i.e. regular vector field according to Definition \ref{linvec} which is Hamiltonian).  There exists a $\mathtt{c}$  (depending  at most on $\gotp_0$ and 
on the dimensions $d,d_1$) such that for any tameness constant 
 \begin{equation}\label{tameconst33}
   |f|_{\vec{v},p} \le \mathtt{c} \, C_{\vec{v},p+1}(f)
\end{equation}
for any $p\geq \gotp_1$.
\end{lemma}

\begin{proof}
On the components $(\tv,0)$, $\tv=\theta,y,w$ and $(y,y)$ the bound \eqref{tameconst33} is proved in Lemma \ref{esempio1}. Let us study the $(y,w)-$component. 
First recall that, since we are in a Hamiltonian setting, then one has
$f^{(y,w)}(\theta)=-i J \del_{\theta}f^{(w,0)}(\theta)$. Hence  for $\Phi\equiv\uno$ and $y=0=w$ one has
 for any $p\geq\gotp_0$ that 
\begin{equation}\label{sharpa2}
|f^{(y,w)}(\theta)|_{\vec{v},p}\leq|\del_{\theta}f^{(w,0)}(\theta)|_{\vec{v},p}\leq |f^{(w,0)}(\theta)|_{\vec{v},p+1}=\|f^{(w,0)}\circ{\Phi}\|_{\vec{v},p+1}\leq\mathtt{c} C_{\vec{v},p+1}(f).
\end{equation} 
Threfore the assertion follows.
\end{proof}

As in Subsection \ref{exNM} we now relate the Mel'nikov conditions to the  homological equation.

\begin{proposition}\label{uffffa2}
Let $\gamma> 0$, $K\ge K_0$,
consider a compact set $\calO \subset \calO_0$ and set $\vec{v}=(\g,\calO,s,a,r)$ and $\vec{v}^{\text{\tiny 0}}=(\g,\calO_0,s,a,r)$.
Consider a  vector field $F\in \calW_{\vec v^{\text{\tiny 0}},p}\cap \calE^{(0)}_{\rm Ham}$ of the form
$$
F= N_0+G: \calO_0\times    D_{a,p+\nu}(r)\times\TTT^{d}_{s}\to V_{a,p}\,, 
$$
where $\calE^{(0)}_{\rm Ham}$ is defined  \eqref{spazioham} and $N_0$ is defined in \eqref{enneo} with $\Omega^{(0)}$ self-adjoint.
Assume that $F$ is  $C^{4}$-tame up to order $q=\gotp_2+2$. 
Assume that $\g\sim {\rm diam} \calO_0$ and set
 \begin{equation}\label{sizes2bb}
\Gamma_{p}:=\g^{-1}C_{\vec{v},p}(G), \quad 
 \Theta_{p}:=\g^{-1}C_{\vec{v},p}(\Pi_{\NN}^\perp G).
 \end{equation}

If $\calO$ satisfies the Mel'nikov conditions of Definition \ref{pippopuffo3} for $(F,K,\vec{v}^{\text{\tiny 0}})$  then 
$\calO$ satisfies the  homological equation of Definition \ref{pippopuffo2} provided that we fix 
parameters $\mu$ and $\mathtt{t}$ as in \eqref{migo}.
\end{proposition}

\begin{proof}
We wish to apply  Proposition \ref{uffffa} in order to prove that items 1. and 2.a-b-c  of Definition \ref{pippopuffo2} are satisfied for $\calO$ satisfying the
Mel'nikov conditions. In order to do so we need to prove \eqref{vino2}. The desired bounds follow from Lemma \ref{sharpa} and from the bounds \eqref{commu2} on tameness constants of commutators. We now prove item  2.d. We claim that there exists a choice of a tameness constant $C_{\vec{v},\gotp_2-1}(\Pi_{\calX}\Phi_{*}F)$ which satisfies \eqref{cribbio420}. Indeed this follows from
 Lemmata \ref{conj}, \ref{derivate}, \ref{normnorm} and Remark \ref{monomi}, see the proof of estimate \eqref{kam9} for more details.  The bound \eqref{cribbio420} follows  from Lemma \ref{sharpa}.
 
In fact one may prove  \eqref{cribbio420} directly (obtaining a slightly better bound). Let  $\Phi=\uno+f$ be the
time one flow map of the field $g$ defined by Proposition \ref{uffffa} and $\Phi^{-1}=\uno+\tilde{f}$ its inverse. Note that $f$ is of the form \eqref{maremma}
and $\Pi_{\calX}\Phi_{*}F=\Pi_{\calX}F\circ{\Phi^{-1}}+df[F\circ{\Phi^{-1}}]$. The bound on the first summand follows by item 1. 
The only non trivial term in the second summand is given by
 $$
 f^{(y,w)}(\theta)\big[d_{w}F^{(w)}({\Phi^{-1}(\theta,0,0)})[w]\big]=\big((d_{w}F^{(w)}({\Phi^{-1}(\theta,0,0)}))^{*}f^{(y,w)}(\theta)\big)\cdot w.
 $$
 By the Hamiltonian structure the operator\footnote{We use the standard notation for the Pauli matrices.}
  $$i \s_3  d_{w}F^{(w)}\,,\quad \s_3:=\begin{pmatrix}
 1 &0 \\ 0 &-1 
 \end{pmatrix}$$ is self-adjoint, hence the bound \eqref{cribbio420} follows by the tame estimates on $F$
 and the fact that $f^{(y,w)}\in \ell_{-a,-\gotp_0-\nu}$. 
 Hence the assertion follows.
\end{proof}

Again, under the same assumptions as in Proposition \ref{uffffa2} and of course assuming  the smallness conditions \eqref{sizes}, we can apply Theorem \ref{thm:kambis};
by Proposition \ref{uffffa2} the set $\calO_{\io}$ of Theorem \ref{thm:kambis} contains the intersection over $n$ of the
sets $\mathcal{C}_{n}$ in which the Mel'nikov conditions are satisfied for $(F_n,K_n,\vec{v}_n^{\text{\tiny 0}})$, therefore we now
analyze the Mel'nikov conditions.

The operator 
$\Pi_\calX{\rm ad}(\Pi_\NN F_n)$, restricted to  the blocks  $\calX_1,\calX_3$ coincide with the  operator $(\omega^{(n)}\cdot \partial_{\theta})\,\uno$  
 while on the block $\calX_2= \VV^{(w,0)}\oplus\VV^{(y,w)}$,  we get
\begin{equation}\label{merdina3Ham}
\begin{pmatrix}
\big( \omega^{(n)}\cdot \partial_\theta\big)\,\uno  -  F^{(w,w)}_n(\theta) & 0 \\
  - F^{(y,w,w)}_n(\theta) & \big( \omega^{(n)}\cdot \partial_\theta\big)\,\uno  + \big(  F^{(w,w)}_n(\theta)\big)^*
\end{pmatrix}\,.
  \end{equation}

Note that the operators appearing on the diagonal of  \eqref{merdina3Ham} are  $i \s_3$ times a self-adjoint operator, moreover the whole
operator maps Hamiltonian vector fields into Hamiltonian vector fields.

As before, the Mel'nikov conditons \eqref{buoni},\eqref{cribbio4} on the component $\calX_1,\calX_3$ amount to requiring that
$\omega^{(n)}$ is $\gamma,\tau$ diophantine up to order $K_n$. 

In conclusion we have proved the following Theorem.

\begin{theo}\label{esempappli}
Consider a  vector field $F\in \calW_{\vec v^{\text{\tiny 0}},p}\cap \calE^{(0)}_{\rm Ham}$ of the form
$$
F= N_0+G: \calO_0\times    D_{a,p+\nu}(r)\times\TTT^{d}_{s}\to V_{a,p}\,, 
$$
where $\calE^{(0)}_{\rm Ham}$ is defined  \eqref{spazioham} and $N_0$ is defined in \eqref{enneo} with $\Omega^{(0)}$ self-adjoint.
Assume that $F$ is  $C^{4}$-tame up to order $q=\gotp_2+2$. 
Fix $\g>0$ such that $\g\sim {\rm diam} \calO_0$ and assume that $G$ satisfies the smallness conditions \eqref{sizes}
of Theorem \ref{thm:kambis}. Then there exists an invariant torus 
for $F$ provided that $\x$ belong to the set $\calO_{\infty}$ of Theorem \ref{thm:kambis}.
Finally $\calO_{\infty}$ contains  $\bigcap_{n}\mathcal{C}_{n}$
where $\mathcal{C}_{n}$ is the set of $\x$ such that $\oo_{n}$ is $(\g,\tau)-$diophantine
and the matrix $\mathfrak{N}_{n}$ in \eqref{merdina3Ham} is approximatively invertible
with tame bounds like \eqref{buoni} and \eqref{cribbio4}.
\end{theo}

We claim that, in the applications, proving the approximate invertibility of
\eqref{merdina3Ham} is significantly simpler than proving the invertibility of \eqref{merdina2}.

\subsection{Example 3: Reversible KAM/Nash-Moser.}\label{pensavo}

 We now wish to obtain a triangular decomposition for the Melnikov conditions as in \eqref{merdina3Ham} but without restricting to Hamiltonian vector fields. To this purpose we set
 \begin{equation}\label{sottoREV}
 \NN:= \VV^{(\theta,0)}\oplus \VV^{(w,w)}\,,
 \quad \calX:=  \VV^{(y,0)}\oplus \VV^{(y,y)} \oplus \VV^{(y,w)}\oplus \VV^{(w,0)}.
 \end{equation}
 or 
 \begin{equation}\label{sottoREV2}
 \NN:= \av{\VV^{(\theta,0)}}\oplus \VV^{(w,w)}\,,
 \quad \calX:=  \VV^{(\theta,0)}_0 \oplus\VV^{(y,0)}\oplus \VV^{(y,y)} \oplus
  \VV^{(y,w)}\oplus \VV^{(w,0)}.
 \end{equation}
such choices are compatible with Definition \ref{nomec} with $\mathtt{n}=1$.
  Note that  both cases come from  a degree decomposition provided that we fix
  $1<{\rm deg}(y)<2$, therefore they are trivially triangular. Now the degree decomposition of \eqref{blodia}, say in case \eqref{sottoREV2}, reads $\mathtt b=4$ and gives   $\calX_1= \VV^{(y,0)}$, $\calX_2=\VV^{(w,0)}$, $\calX_3= \VV^{(y,w)} $ and $\calX_4= \VV^{(\theta,0)}_0\oplus \VV^{(y,y)}$. 
    
    We 
    define the space of regular vector field $\calA_{\vec v,p}$ as the ``finite rank vector field'' of Definition \ref{linvec}.
and  introduce the smoothing operator $\Pi_{K}$ as  in Definition \ref{projector}. By lemma \ref{grrr} such vector fields satisfy all the conditions  of Definition \ref{linvec-abs}. 

\smallskip

Regarding the choice of $\calE$, we require the reversibility condition \eqref{minimarev}, moreover, in order to satisfy condition (i) and (iii) of Definition \ref{nx} we set  
\begin{equation}\label{spazorev}
\calE=\calE^{(1)}:=\Big\{F\in \calE^{(0)}:\;d_w F^{(y)}(\theta,y,w)\in H^{p}(\TTT^{d}_{s};\ell_{-a,-\gotp_0-\nu})\cap
H^{\gotp_0}(\TTT^{d}_{s};\ell_{-a,p-\gotp_1-\gotp_0-\nu})\,,\;
 \Big\}.
\end{equation}

If we set $\calL_n=\uno$ as before, we get changes of variables of the form
$$
y\to y+ h^{(y,0)}(\theta) + h^{(y,y)}(\theta) y + h^{(y,w)}(\theta)\cdot w \,,\quad w\to w+h^{(w,0)}(\theta)\,,\quad   \theta\to \theta+ h^{(\theta,0)}(\theta)
$$
i.e. the changes of variables \eqref{criceto200} of Example 1,  composed with a $(y,w)$-linear change of variable of finite rank. Note that a regular $g\in\calX$ is in $\BB_\calE$ if  it satisfies \eqref{beby2}.
 
On ${\calE}^{(1)}$ we give a slightly stronger definition of tame vector field.
\begin{defi}\label{adjointtt}
 We say that a $C^{3}-$tame vector field $F\in \calE^{(1)}$ is ``adjoint-tame'' if
 there exists a choice of tameness constants $C_{\vec{v},p}(F)$
such that, for any $\Phi$  generated by $g\in \BB_{\calE}$ and for any $h$ as in Definition \ref{tame}, the adjoint\footnote{We recall that, given a linear operator $A:X\to Y  $ its adjoint is  $A^*:Y^* \to X^*$. Our condition implies that $(d_{\mathtt U}F(\Phi))^{*}$ is bounded from $Y_1\to X_1$, with $Y_1\subset Y^*$ and $X_1\subset X^*$ 
this is hence a  much stronger condition.} of $d_{\mathtt U}F(\Phi)$ is tame
and satisfies the bounds. 
Setting 
$$X^{p}:=H^{p+\nu}(\TTT^{d}_{s};\CCC^{d_1}\times\ell_{-a,-\gotp_0-\nu})\cap 
H^{\gotp_0+\nu}(\TTT^{d}_{s};\CCC^{d_1}\times\ell_{-a,p-\gotp_1-\gotp_0-\nu})$$ and 
$$Y^{p}:=
H^{p}(\TTT^{d}_{s};V_{-a,-\gotp_0})\cap 
H^{\gotp_0}(\TTT^{d}_{s};V_{-a,p-\gotp_1-\gotp_0})$$
for $p\geq\gotp_1$,
one has (see formula \eqref{ancoralip22})
 \begin{equation}\label{carota}
  \begin{array}{crcl} 
    (\text{T}_1)^* & 
    \|(d_{\mathtt U}F(\Phi))^{*}[h]\|_{\g,\calO,X^{p}} &\leq & 
     \big(C_{\vec{v},p}(F)+C_{\vec{v},\gotp_0}(F)|g|_{\vec{v}',{p+\nu}} \big)
   \|h\|_{\g,\calO,Y^{\gotp_0}} 
    +C_{\vec{v},\gotp_0}(F)\|{h}\|_{\g,\calO,Y^{p}}.
   \end{array}
   \end{equation}
\end{defi}

We have the following Lemmata.
\begin{lemma}\label{scerpa}
Consider a regular vector field $f\in \calA_{\vec{v},p}$.  Then $f$ satisifes \eqref{carota} with  $C_{\vec v,p}(f)=|f|_{\vec v,p}$. Moreover there exists a $\mathtt{c}$  (depending  at most on $\gotp_0$ and 
on the dimensions $d,d_1$) such that for any tameness constant satisfying \eqref{carota} one has
\begin{equation}\label{tameconst333}
|f|_{\vec{v},p} \le \mathtt{c} \, C_{\vec{v},p}(f)
\end{equation}
for any $p\geq \gotp_1$.
\end{lemma}
\begin{proof}
We only sketch the proof of the Lemma when $\Phi$ is the identity: the general case is essentially identical due to the simple structure of $\calA_{\vec{v},p}$ and $\BB$. The only non-trivial components are $f^{(y_i,w)}\cdot w$. 	 The adjoint of the differential is then the map $\lambda\to f^{(y_i,w)} \lambda$ with $\lambda\in H^p(\TTT^d_s)$. The result follows by the definition of $|\cdot|_{\vec v,p}$.
\end{proof}
\begin{lemma}\label{adjadjadj}
The adjoint-tame vector fields are closed with respect to close to identity  changes of variables $\Phi=\uno+f$ generated by $\psi\in\BB_{\calE}$. 
In particular, setting $F_{+}=\Phi_{*}F$,
 one has that the tameness constants in \eqref{dafare} satisfy condition $(T1)^*$.


\end{lemma}

\begin{proof}
Fix  a  vector field $F : \TTT^d_s\times D_{a,p+\nu}(r) \times \calO \to V_{a,p}$
which is $C^3$--tame,
 By Remark \ref{azz}, we know that if $|f|_{\vec{v},\gotp_1}=\mathtt{c}\rho$ with $\mathtt{c}$ small enough  then there exists
$\Phi^{-1}=\uno+\tilde{f}$ with $|\tilde{f}|_{\vec{v},p}\sim |f|_{\vec{v},p}\sim |\psi|_{\vec{v},p}$ and 
one has, by Lemma \ref{conj}
$F_+:=\Phi_* F: \TTT^d_{s- 2\rho s_0}\times D_{a,p+\nu}(r-2\rho r_0)\times \calO  \to  V_{a- 2\rho a_0,p}$
is $C^3$--tame up to order $q-\nu-1$, with scale of constants
 \begin{equation}\label{dafare222}
 C_{\vec{v}_{2},p}(F_+)\leq (1+\rho)\Big(C_{\vec{v},p}(F)+C_{\vec{v},\gotp_{0}}(F)
 C_{\vec{v}_{1},p+\nu+1}(f)\Big),
 \end{equation}
 where $\vec{v}:=(\la,\calO,s,a,r)$, $\vec{v}_{1}:=(\la,\calO,s-\rho s_0,a-\db a_0, r-\rho r_0)$ and 
 $\vec{v}_{2}:=(\la,\calO,s-2\rho s_0,a-2\rho a_0, r-2\rho r_0)$.
 Now by Lemma \ref{scerpa} since  $f,\tilde{f}$ are ``regular'' vector fields, 
they are also ``adjoint-tame''. 

Consider a transformation $\Gamma$ generated by $g\in\BB_{\calE}$. We need to check that $(d_{\mathtt U}F_{+}(\Gamma))^{*}[h]$
satisfies \eqref{carota} with $C_{\vec{v},p}(F)\rightsquigarrow C_{\vec{v}_2,p}(F_{+})$.

 One can write $F_{+}=F\circ{\Phi^{-1}}+df(\Phi)[F\circ\Phi^{-1}]$ and study the two summands separately.  
First 
note that $\Phi\circ\Gamma= \Psi=\uno+k$ with $k\in\BB$ such that
$$
|k|_{\vec{v}_1,p}\leq |\psi|_{\vec{v},p}|g|_{\vec{v},\gotp_1}+|\psi|_{\vec{v},\gotp_1}|g|_{\vec{v},p}.
$$
In particular $k$ is ``adjoint-tame'' since it is regular.
On the other hand if one has two linear (in $y,w$) vector fields $A$ and $B$ which are ``adjoint-tame'', then 
$AB$ is clearly ``adjoint-tame''. Now one can write 
$(F_{+}\circ{\Gamma})=F(\Phi^{-1}\circ{\Gamma})+df(\Phi\circ{\Gamma})[F(\Phi^{-1}\circ\Gamma)]$.
Let us study for instance the first summand. 
Essentially \eqref{carota} follows by the chain rule, the property $(T1)^{*}$ on $F$, and the tame estimates on the differential of 
$k$ and on its adjoint.

One has that $d_{\mathtt{U}}F(\Phi^{-1}\circ\Gamma)=d_{\mathtt{U}}F(\Phi^{-1}\circ\Gamma)d_{\mathtt{U}}(\Phi^{-1}\circ\Gamma)$. The estimate
\eqref{carota} follows by $(T1)^{*}$ on the field $F$ and the  
estimates on $k$.
%
%
%
%
%
%
%
%
Jus as an example consider the term from the differential of $df [F\circ\Phi^{-1} ]$
is, for $i=1,\ldots,d_1$,  the operator $f^{(y_{i},w)}\cdot d_{w}F^{(w)}(\theta,y,w)[\cdot]$. For 
$h\in H^{p}(\TTT^{d}_{s};\CCC)$ we have the estimate
\begin{equation}\label{carota3}
\begin{aligned}
&\|(d_{w}F^{(w)}(\Psi))^{*}f^{(y_{i},w)} h\|_{\g,\calO,
H^{p+\nu}(\TTT^{d}_{s};\ell_{-a,-\gotp_0-\nu})\cap 
H^{\gotp_0+\nu}(\TTT^{d}_{s};\ell_{-a,p-\gotp_1-\gotp_0-\nu})}\stackrel{(T1)^{*}}{\leq} \\
&\leq 
(C_{\vec{v},p}(F)+C_{\vec{v},\gotp_0}(F)|k|_{\vec{v},p+\nu})\|f^{(y_{i},w)}h\|_{H^{\gotp_0}(\TTT^{d}_{s};\CCC)}+
C_{\vec{v},\gotp_0}(F)(1+|k|_{\vec{v},\gotp_1})\|f^{(y_{i},w)}h\|_{H^{p}(\TTT^{d}_{s};\CCC)}\\
&\leq
(C_{\vec{v},p}(F)+C_{\vec{v},\gotp_0}(F)|k|_{\vec{v},p+\nu})|\psi|_{\vec{v},\gotp_0}
\|h\|_{H^{\gotp_0}(\TTT^{d}_{s};\CCC)}+\\
&+
C_{\vec{v},\gotp_0}(F)(1+|k|_{\vec{v},\gotp_1})(
\|h\|_{H^{p}(\TTT^{d}_{s};\CCC)}|\psi|_{\vec{v},\gotp_1}
+\|h\|_{H^{\gotp_0}(\TTT^{d}_{s};\CCC)}|\psi|_{\vec{v},p}
)\\
&\leq  C_{\vec{v},\gotp_0}(F)(1+2|\psi|_{\vec{v},\gotp_1}|g|_{\vec{v},\gotp_1})|\psi|_{\vec{v},\gotp_1}\|h\|_{H^{p}(\TTT^{d}_{s};\CCC)}\\
&+
\|h\|_{H^{\gotp_0}(\TTT^{d}_{s};\CCC)}
\Big[C_{\vec{v},p}(F)|\psi|_{\vec{v},\gotp_1}+C_{\vec{v},\gotp_0}(F)|\psi|_{\vec{v},p+\nu}
+C_{\vec{v},\gotp_0}(F)|\psi|_{\vec{v},\gotp_1}|g|_{\vec{v},p+\nu}
\Big].
\end{aligned}
\end{equation}

\end{proof}

 \begin{proposition}\label{uffffa3}
Let $\gamma> 0$, $K\ge K_0$,
consider a compact set $\calO \subset \calO_0$ and set $\vec{v}=(\g,\calO,s,a,r)$ and $\vec{v}^{\text{\tiny 0}}=(\g,\calO_0,s,a,r)$.
Consider a  vector field $F\in \calW_{\vec v^{\text{\tiny 0}},p}\cap \calE^{(1)}$ of the form
$$
F= N_0+G: \calO_0\times    D_{a,p+\nu}(r)\times\TTT^{d}_{s}\to V_{a,p}\,, 
$$
where $\calE^{(1)}$ is defined  \eqref{spazorev} and $N_0$ is defined in \eqref{enneo} with $\Omega^{(0)}$ self-adjoint.
Assume that $F$ is  $C^{ 3}$-tame up to order $q=\gotp_2+2$
and adjoint-tame. 
Assume that $\g\sim {\rm diam} \calO_0$  and set
%
 \begin{equation}\label{sizes2b100}
\Gamma_{p}:=\g^{-1}C_{\vec{v},p}(G), \quad 
 \Theta_{p}:=\g^{-1}C_{\vec{v},p}(\Pi_{\NN}^\perp G).
 \end{equation}

If $\calO$ satisfies the Mel'nikov conditions of Definition \ref{pippopuffo3} for $(F,K,\vec{v}^{\text{\tiny 0}})$  then 
$\calO$ satisfies the Homological equation of Definition \ref{pippopuffo2} provided that we fix 
parameters $\mu$ and $\mathtt{t}$ as in \eqref{migo}.
Moreover $\Phi_{*}F$ (defined in \eqref{cribbio420}) is adjoint-tame.
\end{proposition}

 \begin{proof}
 We wish to apply  Proposition \eqref{uffffa} in order to prove that items $(1)$  and 2.a,b,c  of Definition \ref{pippopuffo2} are satisfied. In order to do so we need to prove \eqref{vino2}.
 One has that \eqref{vino2} follows directly by using the adjoint-tameness estimates 
 on $\Pi_{\RR}G$ and $\Pi_{\NN}G$ as done in the proof of Lemma \ref{adjadjadj}.
Regarding item $2.d$. To prove estimate \eqref{cribbio420} one can use the adjoint-tameness of $G$ to get the bound
 for the term $\Phi_{*}G$. The term $\Phi_{*}N_0$ must be treated as
 done in \eqref{uffa3} of Proposition \ref{kamstep}  by using the norm $|\cdot|_{\vec{v},p}$
 instead of the tameness constant. This can be done using the \eqref{vino2} and the fact that
 $g$ in Definition \ref{pippopuffo2} satisfies the homological equation \eqref{cribbio42}.
 The adjoint-tameness of $\Phi_{*}F$ follows by Lemma \ref{adjadjadj} since $g\in \BB_{\calE}$
 by definition.
 \end{proof}
 As in the previous examples, under the same assumptions as in Proposition \ref{uffffa3} and of course assuming \eqref{sizes}, we can apply Theorem \ref{thm:kambis};
 by Proposition \ref{uffffa3} the set $\calO_{\io}$ of Theorem \ref{thm:kambis} contains the intersection over $n$ of the
 sets in which the Mel'nikov conditions are satisfied for $(F_n,K_n,\vec{v}_n^{\text{\tiny 0}})$, therefore we now
 analyze the Mel'nikov conditions \eqref{pippopuffo3} in this case. \\
The operator $\Pi_\calX{\rm ad}(\Pi_\NN F_n)$decomposes as follows: we get the operator $(\omega^{(n)}\cdot \partial_{\theta})\,\uno$ on the blocks  $\calX_1,\calX_4$  while on the blocks $\calX_2,\calX_3$,  we get
\begin{equation}\label{merdina3}
\big( \omega^{(n)}\cdot \partial_\theta\big)\,\uno  - F^{(w,w)}_n(\theta)\,,\quad  \big( \omega^{(n)}\cdot \partial_\theta\big)\,\uno  + \Big( F^{(w,w)}(\theta)\Big)^*
  \end{equation}
respectively. As in the previous  example the Diophantine condition on $\omega^{(n)}$ is used in order to solve the homological equations on $\calX_1,\calX_4$.
In conclusion we have proved the following Theorem, which is the analogous of Theorem \ref{esempappli} in the reversible case, simply requiring less regularity
for $F$, adjoint-tameness and of course the reversible structure instead of the Hamiltonian one.

\begin{theo}\label{esempappli2}
	Consider a  vector field $F\in \calW_{\vec v^{\text{\tiny 0}},p}\cap \calE^{(1)}$ of the form
	$$
	F= N_0+G: \calO_0\times    D_{a,p+\nu}(r)\times\TTT^{d}_{s}\to V_{a,p}\,, 
	$$
	where $\calE^{(1)}$ is defined  \eqref{spazorev} and $N_0$ is defined in \eqref{enneo} with $\Omega^{(0)}$ self-adjoint.
	Assume that $F$ is  $C^{3}$-tame up to order $q=\gotp_2+2$ and adjoint-tame. 
	Fix $\g>0$ such that $\g\sim {\rm diam} \calO_0$ and assume that $G$ satisfies the smallness conditions \eqref{sizes}
	of Theorem \ref{thm:kambis}. Then there exists an invariant torus 
	for $F$ provided that $\x$ belong to the set $\calO_{\infty}$ of Theorem \ref{thm:kambis}.
	Finally $\calO_{\infty}$ contains  $\bigcap_{n}\mathcal{C}_{n}$
	where $\mathcal{C}_{n}$ is the set of $\x$ such that $\oo_{n}$ is $(\g,\tau)-$diophantine
	and the matrices  \eqref{merdina3} are approximatively invertible
	with tame bounds like \eqref{buoni} and \eqref{cribbio4}.
\end{theo}

As in the previous example one could also apply our KAM scheme with the decomposition \eqref{sottoMIN} but  using as changes of variables operators $\calL_n$ which block diagonalize \eqref{merdina1} into \eqref{merdina3}.
\subsection{Example 4: KAM  with reducibility.}\label{reduco}
Up to now we have just reduced our problem to the inversion of \eqref{merdina3Ham} or \eqref{merdina3}. Inverting such operators is not trivial and requires some subtle multiscale arguments as discussed in Example 2, subsection \ref{exNLS} (see also \cite{BB3}).
\\
Clearly a major simplification would appear if we were able to diagonalize \eqref{merdina3}-\eqref{merdina3Ham}. This is indeed the classical KAM approach 
(see \cite{Moser-Pisa-66}, \cite{Po2}) but it requires much stronger non resonance conditions, i.e. the second Mel'nikov conditions. 
How to use reducibility in order to prove bounds of the form \eqref{merdina3}-\eqref{merdina3Ham} for nonlinear PDEs on a circle has been discussed in various papers, see \cite{BBM1}, \cite{BBM}, \cite{FP}. Here we briefly recall the main point in the simplest possible context. We consider a Hamiltonian case and assume to work with the decomposition \eqref{sottoHAM}. 
Let suppose that in Definition \ref{struttHAM} we have  $$h_{a,p}= \oplus_{j\in \NNN}h_j, \quad \ell_{a,p}=\oplus_{j\in \NNN}\ell_j\,,\quad \ell_j= h_j\times h_j\,,\quad w_j=(z_j^+,z_j^-)$$
 with $h_j$ finite dimensional subspaces, for simplicity suppose them one-dimensional.
  Then  one may introduce finite dimensional monomial
  subspaces\footnote{Of course if $\ell_{a,p}$ is infinite dimensional then it is not true 
  that for instance
$
\bigoplus_{i,j} \VV^{(w_i,w_j)} = \VV^{(w,w)}\,.
$}
$ \VV^{(\tv, z^{\s_1}_{j_{1}},\cdots, z^{\s_k}_{j_{k}})}$, with $\s_i=\pm 1$.

For a  linear operator $A(\theta)\in \calL(\ell_{a,p},\ell_{a,p})$ one considers its block decomposition $\{A_j^i\}_{i,j\in \ZZZ^r}$ and  the off-diagonal
decay norm
\begin{equation}\label{decado}
(|A|_{s,a,p}^{\rm dec})^2:= \sum_{h=(h_1,h_2)\in \NNN\times\ZZZ^{d} }
\langle h \rangle^{2p}e^{2(a|h_1|+s|h_2|)}\sup_{j\in\NNN}  |A_{j+h_1}^{j}(h_2)|^2 
\end{equation}
where $|A_j^i|$ is the operator norm on $\calL(\ell_i, \ell_j)$ . Then we consider the corresponding weighted Lipschitz norm which we denote
$|\cdot|^{\rm dec}_{\vec v,p}$.
This gives a special role to diagonal $\theta$-independent vector fields, so we define
%
%
%
%
%
$$
\NN_0:=\av{ \VV^{(\theta,0)}}\bigoplus_{j,\s} \langle\VV^{(z^\s_j,z^\s_j)}\rangle\cap \VV^{(w,w)}\,,$$

Then one can choose $\calE$ as 
\begin{equation}\label{mannaggiac}
\calE^{(2)}_{\rm Ham}:=\{F\in \calE^{(0)}_{\rm Ham}:\;\quad d_w F^{(w)}(\theta,y,u) =D+ M \}
\end{equation}
 with $D$ diagonal  and $M$ a bounded operator with finite $| \cdot|^{\rm dec}_{\vec v,p}$ norm.
We are in the framework of \cite{Po2} or \cite{BB}, but we are {\bf not} requiring analiticity, therefore we need some tameness properties, which we ensure by choosing the norm $|\cdot|^{\rm dec}_{\vec v,p}$. Let us  briefly recall the notations.  We consider the Hamiltonian vector field
$$
F_0= \oo_{0}\cdot\del_{\theta}+\ii\sum_{j}\Omega_{j}^{(0)}z^+_j\del_{z^+_j} - \ii \sum_{j}\Omega_{j}^{(0)}z^-_j\del_{z^-_j} + G_0(\xi,y,\theta,w)
$$
with the following assumptions.
\begin{enumerate}
	\item [A.]  {\it Non-degeneracy}. We require that for all $j\neq k $, $\ell\in \TTT^d$
	\begin{equation}
	\label{nondeg}
	\big|\{\xi\in \calO_0: \oo_0\cdot l+\Omega^{(0)}_{j} \pm \Omega_{k}^{(0)}=0 \} \big| =0 \,,\qquad  \Omega^{(0)}_{j} \pm \Omega_{k}^{(0)}\neq 0\quad \forall \xi\in \calO_0\,
	\end{equation}
	\item[B.] {\it Frequency Asymptotics}.  We assume that $\xi \to \omega^{(0)}(\xi)$ is a lipeomorphism and  $$|\omega(\xi)|,|\omega(\xi)|^{\rm Lip},|\Omega^{(0)}_j(\xi)-j^\nu|, | \Omega^{(0)}_j(\xi)|^{\rm Lip}< M\,,\quad | \xi(\omega)|^{\rm Lip}\le L\,,\quad  \forall \xi\in \calO_0$$ for some $\nu>1$.
	\item[C.] {\it Regularity}. We require $G_0\in \calE^{(2)}_{\rm Ham}$ with $D=0$, more precisely $G$ is $C^4$ tame  bounded   Hamiltonian vector field with  $\Pi_{\NN_0}G=0$ well defined and Lipschitz for $\xi\in \calO_0$ a compact set of positive measure. 
	\item[D.] {\it Smallness}. In Constraint \ref{sceltapar} fix  $$\al=0\,,\quad \chi=3/2\,,\quad \ka_1= 2\ka_0+1, \ka_3= 3\ka_0+1\,,\quad \ka_2= 4 \ka_0+1\,,$$ 
	$$
	 \eta= \mu+ 2\ka_2 +3\,,\quad \Delta \gotp = 9 \ka_0 +3,\quad \gotp_1= \frac{d+2}{2}+\nu +1 $$ this leaves as only parameter $\mu$. Suppose that
	$\tG_0\sim\tR_0\sim 1$,  fix $\e_0$ small  and set
	$K_0= \e_0^{-1/5\ka_0}$ so that  the smallness conditions \eqref{expexpexp},\eqref{1s1}--\eqref{6s2} are satisfied.  
	
	We assume finally that 
	$$
	\calP_0:= \Pi_{\NN}F_{0}-\Pi_{\NN_0}F_{0}= \Pi_{\NN}G_{n}-\Pi_{\NN_0}G_{0}= P^{(0)}(\theta) w \del_w +\frac{\ii}2 J w \cdot\sum_i \partial_{\theta_i} P^{(0)}(\theta) w\del_{y_i}
	$$
	is small i.e.
	$$
\g_0^{-1}| P^{(0)}|^{\rm dec}_{\vec v_0,\gotp_1}\le \e_0
	$$
\end{enumerate}

We are in the context of Subsection \ref{exHam}, so we have Proposition \ref{uffffa2} with $\tb=3,\mathtt t=1$, and for convenience let us set 
\begin{equation}\label{fisso}
\mu= 4 (\mu_1+\nu)+5\,,\quad \mu_1=  2\tau+2\,.
\end{equation}
\begin{theo}\label{pesciolone}
Fix 	$\tau> d+1+ \frac{2}{\nu-1}$  and $\e_0(LM)^{\tau+1}\ll 1$. Let  $F_0$ be a $C^4$ tame vector field  up to order $q= \gotp_2+2$ satisfying assumptions A. to D.  Then  for $\g_0$ small enough there exists  positive measure Cantor like set $\calO_{\infty}(\g_0)\subset \calO_0$ of asymptotically full Lebesgue measure  as $\g\to 0$ and  a symplectic close to identity change of variables $\HH_\infty$ such that  
$$
\Pi_\calX (\HH_\infty)_{*} F_0=0\,,\quad (\Pi_\NN-\Pi_{\NN_0})(\HH_\infty)_{*} F_0=0,\,,\quad \forall \xi\in \calO_{\infty}(\g_0) 
$$
so that $F_0$ has a reducible KAM torus.
\end{theo}
We apply Theorem \ref{thm:kambis} to $F_0$ and we aim to show that we can produce a non-empty set $\calO_\infty$ by choosing the $\calL_n$ appropriately.
By definition,  at each step $n$, set
\begin{align}\label{claKAM}
\Pi_{\NN_0}F_{n} &=\DD_n=\oo_{n}\cdot\del_{\theta}+\ii\sum_{j}\Omega_{j}^{(n)}z^+_j\del_{z^+_j} - \ii \sum_{j}\Omega_{j}^{(n)}z^-_j\del_{z^-_j}\,,\\ \Pi_{\NN}F_{n}-\Pi_{\NN_0}F_{n} &=\calP_n= P^{(n)}(\theta) w \del_w +\frac{\ii}2 J w \cdot\sum_i \partial_{\theta_i} P^{(n)}(\theta) w\del_{y_i}.
\end{align}

\begin{lemma}[KAM reduction step]\label{redKam}
Fix $\vec v_n=(\g_n,\calO_n,s_n,a_n,r_n)$, $\vec v^0_n=(\g_n,\calO_0,s_n,a_n,r_n)$ and $K_n\gg 1$ as in \eqref{numeretti}. Given a  Hamiltonian vector field  $\DD_n+\calP_n$  as in \eqref{claKAM} with $$ \rho_n(M\g_n)^{-1} K_n^{2\tau+1} |P^{(n)}|^{\rm dec}_{\vec v_n,\gotp_1}\ll 1$$ 
 there exists a symplectic change of variables	
$\calL_{n+1}$,  which conjugates
 $$(\calL_{n+1})_{*}(\DD_n+\calP_n):=\widehat\DD_n + \widehat\calP_n 
$$ (here $\widehat\DD_n$ is the projection of the conjugate vector field onto $\NN_0$ ) so that

\begin{enumerate}
	\item $\calL_{n+1}$ is  the time one flow of the  Hamiltonian vector field $$\mathcal S_{n+1}:= S^{(n+1)}(\theta) w \del_w +\frac{\ii}2 J w \cdot\sum_i \partial_{\theta_i} S^{(n+1)}(\theta) w\del_{y_i} \,,\quad {\rm with}\;
|S^{(n+1)}|^{\rm dec}_{\vec v^0_n, p}\le M \g_n^{-1} K_n^{2\tau+1} |P^{(n)}|_{\vec v_n,p}^{\rm dec}.$$
\item Set
\begin{equation}\label{secmel3}
\widehat{\calO}_n:= \{\xi\in\calO_n:\; 
|\oo^{(n)}\cdot l+\Omega^{(n)}_{j} \pm \Omega_{k}^{(n)}  |\geq \frac{M\g_n}{K_n^{\tau}}, \quad |\ell|\le K_n
\}.
\end{equation}   For all $\xi\in \widehat{\calO}_n$,  $\mathcal S_{n+1}$ solves the Homological equation
$$
[\mathcal S_{n+1},\DD_n] + \Pi_{K_n}\calP_n=0 \,,\quad (\Pi_K P)_{ij}= \left\{\begin{aligned}
P_{ij} & \quad {\rm if}\; |i-j|\le K \\ 0 & \quad {\rm otherwise}
\end{aligned}\right.
$$
\item Setting $\widehat v_n:=(\g_n,\widehat\calO_n,s_n,a_n,r_n) $ we have the bounds
\begin{align*}
&|\widehat P^{(n)}|^{\rm dec}_{\widehat v_n,\gotp_1}\le  |S^{(n+1)}|^{\rm dec}_{\vec v, \gotp_1} |P^{(n)}|^{\rm dec}_{\vec v, \gotp_1} + K_n^{-(\gotp_2-\gotp_1)}| P^{(n)}|^{\rm dec}_{\vec v_n,\gotp_2}\,,\\ 
&| \widehat P^{(n)}|^{\rm dec}_{\widehat v_n,\gotp_2} \le  | P^{(n)}|^{\rm dec}_{\vec v_n,\gotp_2} +{\rm const} |S^{(n+1)}|^{\rm dec}_{\vec v, \gotp_2} | P^{(n)}|^{\rm dec}_{\vec v_n,\gotp_2} \,,
\end{align*}
\end{enumerate}
\end{lemma}
Now we may  apply this reduction step at each step in our iteration in Theorem \ref{thm:kambis}, provided that we show  that the $\calL_n$ are compatible changes of variables. This we prove by induction. In fact we recursively obtain that
$$
\gamma_n^{-1}|P^{(n)}|_{\vec v, \gotp_1}^{\rm dec}  \le K_0^{\kappa_1}\e_0 K_n^{-\kappa_1} \,,\quad \gamma_n^{-1}|P^{(n)}|_{\vec v, \gotp_2}^{\rm dec}
 \le  2\mathtt G_0 K_n^{\kappa_1}
$$
so that the $\calL_n$ are compatible  since  $\kappa_3> \kappa_1+2\tau+1$. Now we can choose a set $\calO_{n+1}$ which satisfies the  Melnikov conditions \eqref{merdina3Ham} at step $n$ for $\widehat F_n= (\calL_n)_{*} F_n$.  Since $\Pi_\NN \widehat F_n= (\calL_n)_{*} \Pi_\NN  F_n $, this amounts to finding an approximate inverse for ad$(\widehat \DD_n+\widehat\calP_n)$  which satisfies  \eqref{buoni} with $\mu_1$ fixed in \eqref{fisso}. 
Now, finding a partial inverse of \eqref{merdina3Ham} is equivalent to finding an inverse to  ${\rm ad}\,\widehat\DD_n$, since ${\rm ad}\,\widehat\calP_n$ can be put inside the remainder, see formula \eqref{cribbio42}.
In turn the invertibility of ${\rm ad}\,\widehat\DD_n$ on $\calX_2$ is ensured by requiring 
 lower bounds on the eigenvalues, i.e. choosing at each step 
\begin{equation}\label{secmel300}
{\calO}_{n+1}:= \{\xi\in\widehat\calO_n:\; 
|\widehat\oo_n\cdot l+\widehat\Omega^{(n)}_{j}|\geq \frac{\g_n M}{K_n^{\tau}}\,, \quad |\ell|\le K_n
\}
\end{equation} 
in Theorem \ref{thm:kambis}. One can easily check that, troughout the algorithm 
 $\xi \to \omega^{(n)}(\xi)$ and $\xi \to \widehat\omega^{(n)}(\xi)$ are lipeomorphisms and  $$|\omega^{(n)}-\omega_0|+\gamma_0|\omega^{(n)}-\omega_0|^{\rm Lip},|\Omega^{(n)}_j-\Omega^{(0)}_j|+\g_0|\Omega^{(n)}_j-\Omega^{(0)}_j|^{\rm Lip}< \g_0 \e_0 \,,\quad | \xi_n(\omega)|^{\rm Lip}\le 2L\,,\quad \mbox{in}\; \calO_n,$$  the same for the corresponding $\widehat{\cdot}$ quantities.
\begin{lemma}
For $\tau> d+1+ \frac{2}{\nu-1}$  and $\e_0(LM)^{\tau+1}\ll 1$ we have that $|\calO_0\setminus \cap_n \calO_n| \to 0$  as $\g_0\to 0$.  
\end{lemma}
\begin{proof}
This is proved in \cite{Po2} Corollary C.
\end{proof}

The main point in this approach is that at each step, while constructing our approximate solutions by inverting \eqref{merdina3Ham}, we apply  a change of variables which approximately diagonalizes the linearized operator up to a negligible remainder. Condition \eqref{secmel3}, ensures that the sequence of linear changes of variables converges.
 In fact this last approach could be made slightly more flexible, indeed in our construction we have used the norm \eqref{decado} on the component $\VV^{(w,w)}\oplus \VV^{(y,w,w)}$ in order to perform the reduction. This imposes some unnecessary conditions on the changes of variables, since the only unavoidable request on the $\calL_n$ is that they preserve $\calE$ and do not disrupt the bounds. A typical  example is a change of variables of the form $\calL w(x)= w(x+ a(\theta,x))$. One does not expect such change of variables to have finite \eqref{decado} norm but, if $a$ is chosen appropriately, it can be a compatible change of variables.

\subsection{Application to the NLS.}\label{exNLS}
Let us specify to a PDE context;  typical examples are $\ell_{a,p}= H^{p}(\TTT^d_a)$ or $\ell_{p}= \ell_{a=0,p}= H^{p}(\mathbb G)$ with $\mathbb G$
 a compact Lie group or homogeneous manifold.  Then we may suppose that $\Omega^{(0)}$ in \eqref{enneo} is diagonal w.r.t. the Fourier 
 basis (and self-adjoint) while $G$ is a composition operator w.r.t. the $w$ variable. For simplicity of the exposition, let us restrict ourself to semi-linear PDEs 
 with no derivative in the nonlinearity, 
so that $G(w(x))= f(w(x))$ with $f $ a $C^k$ map $\CCC^2\to \CCC^2$.  More precisely we choose
\begin{equation}\label{labestia}
\calE=\calE^{(1)}_{\rm Ham}:=\Big\{F \in \calE^{(0)}_{\rm Ham}:\quad d_w F^{(w)}(\theta,y,u) =D+ M+ R\,,\quad u\in D_{a,p}(r)\Big\}
\end{equation}
where $D$ is a diagonal operator $M$ is a multiplication operator while $R$ is finite rank. Since we apply symplectic changes of variables which are the identity plus a traslation plus a  finite rank operator, 
$\calE^{(1)}_{\rm Ham}$ is preserved throughout our algorithm.

As an example consider the NLS equation on a simple {\em compact} Lie group $\mathbb G$:
\begin{equation}\label{NLSfin}
\ii  \partial_t u +\Delta u + M_\xi u = \epsilon g(|u|^2) u \,,\quad 
\end{equation}
here $g(y)\in C^q(\RRR,\RRR)$ with $q$ large, while $M_\xi$ is a Fourier multiplier 
$$
M_\xi \phi_i(x)= \xi_i \phi_i(x) \,,\quad i=1,\cdots,d
$$
where $\phi_n(x)$ are distinct eigenfunctions of the Laplace-Beltrami operator. We introduce the variables $\theta,y,w=(z,\bar z)$ by writing (for $I_n>0$ )
\begin{equation}\label{nonvoglioripetere}
u(x)= \sum_{j=1}^d \sqrt{I_j + y_j}e^{\ii \theta_j}\phi_j(x) +z(x)\,,
\end{equation}
where $z(x)$ belongs to the orthogonal complement of Span$(\phi_i(x))_{i=1}^d$ which by definition is $\ell_p=\ell_{0,p}$. As norm we choose the one induced on $\ell_p$ by $H^p(\mathbb G)$.
  This change of variables is symplectic and one obtains a Hamiltonian vector field $F=N_0+G$ of the form \eqref{adattato}, where $\omega^{(0)}_i(\xi)= \lambda_i + \xi_i$, with  $\lambda_i$ being the eigenvalue of the Laplace-Beltrami operator associated to $\phi_i$. Correspondingly $\Lambda^{(0)}$ is the Laplace-Beltrami operator restricted to the complement of Span$(\phi_n(x))_{n=1}^d$. 
Since the non-linearity is a composition operator on $H^p$, then classical results ensure the $C^k$ tameness of the vector field for all $k$.
Moreover the restriction of a multiplication operator to $\ell_{p}$ is a multiplication operator up to a finite rank operator ($\ell_{p}$ is the complement of a finite dimensional subspace), then  
 $F\in \calE^{(1)}_{\rm Ham}$. 
Now we fix $\g_0>1/2$ and take $\tG_0,\tR_0,\e_0\sim \epsilon$ in \eqref{sceltapar}. In this way the NLS equations satisfies all the hypotheses of Theorem \ref{esempappli}
and hence we deduce the existence of an invariant torus in the set $\calO_\io$.

Let us now discuss the measure estimates for the sets $\mathcal C_n$ in this setting.
It is easily seen that $\omega^{(n)}_j= \lambda_j+\xi_j +O(\epsilon)$ so imposing the diophantine conditions on such sequence is simple. As one could expect the key point is the inversion of the operator in \eqref{merdina3Ham}. In turn, since the operator is triangular, this amounts to inverting the diagonal, i.e. the operator
\begin{equation}\label{merdafin}
\gotL_n= (\omega^{(n)}\cdot \partial_\theta\big)\,\uno  -  F^{(w,w)}_n(\theta).
\end{equation}
 acting on $\VV^{(w,0)}$.
 We first remark that $F^{(w,w)}_n(\theta)$ is the linearized of $F^{(w)}$ at an approximate solution up to a finite rank term. This follows from the fact that our changes of variables are traslations plus finite rank.
 From this we deduce that $F^{(w,w)}_n(\theta)$ is a multiplication operator plus a finite rank one.
Now in order to prove that the estimates \eqref{buoni} and \eqref{cribbio4} hold for \eqref{merdafin} for a large set of $\xi$ one can use a multiscale theorem, such as the one in \cite{B3},\cite{BB2} or \cite{BCP}.   Indeed one may verify that $\gotL_n$ in \eqref{merdafin} fits all the hypotheses of Proposition 5.8 in \cite{BCP} so that the tame estimates on the inverse follow by conditions on the eigenvalues. The measure estimates follow by eigenvalue variation (again just as in \cite{BCP}).
 
 If $\mathbb G=\TTT^k$,  this general strategy is carried on in full details in \cite{BB3}, in the more complicated case of a multiplication potential; see also the application to the wave equation \cite{BeBoW}.
 Since the authors follow the Nash-Moser approach they never apply the symplectic change of variables which block-diagonalize, but 
 only deduce its existence by the Hamiltonian structure.
 
 If $\mathbb G=\TTT$ and we look for odd solutions, then the hypoteses A. to D. of subsection \ref{reduco}
 are satisfied and using the same change of variables as in \eqref{nonvoglioripetere} and reasoning as
 above, one can see that $F\in\calE^{(2)}_{\rm Ham}$, see \eqref{mannaggiac}, and therefore one can apply Theorem \ref{pesciolone} and obtain that
 the invariant torus is also reducible; note that in this case the procedure is complete in the sense that the set $\calO_\io$ has positive measure.
 \begin{theorem}
 	The NLS equation \eqref{NLSfin} admits reducible and linearly stable quasi-periodic solutions for all $\e$ small enough and for all $\xi$ in a positive measure set.
 \end{theorem}
Note that  one could avoid adding the Fourier multiplier in \eqref{NLSfin} and obtain the parameters by using Birkhoff Normal form (this is done for instance in \cite{PP4}).
\\
In our application, we have not considered a reversible case: this is due to the fact that the natural structure for the semilinear NLS
is the Hamiltonian one. On the other hand, if one considers DNLS (Derivative NLS) there are interesting examples which are not Hamiltonian
but instead they are reversible; see for instance \cite{ZGY,FP}. We believe (see \cite{Feo,FP2}) that our result can be fruitfully applied also to the fully nonlinear
autonomous case.

\subsection{Some comments}
%

The examples 1-4 show that our KAM approach interpolates from the Nash-Moser scheme to the KAM. 
Each different choice of decomposition depends on the specific application one is studying. 
 Note that 
we could always choose the simplest decomposition \eqref{sottoMIN} (or \eqref{sottoHAM} in the Hamiltonian setting) 
and achieve the block-decoupling of the linearized operators through the compatible change of variables $\calL_{n}$.

A final remark is in order. In the PDEs setting of subsection \ref{exNLS}, applying the changes of variables of Lemma \ref{redKam} implies some loss of information.
Indeed it is easily to see that changes of variables do not preserve $\calE^{(1)}_{\rm Ham}$ unless we can show that the matrices in $S^{(n)}$ are T\"oplitz (up to finite rank). 

Unfortunately in most applications  this is not the case: indeed in classic KAM scheme the PDEs structure is essentially ignored and one works in $\calE^{(2)}_{\rm Ham}$. This has been an obstacle in extending the KAM theory to higher spatial dimensions.  In the latter case it is convenient to choose as $\calE$ a slightly more involved class of vector field, the so called 
\emph{quasi-T\"oplitz} or \emph{T\"oplitz-Lipschitz} vector fields.
The idea is to retain some information on the original PDE structure by showing that the linearized operator in the $w-$component can be ``approximated''
by piecewise multiplication operators.

A very good idea is to follow the approach of \cite{BBM},\cite{BBM1} where the Authors take advantage of the Second Mel'nikov conditions \eqref{secmel3} but do not 
apply the changes of variables which diagonalize the linearized operator. In this way at each step they preserve the PDE structure.
The key observation is that, in a Sobolev regularity setting, the bounds \eqref{cribbio4} and \eqref{buoni} follow from corresponding bounds  in some 
special coordinate system (i.e. the one in which the operator is diagonal) provided that the change of variables 
to the new coordinates is well-defined as an operator from the phase space to itself. Then there is no need to actually apply the change of variables. 
In the analytic context however this approach presents some difficulties, as one can see easily already in the case of the torus diffeomorphism, i.e. in studying 
the conjugate of a vector field $F$ under the map
\begin{equation}\label{diffto}
\TTT^{d}_{s}\ni\theta \mapsto \theta+g(\theta),
\end{equation}
for some $s\geq0$, and $g(\theta)$ small. Note that this change of variables is necessary in order to pass from \eqref{merdina1} to \eqref{merdina2}.
The map \eqref{diffto} induces an operator on the functions $f\in H^{p}(\TTT^{d}_{s})$  defined as $(\TT f)(\theta)=f(\theta+g(\theta))$.
It is easy to check that in the Sobolev context, i.e. $s=0$, the map \eqref{diffto} is a diffeomorphism of $\TTT^{d}$ into itself
and hence $\TT$ is well-defined from $H^{p}(\TTT^{d})$ into itself.
 On the contrary, 
if $s>0$ one has that the map \eqref{diffto} maps $\TTT^{d}_{s}$ into $\TTT^{d}_{s'}$ with $s'<s$, in other words there is a loss of analyticity tied to the size of $g$.
In this case one cannot follow the strategy used in \cite{BBM1}. 
By following directly such strategy one loses all the analyticity after a finite number of steps.
This is due to the fact that, even if the iterative Nash-Moser scheme is coordinate independent,
some of the estimates we perform actually depend on the system of coordinates. 
The basic  idea of our approach is in fact to use at each step the system of coordinates which is more adapted to the problem one is studying.
This point explains also the r\^ole of the compatible transformations $\calL_{n}$.
Indeed such transformation are not fundamental in proving the convergence of the scheme, but are introduced as a degree of freedom
in order to study the set of good parameters in \ref{pippopuffo3}. In particular such transformations are the key point in order to study problems in the analytic setting.

Up to now we have only considered degree decompositions, however one can certainly consider cases in which $\NN,\calX,\RR$ are not triangular.
For instance we may consider $\calE$ as in Example \ref{exHam}  with  the same $\calX$ as in  \eqref{sottoHAM} (i.e. it contains only terms of degree $\leq 0$ w.r.t. the decomposition with deg$(y)=2$) but where  $\RR$ contains  only (and all)  functions of degree $\geq 3$. Now this choice respects Definitions \ref{nomec} and \ref{nx} but clearly the decomposition is not triangular since $\Pi_\calX {\rm ad}N$ is not  block diagonal. However it is easily seen that $\Pi_\calX {\rm ad}(R)=0$ for any $R\in \RR$ so that $\Pi_\calX {\rm ad}(\Pi_\calX^\perp F)=\Pi_\calX {\rm ad}(\Pi_\NN F)$. In fact if  one divides $$\NN= \bigoplus_{j=0}^2 \NN_j$$  in terms of increasing degree then $\Pi_\calX {\rm ad}(\Pi_{\NN_0}F)$ is block diagonal on $\calX=\oplus \calX_j$ while $\Pi_\calX {\rm ad}(\Pi_{\NN_0}^\perp F)$ is upper triangular, so that solving the homological equation only depends on inverting $\Pi_\calX {\rm ad}(\Pi_{\NN_0}F)$ as in the previous example. This means that  we can apply Theorem\ref{thm:kambis} having a pretty explicit description of the set $\calO_\infty$.

\paragraph{On the smallness condition}\label{riscalo}

Let us comment  smallness conditions in Constraint \ref{sceltapar}.
First of all one can note that conditions \eqref{exp} are trivially non empty. Indeed given any $\mu,\nu,\ka_{3},\al,\gotp_{1},\chi$,
(which typically are fixed by the problem) then \eqref{exp1},\eqref{exp2} and \eqref{exp3} give lower bounds one $\ka_{1},\ka_{2},\h$, while
\eqref{exp4} and \eqref{exp5} constrain $\gotp_{2}$ in a non-empty interval.

The conditions on $\e_0,\tG_0,\tR_0$ are more subtle. Indeed \eqref{expexpexp}
gives a lower bound on the size of $K_{0}$, but then it is not trivial to show that \eqref{expexp}
can be fulfilled. Clearly if
$\tG_0\sim \tR_0\sim \e_0$ all the conditions 
reduce to a smallness condition on $\e_0$ in terms of $K_0$. If $\tG_0$ or $\tR_0$ are large the problem gets more complicated, see also Remark \ref{concreto?}. Unfortunately this situation appears in many applications, for this reason we have not made any simplifying assumption on the relative sizes. 

As one sees in \eqref{sizes}, the parameters $\tG_0,\tR_0,\e_0$ give upper bounds on the size of $G_0,\Pi_{\NN^\perp}G_0,\Pi_{\calX} G_0$ w.r.t. the parameter $\gamma_0$.
In applications $G_0,\gamma_0$ are essentially given so that the only hope of modulating the bounds comes from modifying the domain $ \TTT^d_s\times D_{a,p}(r)\times \calO_0$. 
To this purpose we  first make the trivial remark that if one shrinks to a suitably small neighborhood of zero then polynomials of high degree become very small. In order to formalize this fact we define a scaling degree as follows.
$$
\mathtt s(\theta)=0\,,\quad \mathtt{s}(y)=\mathtt s\,, \quad \mathtt{s}(w)=1. 
$$
Then given a monomial vector field this fixes the scaling as 
$$\mathtt s(y^j e^{\ii \theta\cdot \ell} w^\al \partial_{\tv})= \mathtt s  j +|\al|-\mathtt s(\tv) .$$
  By construction the scaling is additive w.r.t. commutators
and behaves just as the degree in  Remark \ref{sonfica}.
We have the following result.

\begin{lemma}\label{scaling}
Consider a tame vector field $F$ as in Definition \ref{tame} of minimal scaling $\bar{\mathtt{s}}$. 
Consider the rescaling $r_0\rightsquigarrow \de r_0$, $r \rightsquigarrow \de r$.
Then one has
\begin{equation}
C_{\vec{v}_1,p}(F)\leq \de^{\bar{\mathtt{s}}} C_{\vec{v},p}(F),
\end{equation}
with $\vec{v}=(\g,\calO,s,a,  r)$ and  $\vec{v}_1=(\g,\calO,s,a,\de r)$.
\end{lemma}

%
%
%

This definition of scaling induces a natural scaling decomposition of a vector field. In particular we remark that 
by construction $\NN$ contains necessarily  some terms of scaling zero, while $\calX$ contains  terms of negative scaling.
However one can fix the scaling so that $\RR$ contains only terms of positive scaling. 
Hence by Lemma \ref{scaling} terms of positive scaling can be made small.
In conclusion given the constants $\bar\tG_0,\bar\e_0,\bar\tR_0$ in a ball $\bar{r}_0$, one can try to fulfill \ref{expexp} by
rescaling $r_0=\de \bar{r_0}$. in this way $\tR_0$ becomes smaller, $\tG_0$ at best does not grow, while 
$\e_0$ necessarily grows.  This procedure produces upper and lower bounds on $\de$. 

Consider the decomposition in \eqref{sottoHAM} where the degree is equal to the scaling. Then  $\NN$ 
contains only terms of degree zero and hence $\tG_0$ is scaling invariant. $\RR$ contains only terms of positive degree
so that rescaling $\tR_0\rightsquigarrow \de \bar{\tR}_0$ at least. $\calX$ has negative degree $\geq-2$, hence $\e_0\rightsquigarrow \de^{-2}\bar{\e}_0$.
As explained in Remark \ref{concreto?} constants $\bar{\tG_0},\bar{\e}_0,\bar{\tR}_0$ are expected to depend only on $\g$. Using the rescaling we have introduced the 
extra parameter $\de$, which should be chosen in terms of $\g$ in order to fulfill the conditions  \eqref{expexp}.
Actually it can be useful to make a finer analysis for $\bar{\e}_0$. Indeed one can bound separately the terms of degree $-2,-1,0$ in $\bar{\e}_0$. Let us denote
them by $\bar{\e}_0^{(i)}$, $i=-2,-1,0$. Then ${\e}_0\rightsquigarrow \sum_{i=-2}^{0}\de^{i}\bar{\e}_0^{(i)}$, and hence the smallness conditions can be taken asymetrically.
The same holds for $\tR_0$. This has been discussed with slightly different notation in \cite{BB}.

\zerarcounters
\section{Proof of the result}\label{itscheme}

We divide the proof of Theorem \ref{thm:kambis}  into two pieces: we first show one
step in full details (this is the same in both cases) and then we prove that it is indeed possible to perform infinitely many
steps and that the procedure converges.

\subsection{The KAM step}

\begin{proposition}\label{kamstep}
Let $\g_0$, $a_0$, $r_0$, $s_0$, $\e_0$  and $K_0$ be the constants appearing in
Theorem \ref{thm:kambis}. 
Fix  $\g, a,r,s\ge0$ so that
\begin{equation}\label{gallina}
\frac{\g_0}{2}\leq \g \le \g_0,\,\,\frac{a_{0}}{2}\leq a\leq a_{0},\;  \frac{s_{0}}{2}\leq s\leq s_{0},\; \frac{r_{0}}{2}\leq r\leq r_{0}
\end{equation}
and  $0<\rho<1$ such that 
\begin{equation}\label{gallina3}
r-8\rho r_0>\frac{r_0}{2},  \;\; {\rm if} \; s_0\neq0 \; {\rm then} \; s-8\rho s_0>\frac{s_0}{2},\,\, {\rm if} \; a_0\neq0 \,\, {\rm then} \;\; a-8\rho a_0>\frac{a_0}2.
\end{equation}
Consider  a  vector field:
\begin{equation}\label{gallina2}
F : \TTT^d_s\times D_{a,p+\nu}( r)\times \calO_0\to V_{a,p},
\end{equation}
 which is  $C^{\mathtt n+2}$-tame up to order $q=\gotp_2+2$.  
 Let $N_0$ be the diagonal vector field appearing in Theorem \ref{thm:kambis} and
write $F=N_0+(F-N_0)=N_0+G$.
 Set $\vec v=(\g,\calO,s,a,r)$, denote by $\calO\subseteq \calO_0$ some set of parameters $\x$ for which 
$F(\x)\in\calE$
and $|\Pi_{\calX}G|_{\vec{v},\gotp_2-1}\leq \mathtt{C}C_{\vec{v},\gotp_{2}}(\Pi_{\NN}^{\perp}G)$ and 
define also
 \begin{equation}\label{sizes2}
\Gamma_{p}:=\g^{-1}C_{\vec{v},p}(G), \quad 
 \Theta_{p}:=\g^{-1}C_{\vec{v},p}(\Pi_{\NN}^\perp G), \quad  
 \de:=\g^{-1}|\Pi_{\calX}G|_{\vec{v},\gotp_1}.
 \end{equation}
Fix $ K>K_0$ and  assume that 
 \begin{equation}\label{kam32}
\rho^{-1} K^{\mu+\nu+3}\Gl \de \le \epsilon
 \end{equation}
 with $\epsilon=\epsilon(\gotp_1,d)$  small enough.
Consider  a map $\calL$  compatible with $(F,K,\vec v,\rho)$ (see Definition \ref{compa}) and set 
$$
\hat{F}=N_0+\hat{G}:=(\calL)_{*}F\,.
$$
Consider any
 $\calO_{+}\subseteq \calO$ solving the homological equation for $(\hat F,K,\vec v^{\text{\tiny 0}}_2,\rho)$ and set
  \begin{equation}\label{vino}
  \begin{aligned}
  \vec w_i&=(\gamma,\calO_+,s-i\rho  s_0,a -i \rho a_0, r-i \rho  r_0),\\
  \vec v_i^{\text{\tiny 0}}&=(\gamma,\calO_0,s-i\rho  s_0,a -i \rho a_0, r-i \rho  r_0),
   \end{aligned}
   \end{equation}
   for $i=1,\ldots,8$.
 The following holds:
\smallskip

\noindent (i)
 there exists an invertible (see Def. \ref{leftinverse}) change of variables
\begin{equation}\label{kam66}
\Phi_+:=\uno+{f}_+ : \TTT^d_{s-4\rho s_0}\times D_{a-4\rho a_0,p}(r-4\rho r_0)\times
\calO_0\longrightarrow \TTT^d_{s-2\rho s_0}\times D_{a-2\rho  a_0,p}(r-2\rho  r_0) ,
\end{equation}
with $f_+$ a regular vector field (see Def. \ref{linvec}) such that 
$\Phi_+$ is $\calE$ preserving for all $\xi\in\calO_+$ and is generated by a vector field $g_+$ which satisfies the bound
\begin{equation}\label{stimatotale2}
\begin{aligned}
|g_+|_{\vec w_2,\gotp_1}&\leq\mathtt C \Gamma_{\gotp_1}K^{\mu} \de
\\
|g_+|_{\vec w_2,\gotp_2}
&\le {\mathtt C} K^{\mu+1} (\Theta_{\gotp_2}+\e_0 K^{\ka_3} \Theta_{\gotp_1} + K^{\al(\gotp_2-\gotp_1)} \de(\Gamma_{\gotp_2}+\e_0 K^{\ka_3} \Gamma_{\gotp_1}))\,;
\end{aligned}
\end{equation}
 (ii) fix $s_+,a_+,r_+$ as
\begin{equation}\label{kam6bis}
s_+=s-8\rho  s_0, \quad a_+=a-8\rho  a_0,
 \quad r_+= r-8\rho  r_0
\end{equation}
then
\begin{equation}\label{gamgee}
F_{+}:=(\Phi_+)_{*}\hat{F}=N_{0}+G_{+}:\TTT^d_{s_+}\times D_{a_{+},p+\nu}(r_{+})\times \calO_0
\to V_{a,p}\,;
\end{equation}

\noindent
(iii) setting  $\g_0/2\leq\g_{+}\leq \g$, $\vec{v}_{+}:=(\g_{+},\calO_{+},a_+,s_+)=\vec w_8$, 
$F_{+}$ is tame 
and denoting the tameness constants as
$$\Gamma_{+,p}:=\g_{+}^{-1}C_{\vec{v}_{+},p}(G_+), \quad \Theta_{+,p}:=\g_{+}^{-1}
C_{\vec{v}_{+},p}(\Pi_{\NN}^\perp G_+),\quad 
\de_{+}:=\g_{+}^{-1}|\Pi_{\calX}G_{+}|_{\vec{v}_{+},\gotp_1}
$$
 one can fix
\begin{equation}\label{kam8}
\begin{aligned}
\Gamma_{+,\gotp_1}&=  
\frac{\g}{\g_+}
(1+\e_0K^{-1})\Gamma_{\gotp_1}+\mathtt C K^\mu \Gl \de (K^{\nu+1}\Gamma_{\gotp_1} + \e_0  K^{-\eta})\,, \\
\Gamma_{+,\gotp_{2}}&=  {\mathtt C}\Big(\Gamma_{\gotp_2}+ \e_0 \Gl K^{\ka_3}+\Gl K^{\mu+\nu+2}
(\Theta_{\gotp_2}+\e_0 K^{\ka_3} \Theta_{\gotp_1}+ K^{\al(p-\gotp_1)} \de(\Gamma_{\gotp_2}+\e_0 K^{\ka_3} \Gamma_{\gotp_1}))
\Big)\,,
\end{aligned}
\end{equation}
\begin{equation}\label{kam9}
\begin{aligned}
 \Theta_{+,\gotp_1}&= \frac{\g}{\g_+}
 (1+\e_0K^{-1})\Theta_{\gotp_1}+ \mathtt C K^\mu \Gl \de (K^{\nu+1}\Gamma_{\gotp_1} + \e_0  K^{-\eta})\,, \\
\Theta_{+,\gotp_2}&=\mathtt{C} \big( \Theta_{\gotp_2}+\e_0 K^{\ka_3} \Theta_{\gotp_1} +
  K^{\mu+\nu+2} \Gl \big( \Theta_{\gotp_2}+\e_0 K^{\ka_3} \Theta_{\gotp_1} + K^{\al(\gotp_2-\gotp_1)} \de(\Gamma_{\gotp_2}+\e_0 K^{\ka_3} \Gamma_{\gotp_1})\big )\big)\,,\\
\g_{+}^{-1}&|\Pi_{\calX}(G_{+})|_{\vec{v}_+,\gotp_{2}-1}\leq \Theta_{+,\gotp_{2}}
\end{aligned}
\end{equation}
and 
\begin{equation}\label{cazzo}
\begin{aligned}
\de_{+}&\le {\mathtt C} \Gl \big ( \de^2 
\Gamma^2_{\gotp_1}K^{2\mu+2\nu+4} + \de \e_0 K^{\mu-\eta}\big)+  K^{\mu+\nu+2-(\gotp_2-\gotp_1) }  \big(\Theta_{\gotp_2}+\e_0 K^{\ka_3} \Theta_{\gotp_1} \big)\\
&\qquad+
\Gl K^{\mu+\nu+2-(\gotp_2-\gotp_1) }  \big (\Theta_{\gotp_2}+\e_0 K^{\ka_3} \Theta_{\gotp_1} + K^{\al(\gotp_2-\gotp_1)} \de(\Gamma_{\gotp_2}+\e_0 K^{\ka_3} \Gl)\big)\,.
\end{aligned}
\end{equation} 
\end{proposition}

\begin{proof}
First of all we note that by the definition of $\calL$  one has
$$\hat F: \TTT^d_{s - 2\rho  s_0 }\times D_{a- 2\rho  a_0 ,p+\nu}(r-2 \rho  r_0 )\times \calO_0 \to V_{a,p}$$ and
$\hat F\in \calE$ for each $\xi$ in $\calO$.
By \eqref{satana2} we have
$$
\begin{aligned} 
&\Pi_\calX \hat G=\Pi_\calX \hat F = \Pi_\calX (\calL)_* F = \Pi_\calX (\calL)_* \Pi_\calX F=  
\Pi_\calX (\calL)_* \Pi_\calX G\,,\\
&\Pi_{\NN}^\perp\hat{G}=
\Pi_\NN^\perp \hat F = \Pi_\NN^\perp (\calL)_* F = \Pi_\NN^\perp (\calL)_* \Pi_\NN^\perp F
=  \Pi_\NN^\perp (\calL)_* \Pi_\NN^\perp G\,.
\end{aligned}
$$
Now since $G$ is $C^{\mathtt n}$-tame and $\NN,\calX$ have maximal degree $\leq\mathtt n$  we have that the tameness constants of $\Pi_\NN G, \Pi_{\RR}G$ as well as $|\Pi_\calX G|$ are controlled by the tameness constant of $G$ .

By \eqref{odio} we have the bounds
\begin{equation}\label{formosa}
\begin{aligned}
&C_{\vec w_2,\gotp_{1}}(\hat G )\leq C_{\vec v,\gotp_{1}}( G)(1+\e_0 K^{-1} )\leq \g \Gl (1+\e_0 K^{-1})\,,\\
&C_{\vec w_2,\gotp_{2}}( \hat G)\leq C_{\vec v,\gotp_{2}}( G)+\e_0 K^{\ka_3} C_{\vec v,\gotp_{1}}(G) \le\g( \Gamma_{\gotp_2}+\e_0 K^{\ka_3} \Gamma_{\gotp_1})\,,
\end{aligned}
\end{equation}

\begin{equation}\label{kam11bis}
\begin{aligned}
&C_{\vec w_2,\gotp_{1}}(\Pi_\NN^\perp\hat G )\leq C_{\vec v,\gotp_{1}}(\Pi_\NN^\perp G)(1+\e_0 K^{-1} )\leq \g \Theta_{\gotp_1} (1+\e_0 K^{-1})\,,\\
&C_{\vec w_2,\gotp_{2}}( \Pi_\NN^\perp \hat G)\leq C_{\vec v,\gotp_{2}}( \Pi_\NN^\perp G)+\e_0 K^{\ka_3} C_{\vec v,\gotp_{1}}(\Pi_\NN^\perp G) \le\g( \Theta_{\gotp_2}+\e_0 K^{\ka_3} \Theta_{\gotp_1})\,,
\end{aligned}
\end{equation}
and 
\begin{equation}\label{kam11}
\begin{aligned}
&C_{\vec w_2,\gotp_{1}}(\Pi_\calX \hat G )\leq C_{\vec v,\gotp_{1}}(\Pi_\calX G)(1+\e_0 K^{-1} )\leq 2\g \de\,,\\
&C_{\vec w_2,\gotp_{2}}(\Pi_\calX \hat G )\leq C_{\vec v,\gotp_{2}}(\Pi_\calX G)+\e_0 K^{\ka_3} C_{\vec v,\gotp_{1}}(\Pi_\calX G) \le\g( \Theta_{\gotp_2}+\e_0 K^{\ka_3} \Theta_{\gotp_1})\,,\\
&|\Pi_{\calX} \hat{G}|_{\vec{w}_2,\gotp_2-1}\leq \g( \Theta_{\gotp_2}+\e_0 K^{\ka_3} \Theta_{\gotp_1}).
\end{aligned}
\end{equation}

Our aim is to define for $\x\in\calO_+$ a vector field ${g}_{+}$ as the ``approximate'' solution of the equation
\begin{equation}\label{maria}
\Pi_{K}\Pi_{\calX}[{g}_+,\Pi_\calX^\perp \hat F ]=\Pi_{K}\Pi_{\calX} \hat{F}.
\end{equation}
By Definition, if $\x\in \calO_+$ then  we can
find ${g}_+$ satisfying  properties (a),(b),(c),(d) of Definition \ref{pippopuffo2}.

By \eqref{buoni22} one gets
\begin{equation}\label{stimatotale}
\begin{aligned}
|{g}_+|_{\vec w_2,p}\leq & \mathtt C
\gamma^{-1}K^{\mu}(|\Pi_K\Pi_{\calX}\hat G|_{\vec w_2,p}
+K^{\al(p-\gotp_1)}|\Pi_K\Pi_{\calX}\hat G|_{\vec w_2,\gotp_1}
\g^{-1}C_{{\vec w_2,p}}(\hat G))
\end{aligned}
\end{equation}
and hence, using \eqref{formosa},\eqref{kam11} and \eqref{P1}  we have \eqref{stimatotale2}
Moreover, by condition $(a)$ of Definition \ref{pippopuffo2}, one has  
$g_{+}\in \BB_\calE$ for $\xi\in \calO_+$  and $|g_+|_{\vec v_2^{\tino},\gotp_1}\le \mathtt C |g_+|_{\vec w_2,\gotp_1}$.

Now, if $\epsilon$ in \eqref{kam32} is small enough, by Definition \ref{linvec-abs} item 5,  $g_+$ generates a change of variables 
${\Phi}_{+}=\uno+{f}_{+}$and 
 $|{f}_{+}|_{\vec v_3^{\text{\tiny 0}},p}\leq 2|\tilde g_+|_{\vec v_2^{\text{\tiny 0}},p}$.
Finally, for possibly smaller $\epsilon$ one has
\eqref{kam66}, by using Definition \ref{linvec-abs} item 4.  Note that the smallness conditions on $\epsilon$ come only from this two conditions.

\smallskip

First we note that, since $N_0$ is diagonal (recall Definition \ref{norm}), one has
\begin{equation}\label{sabbia}
\begin{aligned}
G_{+}&:=(\Phi_{+})_{*}N_0+(\Phi_{+})_{*}\hat{G}-N_0=\int_{0}^{1}dt\,(\Phi_{+})_{*}^{t}[g_+,N_0]+
(\Phi_{+})_{*}\hat{G}\\
&=\int_{0}^{1}dt(\Phi_{+})_{*}^{t}\Pi_{K}\Pi_{\calX}[g_+,N_{0}+\Pi_\calX^\perp\hat{G}]-
\int_{0}^{1}dt(\Phi_{+})_{*}^{t}\Pi_{K}\Pi_{\calX}[g_+,\Pi_\calX^\perp\hat{G}]
+(\Phi_{+})_{*}\hat{G}\\
&=\int_{0}^{1}dt(\Phi_+)_{*}^{t}(\Pi_{K}\Pi_\calX\hat{G}+u)-
\int_{0}^{1}dt(\Phi_+)_{*}^{t}\Pi_{K}\Pi_{\calX}[g,\Pi_\calX^\perp\hat{G}]
+(\Phi_{+})_{*}\hat{G}
\end{aligned}
\end{equation}
where 
\begin{equation}\label{kam16}
u:=\Pi_{K}\Pi_{\calX}[g_{+},N_{0}+\Pi_\calX^\perp\hat{G}]-\Pi_{K}\Pi_\calX\hat{G}
\end{equation}
and $u$ in \eqref{kam16} satisfies  \eqref{cribbio42}, so by applying \eqref{formosa} and \eqref{kam11}, we get
\begin{equation}\label{dorata}
\begin{aligned}
|u|_{\vec w_2,\gotp_1} & \leq \mathtt C \g  \e_0 \Gamma_{\gotp_1} K^{-\eta+\mu} \de 
 \\
|u|_{\vec w_2,\gotp_2} & \leq \mathtt C \g K^{\mu+1}\left(
   (\Theta_{\gotp_2} +\e_0 K^{\ka_3}\Theta_{\gotp_1})\Gamma_{\gotp_1}+K^{\al(\gotp_2-\gotp_1)}\de (\Gamma_{\gotp_2} +\e_0 K^{\ka_3}\Gamma_{\gotp_1})
    \right);
\end{aligned}
\end{equation}

Regarding the first summand of \eqref{sabbia}, using Lemma \ref{conj}, we have
\begin{equation}\label{gallina5}
\begin{aligned}
C_{\vec{w}_{8},\gotp_{1}}\left(\int_{0}^{1}dt\Phi_{*}^{t}(\Pi_{K}\Pi_\calX\hat{G}+u)\right)&
\stackrel{\eqref{dorata},\eqref{kam32}}{\leq}
\g{\mathtt C}\de(1+\g  \e_0 \Gamma_{\gotp_1} K^{-\eta+\mu} )
\end{aligned}
\end{equation}
 Moreover
\begin{equation}\label{kam18}
\begin{aligned}
\!\!\!\! C_{\vec{w}_{8},\gotp_{2}}\left(\int_{0}^{1}dt\Phi_{*}^{t}(\Pi_{K}\Pi_\calX\hat{G}+u)\right)&\leq
(1+\rho )
(C_{\vec{w}_6,\gotp_{2}}(\Pi_\calX\hat{G}+u)+C_{\vec w_6,\gotp_{1}}(\Pi_\calX\hat{G}+u)
|f_+|_{\vec w_7,\gotp_2+\nu+1})\\
\stackrel{(\ref{cribbio42}),\eqref{eq106},\eqref{kam32}}{\leq}&{\mathtt C}\g  K^{\mu+1}\left(
   (\Theta_{\gotp_2} +\e_0 K^{\ka_3}\Theta_{\gotp_1})\Gamma_{\gotp_1}+K^{\al(\gotp_2-\gotp_1)}\de (\Gamma_{\gotp_2} +\e_0 K^{\ka_3}\Gamma_{\gotp_1})\right) \,.
%
%
%
\end{aligned}
\end{equation}

The second term of \eqref{sabbia} can be estimated as follows. First we note that, since $\hat G$ is $C^{\mathtt n+2}$-tame up to order\footnote{due to the truncation $\Pi_K$} $q+1$, then  the vector field 
$\Pi_{K}\Pi_{\calX}[g_+,\Pi_\calX^\perp\hat{G}]$ is $C^{\mathtt n+1}$-tame up to order $q$. This implies, due to the presence of the projection onto $\calX$, that it is indeed $C^{k}$-tame for all $k$.
The tameness constant is given by
\begin{equation}\label{kam19}
\begin{aligned}
C_{\vec{w}_8,p}(\Pi_{K}\Pi_{\calX}[g_+,\Pi_\calX^\perp\hat{G}])&\stackrel{\eqref{P3}}{\leq} 
K C_{\vec{w}_8,p-1}(\Pi_{\calX}[g_+,\Pi_\calX^\perp\hat{G}])\\
&\leq \mathtt C K^{\nu+1}(|g_{+}|_{\vec w_8,p}C_{\vec{w}_8,\gotp_{0}}(\Pi_\calX^\perp\hat{ G})+
|g_{+}|_{\vec w_8,\gotp_0}C_{\vec{w}_8,p}(\Pi_\calX^\perp\hat{ G}))\\
& \stackrel{{rmk} \ref{monomi}}{\leq}  \mathtt C K^{\nu+1}(|g_{+}|_{\vec w_8,p}C_{\vec{w}_8,\gotp_{0}}(\hat{ G})
+
|g_{+}|_{\vec w_8,\gotp_0}C_{\vec{w}_8,p}(\hat{ G})
),
\end{aligned}
\end{equation}
hence, by \eqref{formosa} and \eqref{stimatotale2}, we have
\begin{equation}\label{kam19tris}
\begin{aligned}
C_{\vec{v}_+,\gotp_1}(\Pi_{K}\Pi_{\calX}[g_+,\Pi_\calX^\perp\hat{G}])&\leq \g \mathtt C \Gamma^2_{\gotp_1}
K^{\mu+\nu+1}\de\,,
\end{aligned}
\end{equation}
\begin{equation}\label{kam19bis}
\begin{aligned}
C_{\vec{v}_+,\gotp_2}(\Pi_{K}\Pi_{\calX}[g_+,\Pi_\calX^\perp\hat{G}])
&{\leq}
{\mathtt C} \g K^{\mu+\nu+2}\Gamma_{\gotp_1}\Big(
(\Theta_{\gotp_2}+\e_0 K^{\ka_3} \Theta_{\gotp_1} + K^{\al(\gotp_2-\gotp_1)} \de(\Gamma_{\gotp_2}+\e_0 K^{\ka_3} \Gamma_{\gotp_1}) 
\Big)\,.
\end{aligned}
\end{equation}
Therefore using Lemma \ref{conj} we have
\begin{equation}\label{gallina6}
C_{\vec{w}_{8},\gotp_1}(\int_{0}^{1}dt(\Phi_{+})_{*}^{t}\Pi_{K}\Pi_{\calX}
[g_+,\Pi_\calX^\perp\hat{G}])\leq\g\mathtt C \Gamma_{\gotp_1}^2K^{\mu+\nu+1}\de
\end{equation}
and
\begin{equation}\label{gallina6bis}
\begin{aligned}
C_{\vec{w}_{8},\gotp_2}(\int_{0}^{1}dt(\Phi_{+})_{*}^{t}\Pi_{K}\Pi_{\calX}
[g_+,\Pi_\calX^\perp\hat{G}])\leq \g
{\mathtt C}  & K^{\mu+\nu+2}\Gamma_{\gotp_1} \Big(
\Theta_{\gotp_2}+\e_0 K^{\ka_3} \Theta_{\gotp_1} \\
 &+ K^{\al(\gotp_2-\gotp_1)} 
\de(\Gamma_{\gotp_2}+\e_0 K^{\ka_3} \Gamma_{\gotp_1}) 
\Big)
\end{aligned}
\end{equation}
by \eqref{kam32}.

Finally, again by Lemma \ref{conj}, we estimate the third summand in \eqref{sabbia} as
\begin{equation}\label{kam17bis}
\begin{aligned}
C_{\vec{w}_{8},\gotp_1}(\Phi_{*}(\hat{G}))&\stackrel{\eqref{formosa},\eqref{stimatotale2}}{\le}
\g \Gamma_{\gotp_1}(1+{\mathtt C}\de \Gamma_{\gotp_1} K^\mu)(1+\e_0 K^{-1})
\end{aligned}
\end{equation}
and
\begin{equation}\label{spapagno}
\begin{aligned}
C_{\vec{w}_{8},\gotp_2}((\Phi_+)_{*}\hat{G}) &\le (1+{\mathtt c}^{-1}
|f_+|_{\vec w_7,\gotp_1})
\Big(
C_{\vec w_6,\gotp_2}(\hat G)+C_{\vec w_6,\gotp_1}(\hat G)|f_+|_{\vec w_7,\gotp_2+\nu+1}
\Big)\\\\
&\stackrel{\eqref{formosa},\eqref{stimatotale2}}{\le} \g {\mathtt C}
\Big(\Gamma_{\gotp_2}+\e_0 K^{\ka_3} \Gl\\
&\qquad\qquad+  \Gl K^{\mu+\nu+2}(\Theta_{\gotp_2}+\e_0 K^{\ka_3} \Theta_{\gotp_1} + 
K^{\al(p-\gotp_1)} \de(\Gamma_{\gotp_2}+\e_0 K^{\ka_3} \Gamma_{\gotp_1}))
\Big)\,.
\end{aligned}
\end{equation}

The bounds \eqref{kam8} follow by collecting together
 \eqref{gallina5}, \eqref{kam18}, \eqref{gallina6},\eqref{kam17bis} and \eqref{spapagno}.

Let us study $\Theta_{+,p}$. First of all we see that
\begin{equation}\label{uffa3}
\Pi_{\NN}^\perp G_{+}=\Pi_\NN^\perp F_+=\Pi_\NN^\perp(\Phi_+)_*\hat F=\Pi_{\NN}^\perp\Big(
(\Phi_{+})_{*}N_0+(\Phi_{+})_{*}(\Pi_{\NN}\hat{G})+(\Phi_{+})_{*}(\Pi_{\NN}^\perp\hat{G})
\Big)\,.
\end{equation}

In order to estimate the first term $\Pi_{\NN}^\perp(\Phi_{+})_{*}N_0$, we first note that
\begin{equation}\label{telofo}
\begin{aligned}
\Pi_{\NN}^\perp(\Phi_{+})_{*}N_0&=\Pi_{\NN}^\perp((\Phi_{+})_{*}N_0-N_0)
=\Pi_\NN^\perp\int_0^1d t\,(\Phi_+^t)_*[g_+,N_0]
=\Pi_\NN^\perp\int_0^1d t\,(\Phi_+^t)_*(\Pi_{K}\Pi_{\calX}[g_+,N_0])\\
&=\Pi_\NN^\perp\Big(
\int_{0}^{1}dt(\Phi_{+})_{*}^{t}\Pi_{K}\Pi_{\calX}[g_+,N_{0}+\Pi_\calX^\perp\hat{G}]-
\int_{0}^{1}dt(\Phi_{+})_{*}^{t}\Pi_{K}\Pi_{\calX}[g_+,\Pi_\calX^\perp\hat{G}]
\Big)\,,
\end{aligned}
\end{equation}
 substituting \eqref{kam16} and using  Remark \ref{monomi} in order to remove the projection
we have
\begin{equation}\label{ennezero}
\begin{aligned}
&C_{\vec w_8,\gotp_1}(\Pi_{\NN}^\perp(\Phi_{+})_{*}N_0)\stackrel{\eqref{gallina5},\eqref{gallina6}}{\le}  \mathtt C \g K^\mu \Gl \de (K^{\nu+1}\Gamma_{\gotp_1} + \e_0  K^{-\eta}) \,,\\
&C_{\vec w_8,\gotp_2}(\Pi_{\NN}^\perp(\Phi_{+})_{*}N_0)\stackrel{\eqref{kam18},\eqref{gallina6}}{\le}
{\mathtt C}\g   K^{\mu+\nu+2}\Gamma_{\gotp_1}\Big(
   \Theta_{\gotp_2}+\e_0 K^{\ka_3} \Theta_{\gotp_1} + K^{\al(p-\gotp_1)} \de(\Gamma_{\gotp_2}+\e_0 K^{\ka_3} \Gamma_{\gotp_1}) 
   \Big)\,.
\end{aligned}
\end{equation}

In order to bound the second summand we use  Lemma \ref{normnorm} with $\calU=\NN$ and obtain
\begin{equation}\label{uffa6}
\begin{aligned}
C_{\vec{w}_{8},\gotp_1}(\Pi_{\NN}^\perp(\Phi_{+})_{*}(\Pi_{\NN}\hat{G}))&\leq
{\mathtt C} 
|f_+|_{\vec w_6,\gotp_1+\nu+1}C_{\vec w_6,\gotp_1}(\Pi_{\NN}\hat{G})
\stackrel{rmk\ref{monomi}}{\le} {\mathtt C} 
|f_+|_{\vec w_6,\gotp_1+\nu+1}C_{\vec w_6,\gotp_1}(\hat{G})\\
&\stackrel{\eqref{formosa}}{\le}
\g{\mathtt C} K^{\mu+\nu+1} \de\Gamma^2_{\gotp_1}, 
\end{aligned}
\end{equation}
and
\begin{equation}\label{uffa5}
\begin{aligned}
C_{\vec{w}_{8},\gotp_2}(\Pi_{\NN}^\perp(\Phi_{+})_{*}(\Pi_{\NN}\hat{G}))&\leq
{\mathtt C}\Big(
 C_{\vec{w}_2,\gotp_2}(\hat{G})|f_{+}|_{\vec w_6,\gotp_1}+C_{\vec{w}_2,\gotp_1}(\hat{G})
 |f_+|_{\vec w_6,\gotp_2}\Big)\\
 &\stackrel{(\ref{formosa})}{\leq}\g {\mathtt C}
 \Gl K^{\mu+1} (\Theta_{\gotp_2}+\e_0 K^{\ka_3} \Theta_{\gotp_1} + K^{\al(p-\gotp_1)} \de(\Gamma_{\gotp_2}+\e_0 K^{\ka_3} \Gamma_{\gotp_1}))
\end{aligned}
\end{equation}

Regarding the third summand, using Remark \ref{monomi}, Lemma \ref{conj} and \eqref{kam11bis} we obtain
\begin{equation}\label{uffa4}
\begin{aligned}
C_{\vec{w}_8,\gotp_{1}}(\Pi_\NN^\perp(\Phi_{+})_{*}(\Pi_{\NN}^\perp \hat{G}))
&\stackrel{\eqref{stimatotale2}}{\leq} \g (1\!+\!\mathtt C \Gl K^{\mu}\de)(1+\e_0 K^{-1})
\Theta_{\gotp_1},\\
\!\!\!C_{\vec{w}_8,\gotp_{2}}(\Pi_\NN^\perp(\Phi_{+})_{*}
(\Pi_{\NN}^\perp\hat{G}))&\!\!\!\!\stackrel{\eqref{stimatotale2}}{\leq}\!\!\!
\mathtt{C}\g ( \Theta_{\gotp_2}+\e_0 K^{\ka_3} \Theta_{\gotp_1})+\\
&+
{\mathtt C} \g K^{\mu+\nu+2} \Tl \big( \Theta_{\gotp_2}+\e_0 K^{\ka_3} \Theta_{\gotp_1} + K^{\al(p-\gotp_1)} \de(\Gamma_{\gotp_2}+\e_0 K^{\ka_3} \Gamma_{\gotp_1})\big ).
\end{aligned}
\end{equation} 

By collecting together \eqref{ennezero}, \eqref{uffa6}, \eqref{uffa5} and \eqref{uffa4} we obtain the first two lines of \eqref{kam9}. In order to prove the last of \eqref{kam9} we use item (d) of Definition \ref{pippopuffo2}.
Indeed we substitute \eqref{stimatotale2} ,\eqref{formosa} and \eqref{kam11bis} in \eqref{cribbio420} and we obtain the desired bound.
\smallskip

In order to prove the bound for $\de_+$ we first write
\begin{equation}\label{orchi}
\Pi_{\calX} {G}_{+}= \Pi_\calX (\hat{F}+[\hat{F},g_{+}]+Q), \qquad
Q:=(\Phi_{+})_* \hat{F}-(\hat{F}+[\hat{F},g_{+}])
\end{equation}
and hence
\begin{equation}\label{orchi2}
\begin{aligned}
\Pi_\calX {G}_{+}&= 
\Pi_\calX\hat{F}+\Pi_\calX[\Pi_\calX^\perp\hat{F},g_{+}]+\Pi_\calX[\Pi_\calX\hat{F},g_{+}]+\Pi_\calX Q\\
&=
u+\Pi_{K}^{\perp}\Big(
\Pi_\calX\hat{G}+\Pi_\calX [\Pi_\calX^\perp\hat{G},g_{+}]
+\Pi_\calX[\Pi_\calX\hat{G},g_{+}]\Big) + \Pi_{K}\Pi_\calX[\Pi_\calX\hat{G},g_{+}] +\Pi_\calX Q\\ 
&=
u+\Pi_{K}^{\perp}\Big(
\Pi_\calX\hat{G}+\Pi_\calX [\hat{G},g_{+}]\Big) + \Pi_{K}\Pi_\calX[\Pi_\calX\hat{G},g_{+}] +\Pi_\calX Q
\end{aligned}
\end{equation}
where $u$ is defined in \eqref{kam16} and bounded in \eqref{dorata}.

The second summand in \eqref{orchi2} can be bounded as
\begin{equation}\label{orchi8}
\begin{aligned}
&|\Pi_K^{\perp}\Big(
\Pi_\calX\hat{G}+\Pi_\calX [\hat{G},g_{+}]
\Big)|_{\vec{w}_8,\gotp_1}\stackrel{\eqref{P2}}{\leq}
K^{-(\gotp_2-\gotp_1)+2} |\Pi_\calX\hat{G}+\Pi_\calX [\hat{G},g_{+}])|_{\vec{w}_6,\gotp_{2}-2}\\
&\stackrel{(\ref{kam11}),(\ref{buoni22})}{\leq}{\mathtt C}\g
K^{-(\gotp_2-\gotp_1)+\nu+\mu +2 } \Gl (\Theta_{\gotp_2}+\e_0 K^{\ka_3} \Theta_{\gotp_1}
+ K^{\al(\gotp_2-\gotp_1)} \de(\Gamma_{\gotp_2}+\e_0 K^{\ka_3} \Gamma_{\gotp_1}))\\
&\quad \quad+{\mathtt C}\g K^{-(\gotp_2-\gotp_1)+2}(\Theta_{\gotp_2}+\e_0 K^{\ka_3} \Theta_{\gotp_1}).
\end{aligned}
\end{equation}

We can choose the tameness constant of the third summand in \eqref{orchi2}  as follows
\begin{equation}\label{caffe}
\begin{aligned}
C_{\vec w_8,\gotp_1}(\Pi_{K}\Pi_\calX[\Pi_\calX\hat{F},g_+])&\stackrel{\eqref{P3}}{\leq}K
C_{\vec w_8,\gotp_1-1}(\Pi_{K}\Pi_\calX[\Pi_\calX\hat{F},g_+])
\stackrel{\eqref{commu2}}{\le}
{\mathtt C}K^{\nu+1}
C_{\vec w_8,\gotp_1}(\Pi_\calX\hat{F})|g_+|_{\vec w_8,\gotp_1}\\
&\stackrel{\eqref{kam11},\eqref{stimatotale2}}{\le} \g{\mathtt C}K^{\mu+\nu+1}\Gl \de^2\,.
\end{aligned}
\end{equation}

Finally we deal with the last summand in \eqref{orchi2} as follows.
Using Remark \ref{resto} and Definition \ref{norm} one can reason as in
\eqref{sabbia} and write
\begin{equation}\label{orchi3}
\begin{aligned}
\Pi_{\calX}Q&=\Pi_{\calX}\left(\int_{0}^{1}dt\int_{0}^{t}ds(\Phi_{+})_{*}^{s}\left(\big[g_+,[g_{+},\hat{G}]\big]
+\big[g_{+},\Pi_{K}\Pi_\calX(\hat{F}+u-[g_{+},\Pi_{\calX}^\perp\hat{G}])\big]
\right)
\right)\,.
\end{aligned}
\end{equation}
Regarding the first summand in \eqref{orchi3} one uses \eqref{rem1} and obtains
\begin{equation}\label{orchi6}
\begin{aligned}
C_{\vec{w}_{8},\gotp_{1}}&\left(\Pi_{\calX}\int_{0}^{1}ds\int_{0}^{t}ds(\Phi_{+})_{*}^{s}
\left(\big[g_+,[g_{+},\hat{G}]\big]\right)\right)\leq
C_{\vec{w}_{6},\gotp_1+2}(\hat{G})|g_{+}|_{\vec{w}_{6},\gotp_{1}}|g_{+}|_{\vec{w}_{6},\gotp_{1}+\nu+2} \\
&\stackrel{(\ref{stimatotale2})}{\leq}
\mathtt C K^{\nu+2+2\mu}\Gamma_{\gotp_1}^2\de^2 \left(C_{\vec{w}_{6},\gotp_1+2}(\Pi_{K}\hat{G})+
C_{\vec{w}_{6},\gotp_1+2}(\Pi_{K}^{\perp}\hat{G})\right)
\\
&\stackrel{(\ref{formosa})}{\leq}
{\mathtt C}\g\de^2 \Gamma_{\gotp_1}^2K^{\nu+2\mu+4}\Big(
\Gamma_{\gotp_1}+K^{-(\gotp_2-\gotp_1)}(\Gamma_{\gotp_2} +\e_0 K^{\ka_3}\Gl)
\Big)\,.
\end{aligned}
\end{equation}
Again we are using the fact that $G$ is $C^{\mathtt n+2}$-tame to infer that the double commutator is $C^{\mathtt n}$-tame, and then the projection onto $\calX$ in order to recover the $C^k$-tameness for all $k$.
Regarding the second summand in \eqref{orchi3}, since each term is a polynomial in $E^{(K)}$,
 using Lemmata \ref{conj}, \ref{derivate}-(iii) and the bounds \eqref{gallina5}, \eqref{kam19tris}
we obtain
\begin{equation}\label{falco}
\begin{aligned}
C_{\vec{w}_{8},\gotp_1}(\int_{0}^{1}\int_{0}^{t}&(\Phi_{+})_{*}^{s}[g_{+},\Pi_{K}\Pi_\calX(\hat{F}+u
-[g_{+},\Pi_\calX^\perp\hat{G}])])
\leq \\
&\le \mathtt C |g_+|_{\vec w_6,\gotp_1+1}C_{\vec w_6,\gotp_1+\nu+1}(
\Pi_{K}\Pi_\calX(\hat{F}+u-[g_{+},\Pi_\calX^\perp\hat{G}]))\\
&\le {\mathtt C}\g K^{2\mu+2\nu+4}\Gamma_{\gotp_1}^3\de^2\,.
\end{aligned}
\end{equation} 
Therefore, collecting together \eqref{falco} and \eqref{orchi6} we obtain
\begin{equation}\label{orchi7}
C_{\vec{v}_{+},\gotp_{1}}(\Pi_{\calX}Q)\leq{\mathtt C}\g\de^2 \Gamma^2_{\gotp_1}K^{2\mu+2\nu+4}\Big(
\Gamma_{\gotp_1}+K^{-(\gotp_2-\gotp_1)}(\Gamma_{\gotp_2}+\e_0 K^{\ka_3}\Gl)
\Big).
\end{equation}

In conclusion one has
\begin{equation}\label{orchi10}
\begin{aligned}
C_{\vec{v}_{+},\gotp_1}(\Pi_{\calX}F_{+})&\leq{\mathtt C}\g \Gl K^{\mu+\nu+1}\Big ( \de^2 
\Gamma_{\gotp_1}K^{\mu+\nu+3}\big(
\Gamma_{\gotp_1}+K^{-(\gotp_2-\gotp_1)}(\Gamma_{\gotp_2}+\e_0 K^{\ka_3}\Gl)\big)+\\
& K^{-(\gotp_2-\gotp_1) }  (\Theta_{\gotp_2}+\e_0 K^{\ka_3} \Theta_{\gotp_1} + K^{\al(\gotp_2-\gotp_1)} \de(\Gamma_{\gotp_2}+\e_0 K^{\ka_3} \Gl)\Big)+ \Gl \de  \e_0  K^{-\eta+\mu}.
\end{aligned}
\end{equation}
Recalling that the norm $|\cdot|_{\vec{v}_+,\gotp_1}$ is the sharp tameness constant (see \eqref{tameconst3}), then \eqref{orchi10}
 implies  the bound \eqref{cazzo} since $\de \Gl K^{\mu+\nu+3}\le 1$ by \eqref{kam32}.
\end{proof}

\subsection{Proof of Theorem \ref{thm:kambis}: iterative scheme.}

We now prove Theorem \ref{thm:kambis} by induction on $n$.
The induction basis is trivial with $g_0=0$.
Assuming \eqref{lamorte} up to $n$ we prove the inductive step using the ``KAM step'' of Proposition
\ref{kamstep}.
First of all we ensure that 
\begin{equation}\label{porcocazzo}
\rho_n^{-1} K_n^{\mu+\nu+3}\Gamma_{n,\gotp_1} \de_n \le \mathtt c,
\end{equation}
which, by the inductive hypothesis and \eqref{numeretti} reads
\begin{equation}\label{merdafritta}
2^{n+9}K_0^{(\mu+\nu+3-\ka_2)\chi^n} \mathtt G_0 \e_0 K_0^{\ka_2} \le \mathtt c;
\end{equation}
this is true because
 by \eqref{exp2} and the fact that $K_0$ is large enough depending on $\chi,d,\gotp_0$,
the left hand side \eqref{merdafritta} is decreasing in $n$
so that \eqref{porcocazzo} follows form
$$
K_0^{\mu+\nu+4}\mathtt G_0\e_0 <1
$$
which is indeed implied by \eqref{1s1}
because $\tG_0\ge \tR_0$.

Hence we can apply the ``KAM step''   to $F_n:= (\Phi_n\circ\calL_n)_*F_{n-1}\in \calW_{\vec v^{\text{\tiny 0}}_n,\gotp_2}$ which is a  $C^{\mathtt n+2}$-tame up to order $q=\gotp_2+2$. We fix $(K_{n},\g_n,a_n,s_n,r_n,\rho_n,\calO_n)\rightsquigarrow(K,\g,a,s,r,\rho,\calO)$,
 $\Gamma_{n,p}\rightsquigarrow \Gamma_{p}$,
$\Theta_{n,p}\rightsquigarrow \Theta_{p}$,
$\de_{n}\rightsquigarrow\delta$,
$(\g_{n+1},a_{n+1},s_{n+1},r_{n+1},\rho_{n+1},\calO_{n+1})\rightsquigarrow(\g_+,a_+,s_+,r_+,\rho_+,\calO_+)$. The KAM steps produces a bounded regular vector field $g_{n+1}$ and a left invertible change of variables $\Phi_{n+1}= \uno +f_{n+1}$  such that  
$F_{n+1}:= (\Phi_{n+1}\circ\calL_n)_*F_{n}\in \calW_{\vec v^{\text{\tiny 0}}_{n+1},\gotp_2}$ is  $C^{\mathtt n+2}$-tame up to order $q=\gotp_2+2$.
We  now verify that the bounds \eqref{lamorte} hold with $\Gamma_{n+1,p}\rightsquigarrow \Gamma_{+,p}$,
$\Theta_{n+1,p}\rightsquigarrow \Theta_{+,p}$, ,
$\de_{n+1}\rightsquigarrow\delta_+$.

Let us prove {\bf (i)}.
By substituting into \eqref{stimatotale2} we immediately obtain the bounds for $g_{n+1}$ of \eqref{lamorte}.

Now we recall that, by definition
$$
\frac{\g_n}{\g_{n+1}}=1+\frac{1}{2^{n+3}-1}.
$$

We use \eqref{kam8} together with the inductive hypotheses to obtain   
$$
\Gamma_{n+1,\gotp_1}\leq 
\Big(1+\frac{1}{2^{n+3}-1}\Big)
\tG_n + 2\e_0K_n^{-1}\tG_0+\mathtt C K_n^{\mu-\ka_2} \tG_0 \e_0 K_0^{\ka_2}  (K_n^{\nu+1}\tG_0+ \e_0  K_n^{-\eta})\le \tG_{n+1}\,,
$$
which follow by requiring
$$
\max( 2^nK_n^{-1}\e_0, 2^n K_n^{\mu+\nu+1-\ka_2} K_0^{\ka_2}\tG_0 \e_0  ,  2^n  K_n^{-\eta-\ka_2+\mu } K_0^{\ka_2}\e_0)\le \mathtt c\,,
$$
and as before this follows by \eqref{exp3} and \eqref{1s1}.

Regarding $\Theta_{n+1,\gotp_1}$, using \eqref{kam9} we get
$$
\Theta_{n+1,\gotp_1}\leq \Big(1+\frac{1}{2^{n+3}-1}\Big)
 \tR_n +2\e_0K_n^{-1}\tR_0+ \mathtt C K_n^{\mu+\nu+1-\ka_2} \tG_0^2 \e_0 K_0^{\ka_2}  
 +K_n^{-\eta-\ka_2+\mu}\tG_0 \e^2_0 K_0^{\ka_2}\le \tR_{n+1}.
$$
which again follows from \eqref{exp3} and \eqref{1s1}.

For $\de_{n+1}\rightsquigarrow \de_+$, we apply \eqref{cazzo} and get
\begin{align*}
\de_{n+1}\leq &{\mathtt C} \tG_0 \Big ( \e_0^2 K_0^{\ka_2} (
\tG_0^2  K_0^{\ka_2} K_n^{2\mu+2\nu+4-2\ka_2} +   K_n^{\mu-\eta-\ka_2})+ \\ ( K_n^{\ka_1}+\e_0 K_n^{\ka_3} )&(\tR_0 K_n^{\mu+\nu+1-\Delta\gotp } + 
\e_0 K_0^{\ka_2} \tG_0 K_n^{\mu+\nu+1-(1-\al)\Delta\gotp -\ka_2} )\Big)+( K_n^{\ka_1}+\e_0 K_n^{\ka_3} ) \tR_0 K_n^{-\Delta\gotp +1} \le \e_0 K_0^{\ka_2}K_{n}^{-\chi \ka_2}
\end{align*}
which follows by \eqref{exp2}, \eqref{exp3}, \eqref{exp4}, \eqref{1s1} and \eqref{4s2}.

Regarding $\Gamma_{n+1,\gotp_2}$, by \eqref{kam8} we get
$$
\Gamma_{n+1,\gotp_{2}}\leq  {\mathtt C}\tG_0(K_n^{\ka_1}+ \e_0  K_n^{\ka_3}) \Big( 1+ K_n^{\mu+\nu+2}(\tR_0 + K_n^{\al\Delta\gotp-\ka_2} \e_0 K_0^{\ka_2} \tG_0 )
\Big)\le \tG_0 K_n^{\chi \ka_1}
$$
which follows by \eqref{exp1}, \eqref{exp5} and \eqref{6s2}.
 
Finally, by \eqref{kam9}
$$
\Theta_{n+1,\gotp_2}\leq  \mathtt{C}\tR_0(K_n^{\ka_1}+\e_0 K_n^{\ka_3})+
{\mathtt C}  K_n^{\mu+\nu+2} (K_n^{\ka_1}+\e_0 K_n^{\ka_3})\tG_0 \big( \tR_0 + K_n^{\al\Delta\gotp-\ka_2} \e_0 K_0^{\ka_2} \tG_0\big )\leq \tR_0 K_n^{\chi \ka_1}
$$
which follows again by \eqref{exp1}, \eqref{exp5} and \eqref{6s2}.

\smallskip

We now prove {\bf (ii)}.
Setting $\vec{w}_{4,n}=(\g_{n},\calO_{n+1},s_n-4\rho_{n}s_0,a_n-4\rho_{n}a_0,r_n-4\rho_{n}r_0)$, we prove inductively
\begin{equation}\label{accaenne}
\|(\calH_{n+1}-\uno)(u)\|_{\vec{w}_{4,n},\gotp_1}\le   \e_0\sum_{k=0}^{n+1}\frac{1}{2^{k}}.
\end{equation}
Note that the choice of $\vec w_{4,n}$ \eqref{accaenne} is consistent with the fact that $F_n:=(\calH_{n})_*F_0$  is defined on the domain
$$
\TTT_{s_n}^d \times D_{a_n,p}(r_n)\,.
$$

For $n=-1$ this is obvious. Then by induction we have
\begin{equation}\label{milano5}
\begin{aligned}
\|({\calH}_{n+1}-\uno)(u)\|_{\vec{w}_{4,n},\gotp_1}&\leq\|(\calH_{n}-\uno)(u)\|_{\vec{w}_{4,n},\gotp_1}+
\|f_{n+1}(\calL_{n+1}\circ\calH_{n}(u))\|_{\vec{w}_{4,n},\gotp_1}+\|(\calL_{n+1}-\uno)
\calH_{n}(u)\|_{\vec{w}_{4,n},\gotp_1}\\
&\stackrel{\eqref{accaenne}, (\ref{satana})}{\leq}
  \e_0\sum_{k=0}^{n}\frac{1}{2^{k}}+2|g_{n+1}|_{\vec{w}_{4,n},\gotp_1}
 \|\calL_{n+1}\circ\calH_{n}(u)\|_{\vec{w}_{4,n},\gotp_1}
 +{\mathtt C}\e_{0}K_{n}^{-1}
 \|\calH_{n}(u)\|_{\vec{w}_{4,n},\gotp_1}\\
 &\stackrel{(\ref{stimatotale2})}{\leq} \e_0\sum_{k=0}^{n}\frac{1}{2^{k}}+K_0^{\ka_2}\e_0 \tG_0 K_n^{-\ka_2+\mu+1} + 2 \mathtt C  \e_{0}K_{n}^{-1}
\\
 &\stackrel{\eqref{numeretti}}{\leq}  \e_0\Big(\sum_{k=0}^{n}\frac{1}{2^{k}}+\frac{1}{2^{n+1}}\Big)
\end{aligned}
\end{equation}
 for $K_0$ large enough.
Moreover as before
\begin{equation}\label{milano6}
\|(\calH_{n+1}-\calH_{n})u\|_{\vec{w}_{4,n},\gotp_1}\leq \|(\calL_{n+1}-\uno)\calH_{n}(u)\|_{\vec{w}_{4,n},\gotp_1}+
\|f_{n+1}(\calL_{n+1}\circ\calH_{n})(u)\|_{\vec{w}_{4,n},\gotp_1}\le  \e_0 2^{-(n+1)}
\end{equation}
which implies that the sequence $\calH_n$ is Cauchy and therefore
there exists a limit map ${\calH}_\io=\lim_{n\to\infty}\calH_{n}$. Moreover
\begin{equation}\label{torino}
\begin{aligned}
\!\!\!\!\!\!
\|({\calH}_\io-\uno)(u)\|_{\g_{\infty},\calO_{\infty},\frac{s_0}{2},\frac{a_0}{2},\gotp_1}&\leq\|(\HH_{1}-\uno)(u)\|_{\g_{\infty},\calO_{\infty},\frac{s_0}{2},\frac{a_0}{2},\gotp_1}+
\sum_{n\geq2}\|(\HH_{n}-\calH_{n-1})(u)\|_{\g_{\infty},\calO_{\infty},\frac{s_0}{2},\frac{a_0}{2},\gotp_1}\\
&\leq 2\e_0 \,,
\end{aligned}
\end{equation}
so that for $\e_0$ small enough also \eqref{dominio} holds. 
We are left with the proof of \eqref{fine}.
By definition the limit vector field is
\begin{equation}\label{torino2}
F_{\infty}:=\lim_{n\to\infty}F_{n},\quad F_{n}:=(\Phi_{n}\circ\calL_{n})_{*}F_{n-1}.
\end{equation}
On the other hand we have
$$
\calF_{\infty}:=(\calH_\io)_{*}F_0
$$
We want to prove that $\calF_{\infty}=F_{\infty}$; setting $\calF_{n}:=(\calH_{n})_{*}F_{0}$ we show inductively that
 $\calF_{n}=F_{n}$ for any $n\geq1$. For $n=1$ one has $\calH_{1}=\Phi_{1}\circ\calL_{1}$ and hence $F_{1}=\calF_{1}$. 
 Now assume that $\calF_{n-1}=F_{n-1}$. By definition one has  
 $$
 \calH_{n}=\calK_{n}\circ\calH_{n-1}:=(\Phi_{n}\circ\calL_{n})\circ\calH_{n-1}, \qquad \calH^{-1}_{n}=\calH^{-1}_{n-1}\circ\calK_{n}^{-1}.
 $$
 Hence we have
 \begin{equation}\label{fine2}
 \begin{aligned}
 F_{n}-\calF_{n}&=(\calK_{n})_{*}F_{n-1}-(\calH_{n})_{*}F_0=d\calK_{n}F_{n-1}\circ\calK_{n}^{-1}-d\calH_n F_0\circ\calH^{-1}_{n}\\
 &=d\calK_{n}F_{n-1}\circ\calK_{n}^{-1}-d\calK_{n}d\calH_{n-1}F_0\circ\calH_{n-1}^{-1}\circ\calK_{n}^{-1}\\
 &=d\calK_n(F_{n-1}-d\calH_{n-1}F_0\circ\calH_{n-1}^{-1})\circ\calK_n^{-1}=0.
 \end{aligned}
\end{equation}

This concludes the proof of Theorem \ref{thm:kambis}.
\EP

\zerarcounters
\appendix
\section{Smooth functions and vector fields on the torus}\label{app:toro}

Here we provide some technical results.

The following one is a general result about smooth maps on the torus. First of all, for any $p\ge0$ and $\ze\ge0$
we denote as usual
\begin{equation}\label{space}
H^{p}(\TTT_{s}^{b}; \CCC):=\big\{u=\sum_{l\in\ZZZ^b}u_{l}e^{\ii l\cdot \theta} : 
\|u\|^{2}_{s,p}:=\sum_{l\in\ZZZ^{b}}\langle l\rangle^{2p}|u_{l}|^{2}e^{2s|l|}<\infty\big\}\,,
\end{equation}
the space of functions which are analytic on the strip $\TTT_{s}^b$, Sobolev on its boundary,
and have Fourier coefficients $u_{l}$. By Cauchy formula 
for analytic complex functions we have that this $u$ is uniquely determined by the values that assume on the edge of the domain  i.e $z=x\pm i\s s$
where $\s\in\{\pm 1\}^{b}$. We can define a natural norm using the Sobolev norm of the function on the boundary
\begin{equation}\label{ananorm}
|u|^{2}_{s,p}:=\sum_{\s\in\{\pm1\}^{b}}\int_{\TTT^{b}}\langle\nabla\rangle^{2p}|u(x+i\s s)|^{2}
\end{equation}
Using the Fourier basis it reads
$$
|u|^{2}_{s,p}:=\sum_{\s\in\{\pm1\}^{b}}\sum_{l\in \ZZZ^{b}}\langle l\rangle^{2p}|u_{l}|^{2}e^{-2s\s\cdot l}\,.
$$
\begin{lemma}\label{equi}
The norm $|\cdot|_{s,p}$ and 
\begin{equation}\label{ananorm2}
\|u\|_{s,p}^{2}:=\sum_{l\in\ZZZ^{b}}\langle l\rangle^{2p}|u_{l}|^{2}e^{2s|l|}
\end{equation}
are equivalent.
\end{lemma}

The following Lemma lists some important properties of  Sobolev spaces $H^{s}:=H^{s}(\TTT^{b};\CCC)$ with norm
$$
\|u\|^{2}_{s}:=\sum_{l\in\ZZZ^{b}}\langle l\rangle^{2p}|u_{l}|^{2}.
$$
The same results holds also for our analytic norm in \eqref{space}. The proof of the Lemma is classical.
\begin{lemma}\label{A} Let $s_{0}>d/2$. Then
\begin{enumerate}
\item[(i)] {\bf Embedding.} $\|u\|_{L^{\infty}}\leq C(s_{0})\|u\|_{s_{0}}$, $\forall \; u\in H^{s_{0}}$.
\item[(ii)] {\bf Algebra.} $\|uv\|_{s_{0}}\leq C(s_{0})\|u\|_{s_{0}}\|v\|_{s_{0}}$, $\forall\; u,v\in H^{s_{0}}$.
\item[(iii)] {\bf Interpolation.} For $0\leq s_{1}\leq s\leq s_{2}$, $s=\la s_{1}+(1-\la)s_{2}$,
\begin{equation}\label{A1}
\|u\|_{s}\leq \|u\|^{\la}_{s_{1}}\|u\|_{s_{2}}^{1-\la}, \quad \forall\; u\in H^{s_{2}}.
\end{equation}
\item[(iv)] {\bf Asymmetric tame product.} For $s\geq s_{0}$ one has
\begin{equation}\label{A5}
\|uv\|_{s}\leq C(s_{0})\|u\|_{s}\|v\|_{s_{0}}+C(s)\|u\|_{s_{0}}\|v\|_{s}, \quad \forall\; u,v\in H^{s}.
\end{equation}
%
\item[(vi)] {\bf Mixed norms asymmetric tame product.} For $s\geq0$, $s\in\NNN$  
setting $|u|^{\infty}_{s }:=\sum_{|\al|\leq s}||D^{\al}u||_{L^{\infty}}$
 the norm in $W^{s,\infty}$
one has
\begin{equation}\label{A7}
\|uv\|_{s}\leq\frac{3}{2}\|u\|_{L^{\infty}}\|v\|_{s}+C(s)|u|_{s,\infty}\|v\|_{0},
\forall\; u\in W^{s,\infty}, v\in H^{s}.
\end{equation}
If $u:=u(\la)$ and $v:=v(\la)$ depend in a Lipschitz way on $\la\in\Lambda\subset\RRR^{b}$,
all the previous statements hold if one replace the norms
$\|\cdot\|_{s}$, $|\cdot|^{\infty}_{s }$ with 
$\|\cdot\|_{s,\la}$, $|\cdot|^{\infty}_{s,\la }$ defined as in \eqref{ancoralip}.
\end{enumerate}
\end{lemma}

We now introduce the space
\begin{equation}\label{pio}
W^{p,\io}(\TTT^b_\zeta):=\big\{\be:\TTT^b_\ze\to\TTT^b_\ze \;:\;|\be|_{p,\ze,\io}:=\sum_{k=0}^{p}
\|d^k\be\|_{L^{\infty}(\TTT^{b}_{\zeta})}<\io\big\}\,,
\end{equation}
and note that one has $H^{\ze,p+\gotp_0}(\TTT^b_\ze)\subset W^{p,\io}(\TTT^b_\ze)$.

\begin{lemma}[{\bf Diffeo}]\label{lem.diffeo}
Let $\be\in W^{p,\io}(\TTT^{b}_\ze)$ for some $p,\ze\ge0$ such that
\begin{equation}\label{diffeo2}
\|\be\|_{\ze,\gotp_{0}}\leq 
\frac{\de}{2C_{1}}, \quad \|\be\|_{\ze,\gotp_{0}}\leq\frac{1}{2C_{2}}, \quad 0<\de<\frac{\ze}{2}, \;\; C_{1},C_{2}>0,
\end{equation}
and let us consider $\Phi:\TTT^b_\ze\to\TTT^b_{2\ze}$ of the form
\begin{equation}\label{diffeo}
x\mapsto x+\be(x)=\Phi(x).
\end{equation}
Then the following is true.

\noindent
(i) There exists $\Psi:\TTT^b_{\ze-\de}\to\TTT^b_{\ze}$ of the form $\Psi(y)=y+\tilde{\be}(y)$ with
$\tilde{\be}\in W^{p,\io}(\TTT^b_{\ze-\de})$ 
satisfying
\begin{equation}\label{diffeo3}
\|\tilde{\be}\|_{\ze-\de,\gotp_{0}}\leq\frac{\de}{2}, \quad 
\|\tilde{\be}\|_{\ze-\de,p}\leq2\|\be\|_{\ze,p}\,,
\end{equation}
such that for all $x\in\TTT^b_{\ze-2\de}$ one has $\Psi\circ \Phi(x)=x$.

\noindent
(ii) For all $u \in H^{\ze,p}(\TTT^b_\ze)$, the composition $(u\circ \Phi)(x)=u(x+\be(x))$ satisfies
\begin{equation}\label{diffeo4}
\|u\circ \Phi\|_{\ze-\de,p}\leq C(\|u\|_{\ze,p}+|d\be|_{p-1,\ze,\infty}\|u\|_{\ze,\gotp_0}).
\end{equation}
%

 \end{lemma}

 \prova
 For $\ze=0$ the result is proved in \cite{Ba2} 
thus in the following we assume $\ze>0$.

\noindent
\emph{(i)} First of all recall that, if $\gotp_{0}\geq b/2$ then $\|u\|_{L^{\infty}}\leq \|u\|_{\ze,\gotp_{0}}$.
We look for $\tilde{\be}$ such that
\begin{equation}\label{diffeo44}
\tilde{\be}(y)=-\be(y+\tilde{\be}(y)).
\end{equation}
The idea is to rewrite the problem as a fixed point equation. We define the operator
$\calG : H^{\ze,p}\to H^{\ze,p}$ as $\calG(\tilde{\be})=-\be(y+\tilde{\be})$. First of all we need to show that $\calG$ 
maps the ball $B_{\de/2}:=\{\|u\|_{\ze-\de,p}<\de/2\}$ into itself.

One has
\begin{equation}\label{diffeo5}
\begin{aligned}
\|\calG(\tilde{\be})\|_{\ze-\de,\gotp_{0}}&=\left\| \sum_{n\geq0}\frac{1}{n!}(\del^{n}\be) \tilde{\be}^{n}
\right\|_{\ze-\de,\gotp_{0}}\leq
\sum_{n\geq0}\frac{1}{n!}\|\be\|_{\ze-\de,\gotp_{0}+n}\|\tilde{\be}\|^{n}_{\ze-\de,\gotp_{0}}\,,
\end{aligned}
\end{equation}
where $\del \be$ denotes the derivative of $\be$ w.r.t. its argument.
Note that for any $u\in H^{\ze+\de,s}$ and $\tau>0$ one has
\begin{equation}\label{diffeo6}
\|u\|_{\ze,s+\tau}\leq \left(\frac{\tau}{e}\right)^{\tau}\frac{1}{\de^{\tau}}\|u\|_{\ze+\de,s}\,;
\end{equation}
indeed
\begin{equation*}
\begin{aligned}
\|u\|_{\ze,p+\tau}^{2}&=\sum_{l\in\ZZZ^b}\langle l \rangle^{2(p+\tau)}e^{2\ze|l|}|u_l|^{2}\leq\sum_{l\in\ZZZ^b}
\langle l\rangle^{2p}|l|^{2\tau}e^{-2\de(|l|)}e^{2(\ze+\de)|l|}|b_l|^{2}\,,
\end{aligned}
\end{equation*}
and the function $f(x):=x^{2\tau}e^{-2\de x}$ reach its maximum at $x=\tau/\de$ and $f(\tau/\de)=(\tau/\de e)^{2\tau}$,
so that \eqref{diffeo6} follows.  Then using  \eqref{diffeo6} and the fact that
$n!=(1/\sqrt{2\pi n})(n/e)^{n}(1+O(1/n))$ as $n\to\infty$, we obtain
\begin{equation}\label{diffeo7}
\begin{aligned}
\|\calG(\tilde{\be})\|_{\ze-\de,\gotp_{0}}&\leq\sum_{n\geq0}\frac{1}{n!}\left(\frac{n}{e}\right)^{n}\frac{1}{\de^{n}}
\|\be\|_{\ze,\gotp_{0}}\|\tilde{\be}\|^{n}_{\ze-\de,\gotp_{0}}\leq \|\be\|_{\ze,\gotp_{0}}\sum_{n\geq0}C
\left(\frac{\|\tilde{\be}\|_{\ze-\de,\gotp_{0}}}{\de}\right)^{n}\\
&\leq 2 C \|\be\|_{\ze,\gotp_{0}}\stackrel{(\ref{diffeo2})}{\leq} \frac{\de}{2}\,.
\end{aligned}
\end{equation}
Finally we show that $\calG$ is a contraction. One has
\begin{equation}\label{diffeo8}\
\begin{aligned}
\|\calG(\tilde{\be}_{1})-\calG(\tilde{\be}_{2})\|_{\ze-\de,p}
&=\left\| \sum_{n\geq1}\frac{1}{n!}(\del^{n}\be)\tilde{\be}_{1}^{n}-
\sum_{n\geq1}\frac{1}{n!}(\del^{n}\be)\tilde{\be}_{2}^{n}
\right\|_{\ze-\de,p}\\
&=\left\|\sum_{n\geq1}\frac{1}{n!}(\del^{n}\be)(\tilde{\be}_{1}-\tilde{\be}_{2})
\left(\sum_{k=0}^{n-1}\tilde{\be}_{1}^{k}\tilde{\be}_{2}^{n-1-k}\right)\right\|_{\ze-\de,p}\\
&\le\|\tilde{\be}_{1}-\tilde{\be}_{2}\|_{\ze-\de,p}
\sum_{n\geq1}\frac{1}{n!}\left(\frac{n}{e}\right)^{n}\frac{1}{\de^{n}}\|\be\|_{\ze,p}
\|\sum_{k=0}^{n-1}\|\tilde{\be}_{1}\|^{k}_{\ze-\de,p}\|\tilde{\be}_{2}\|_{\ze-\de,s}^{n-1-k}\\
&\leq \|\tilde{\be}_{1}-\tilde{\be}_{2}\|_{\ze-\de,p} C_{2}\|\be\|_{\ze-\de,p}\stackrel{(\ref{diffeo2})}{\leq} \frac{1}{2}
\|\tilde{\be}_{1}-\tilde{\be}_{2}\|_{\ze-\de,p}	\,.
\end{aligned}
\end{equation}
Then we deduce that there exists a unique fixed point in $B_{\de/2}$, hence a solution of the equation \eqref{diffeo4}.

\noindent
\emph{(ii)} One can follow almost word by word the proof of Lemma 11.4 in \cite{Ba2}
using the norm \eqref{ananorm} instead of \eqref{ananorm2} and  the interpolation properties 
of the $W^{p,\infty}(\TTT^{b}_{\ze})$-norms.
\EP

\begin{rmk}\label{pdimerda}
Note that by Lemma \ref{equi}, one has
$$
\begin{aligned}
&\|f^{(\theta)}(\theta,y,w)\|_{s,a,p}\approx \frac{1}{s_0}\max_{1\le i\le d}
\sum_{\s\in\{\pm1\}^d}\|f^{(\theta_i)}({\rm Re}(\theta)+\ii\s s)\|_{H^{p}}\,,\\
&\|f^{(y)}(\theta,y,w)\|_{s,a,p}\approx \frac{1}{r_0^\mathtt s}\sum_{i=1}^{d_1}
\sum_{\s\in\{\pm1\}^d}\|f^{(y_i)}({\rm Re}(\theta)+\ii\s s)\|_{H^{p}}\,,\\
&\|f^{(w)}(\theta,y,w)\|_{s,a,p}\approx \frac{1}{r_0}
\sum_{\s\in\{\pm1\}^d}\big(\|{\mathtt f}_{\gotp_0}({\rm Re}(\theta)+\ii\s s)\|_{H^p(\TTT^d_s)}+ \|{\mathtt f}_{p}({\rm Re}(\theta)+\ii\s s)\|_{H^{\gotp_0}(\TTT^d_s)}\big)
\end{aligned}
$$
where
$ {\mathtt f}_p(\theta)$ is defined as in \eqref{effepi}.
In particular this means that for all $s\ge0$, $a\ge0$ and $p\ge \ol{\gotp}>n/2$ one has the standard
algebra, interpolation and tame properties w.r.t. composition with functions
in $H^p(\TTT^d_s)$; see for instance \cite{BB1,BBM1,FP,BCP} just to mention a few.
%
\end{rmk}

\noindent
From Lemma \ref{lem.diffeo} and Remark \ref{pdimerda} above we deduce the following result.
 \begin{lemma}\label{compo}
 Given a tame vector field $f\in \VV_{\vec{v},p}$ with scale of constants $C_{p}(f)$
 of the form \eqref{vectorfield} and given 
 a map $\Phi(\theta)=\theta+\be(\theta):\TTT^d_{s'}\to\TTT^d_s$ as in \eqref{diffeo}
 with $b=d$ and $\ze=s$, then 
 the composition $f\circ \Phi$ is a tame vector field with constant 
 \begin{equation}\label{pillole}
 C_{p}(f\circ \Phi)\leq C_{p}(f)+C_{\gotp_0}(f)\|\be\|_{s,p+\nu+3}.
 \end{equation}
 Moreover if $f$ is a regular vector field, i.e. it satisfies \eqref{maremma}, then
  \begin{equation}\label{pillole2}
  |f\circ \Phi|_{\vec{v}_1,p}\leq |f|_{\vec{v},p}+|f|_{\vec{v},\gotp_0}\|\be\|_{s,p+\nu+3}.
  \end{equation}
  where $\vec{v}_1=(\la,\calO,s',a)$.
\end{lemma}
%
\begin{proof}
By Lemma \ref{lem.diffeo} one has that if $\|\be\|_{s,\gotp_{1}}$ is sufficiently small, then the vector field $f\circ\Phi$ is defined on $\TTT^d_{s-\rho s_0}\times D_{a,p}(r-\rho r_{0})$.
Lemma \ref{lem.diffeo} guarantees that for a function $u(\theta)\in \CCC$ the estimate \eqref{diffeo4} holds. Hence also
the components $f^{(v)}(\theta,y,w)$ for $v=\theta,y$ satisfy the same bounds (recall that for the norm \eqref{totalnorm} $y,w$ are  parameters). Let us study the composition
of $f^{(w)}(\theta+\be(\theta),y,w)$.
By Remark \ref{chede} one has 
$$
\|f^{(w)}\|_{s,a,p}=\frac{1}{r_0}(\|{\mathtt f}_{\gotp_0}\|_{s,p}+\|{\mathtt f}_{p}\|_{s,\gotp_0}).
$$
Now ${\mathtt f}_p:\TTT^d_{s}\to \CCC$ and hence we can apply Lemma \ref{lem.diffeo}
to obtain the result. The bounds on the derivatives follow in the same way.
\end{proof}

 \section{Properties of Tame and regular vector fields}\label{tameprop}

\zerarcounters

We now discuss the main properties of $C^k$-tame and regular vector fields; in particular
we need to control the changes in the tameness constants when conjugating via changes
of variables generated by regular bounded vector fields.

\begin{lemma}\label{derivate}
Consider any two $C^k$-\emph{tame} vector fields $F,G\in \mathcal W_{\vec v,p}$, then the following holds.
\begin{itemize}
\item[(i)] 
For $l=1,\ldots,d$ one has  that $\del_{\theta_{l}}F$ is a  $C^k$-tame vector field up to order $q-1$ with
tameness constants $C_{\vec v, p+1}(F)$.
\item[(ii)] For $l=1,\ldots,d$ one has that $\del_{y_{l}}F, d_{w}F[w]$ are 
$C^{k-1}$-tame vector fields up to order $ q$ with tameness constants $C_{\vec v,p}(F)$.
for any $h\geq 0$.
\item[(iii)] The commutator $[G,F]$ is a  $C^{k-1}$-tame vector field up to order $ q-1$ with scale of
constants 
 \begin{equation}\label{commu2}
C_{\vec{v},p}([G,F])\leq 
\mathtt C (C_{\vec{v},p+\nu_G+1}(F)C_{\vec{v},\gotp_{0}+\nu_F+1}(G)+C_{\vec{v},\gotp_{0}+\nu_G+1}(F)C_{\vec{v},p+\nu_F+1}(G)),
 \end{equation}
 where $\nu_F$, $\nu_G$ are the loss of regularity of $F$, $G$ respectively.
  
 \item[(iv)] If $F$ is a polynomial of maximal degree $k$ in $y,w$   then it is $C^{\infty}$-tame up to order $q$.
\end{itemize}
\end{lemma}
\begin{proof} 
Let us check item $(i)$. We consider a map $\Phi:=\uno+f$ as in Definition \ref{tame}. Recall that  $\|f\|_{s,a,\gotp_1}<1/2$. One has that
\begin{equation}\label{polline}
\begin{aligned}
\|(\del_{\theta}F)\circ\Phi\|_{s,a,p}&\leq \|\del_{\theta}(F\circ \Phi)\|_{s,a,p}+\|(\del_{\theta}F)\circ\Phi \cdot \del_{\theta}f\|_{s,a,p}\\
&\leq\|\del_{\theta}(F\circ \Phi)\|_{s,a,p}+
\|(\del_{\theta}F)\circ\Phi\|_{s,a,p}\|\del_{\theta}f\|_{s,a,\gotp_0}+\|(\del_{\theta}F)\circ\Phi\|_{s,a,\gotp_0}\|\del_{\theta}f\|_{s,a,p}.
\end{aligned}
\end{equation}
Now, for $p=\gotp_0$,  by \eqref{polline} one gets
\begin{equation}\label{polline2}
\big(1-2\|\del_{\theta}f\|_{s,a,\gotp_0}\big)\|(\del_{\theta}F)\circ\Phi\|_{s,a,\gotp_0}\leq \|\del_{\theta}(F\circ\Phi)\|_{s,a,\gotp_0}\stackrel{(T1)}{\leq}
C_{s,a,p+1}(F)(1+\|\Phi\|_{s,a,\gotp_0}),
\end{equation}
hence, for $p>\gotp_0$ one has
\begin{equation}\label{polline3}
\|(\del_{\theta}F)\circ\Phi\|_{s,a,p}\leq\mathtt{c}( C_{s,a,p+1}(F)+C_{s,a,\gotp_0}(F)\|\Phi\|_{s,a,p+1}),
\end{equation}
for some $\mathtt{c}$ independent of $p$. Equation \eqref{polline3} implies property $(T1)$ for the vector field $(\del_{\theta}F)(\theta,y,w)$.
Clearly it holds for $p\leq q-1$. The other bounds follows similarly. 
Finally items $(ii),(iii),(iv)$ follow by the definitions.

\end{proof}
\begin{rmk}\label{monomi}
For $k\geq 0$ and any $\tv\in \mathtt V, \tv_1,\ldots,\tv_k\in \mathtt U$  consider any monomial subspace
$\VV^{(\tv,\tv_1,\dots,\tv_k)}$ as in \eqref{sotto}. Then for all $C^k$-tame vector fields $F$ one has that
$\Pi_{\VV^{(\tv,\tv_1,\dots,\tv_k)}}F$ is $C^\io$-tame (up to the same order  as $F$) and one can
choose the constant as
$C_{\vec v,p}(\Pi_{\VV^{(\tv,\tv_1,\dots,\tv_k)}} F)=  C_{\vec v,p}(F)$. The same holds for
the direct sum $\calU$ of a finite number of monomial spaces
and their orthogonal, namely one can chose
$$
C_{\vec v,p}(\Pi_{\calU} F) = C_{\vec v,p}(F),\qquad
C_{\vec v,p}(\Pi_{\calU}^\perp F) = C_{\vec v,p}(F)\,.
$$
\end{rmk}

\begin{lemma}[{\bf Conjugation}]\label{conj}  
Consider  a tame left invertible map $\Phi=\uno+f$  with tame inverse $\Psi= \uno+h $ 
as in Definition \ref{leftinverse} such that \eqref{sinistra} holds.
Assume that $\gotp_1\ge\gotp_0+\nu+1$ and the fields $f, h$ are such that
$C_{\vec{v},\gotp_1}(f)=C_{\vec{v},\gotp_1}(h)\leq {\mathtt c}\rho$ for $\rho>0$ and $\mathtt c$ the same
appearing in Remark \ref{azz} \footnote{By Remark \ref{azz} the smallness
of the constants $C_{\vec{v},\gotp_1}(f),C_{\vec{v},\gotp_1}(h)$ automatically implies
that $\Phi,\Psi$ satisfy \eqref{laputtana250}.}. 
For any   vector field 
\begin{equation}\label{miodio}
F : \TTT^d_s\times D_{a,p+\nu}(r) \times \calO \to V_{a,p},
\end{equation}  which is $C^k$--tame up to order $q$, one has that the push--forward
\begin{equation}\label{push2}
F_+:=\Phi_* F: \TTT^d_{s- 2\rho s_0}\times D_{a,p+\nu}(r-2\rho r_0)\times \calO  \to  V_{a- 2\rho a_0,p}
\end{equation}
is $C^k$--tame up to order $q-\nu-1$, with scale of constants
 \begin{equation}\label{dafare}
 C_{\vec{v}_{2},p}(F_+)\leq (1+\rho)\Big(C_{\vec{v},p}(F)+C_{\vec{v},\gotp_{0}}(F)
 C_{\vec{v}_{1},p+\nu+1}(f)\Big),
 \end{equation}
 where $\vec{v}:=(\la,\calO,s,a,r)$, $\vec{v}_{1}:=(\la,\calO,s-\rho s_0,a-\db a_0, r-\rho r_0)$ and 
 $\vec{v}_{2}:=(\la,\calO,s-2\rho s_0,a-2\db a_0, r-2\rho r_0)$.
\end{lemma}
 \begin{proof}
By \eqref{push} 
the vector field $F_{+}$ is defined in $\TTT_{s-2\rho s_0}\times D_{a,p}(r-2\rho r_{0})\times \calO$.
Then, given a change of coordinates
$\Gamma:\TTT^d_{s_1}\times D_{a',p'}(r_1)\times \calO
\rightarrow \TTT_{s-2\rho s_0}\times D_{a,p}(r-2\rho r_{0})$
we can consider the composition of $F_{+}$ with $\Gamma$; in particular
$$
\Psi\circ \Gamma: \TTT^d_{s_1}\times D_{a',p'}(r_1)\times \calO\longrightarrow
\TTT_{s}\times D_{a,p+\nu}(r ),
$$
namely the domain of $F$.
Let us check the property $($T$_0)$ for the vector field $F_{+}$.
In the following we will keep track only of the index $p$.
One has
\begin{equation}\label{topolino2}
\begin{aligned}
&\|\Psi(\Gamma)\|_{p}\leq C_{p}(f)+(1+C_{\gotp_{0}}(f))\|\Gamma\|_{p},\\
&\|\Psi(\Gamma)\|_{\gotp_{0}+\nu}\leq 1+2 C_{\gotp_{0}+\nu}(f).
\end{aligned}
\end{equation}
so that we get
\begin{equation}\label{topolino}
\begin{aligned}
\|F_{+}(\Gamma)\|_{p}&\leq \|F(\Psi(\Gamma))\|_{p}+\|{ d}f(\Psi(\Gamma))[F(\Psi(\Gamma))]\|_{p}\\
&\stackrel{(\text{T}_1)}{\leq} (1+C_{\gotp_{0}+1}(f))\|F(\Psi(\Gamma))\|_{p}+
\left(C_{p+1}(f)+C_{\gotp_{0}+1}(f)\|\Psi(\Gamma)\|_p\right)\|F(\Psi(\Gamma))\|_{\gotp_{0}}\\
&\stackrel{(\text{T}_0)}{\leq} (1+C_{\gotp_{0}+1}(f))\left[C_{p}(F)+C_{\gotp_{0}}(F)\|\Psi(\Gamma)\|_{p+\nu}
\right]\\
&\qquad\qquad+\left(C_{p+1}(f)+C_{\gotp_{0}+1}(f)\|\Psi(\Gamma)\|_{p}\right)
\left[C_{\gotp_{0}}(F)+C_{\gotp_{0}}(F)\|\Psi(\Gamma)\|_{\gotp_{0}+\nu}
\right]\,,
\end{aligned}
\end{equation}
and therefore
\begin{equation}\label{topolino3}
\begin{aligned}
\|F_{+}(\Gamma)\|_{p}&\leq C_{p}(F)(1+C_{\gotp_{0}+\nu+1}(f))+5 C_{\gotp_{0}}(F)(1+C_{\gotp_{0}+1}(f))C_{p+\nu+1}(f)\\
&+\|\Gamma\|_{p+\nu}\left[ C_{\gotp_{0}}(F)(1+3C_{\gotp_{0}+\nu+1}(f))^{2}\right],
\end{aligned}
\end{equation}
that is  $($T$_0$). The other properties are obtained with similar calculations using also the fact that 
the vector field $f$ is linear in the variables $y,w$. Hence
$F_{+}$ is tame with scale of constants in \eqref{dafare}.
 \end{proof}
\begin{rmk}
In Lemma \ref{conj}, if  $f,h\in E^{(K)}$, then  the smoothing estimate \eqref{P1} applied to $|f|_{\vec v,p+\nu+1}$ implies that  $F_{+}$ is $C^k$--tame up to order $q$. The same holds if $\Phi$ is generated by a vector field $g\in E^{(K)}$.
\end{rmk}

\begin{rmk}\label{resto}
Consider a vector field $g$ which generates a well defined flow $\Phi^t$ for $t\leq 1$ and set $\Phi:=\Phi^1$.
Then, for all $p\ge\gotp_0$ for all  vector fields $F $  such that the push-forward  with $\Phi^t$ is well defined, 
one has
\begin{equation}\label{res1}
\begin{aligned}
L:= \Phi_{*}F-F&=\int_{0}^{1}\Phi_{*}^{t}[g,F]d t,\\
Q:= \Phi_{*}F-[g,F]-F&=\int_{0}^{1}\int_{0}^{t}\Phi_{*}^{s}[g,[g,F]]d s d t.
\end{aligned}
\end{equation}
If moreover  $g\in \BB$ satisfies Definition \ref{linvec-abs} item 5,  and $F$ is as in formula \eqref{miodio}  then
$L$ is $C^{k-1}$ tame and $Q$ 
is  $C^{k-2}$ tame up to order $q-\nu-1$ with constants
\begin{subequations}
\begin{align}
& C_{\vec{v_2},p}(L)\leq  \mathtt C \Big(
C_{\vec{v},p+1}(F)|g|_{\vec{v},\gotp_{0}+\nu+1}+C_{\vec{v},\gotp_{0}+1}(F)|g|_{\vec{v},p+\nu+1}\Big) \label{rem0} \\
&  C_{\vec{v_2},p}(Q)\leq  \mathtt C \Big(
 C_{\vec{v},p+2}(F)|g|^2_{\vec{v},\gotp_{0}+\nu+1}+C_{\vec{v},\gotp_{0}+2}(F)|g|_{\vec{v},\gotp_{0}+\nu+2}|g|_{\vec{v},p+\nu+2} \Big) \label{rem1}
\end{align}
\end{subequations}
$\vec{v}:=(\la,\calO,s,a,r)$, $\vec{v}_{2}:=(\la,\calO,s-2\rho s_0, a, r-2 \rho r_0)$.
Finally if $g\in E^{(K)}$ then $L,Q$ are tame up to order $q$.
%
 \end{rmk}

\begin{lemma}\label{normnorm}
Consider any subspace $\calU$ which is a finite direct sum of {\em monomial} subspaces
 as in formula \eqref{sotto}, having degree at most $k$, or their average in $\theta$. 
 Under the hypotheses of Lemma \ref{conj}, assume also that\footnote{hence they are both $C^{\io}$-tame.}
$F\in \calU$ and $f\in\BB$. 
Then for all $p\ge\gotp_0$ one has
\begin{equation}\label{uffa}
C_{\vec{v}_{2},p}(\Pi_{\calU}^\perp \Phi_{*}F)\leq {\mathtt C}\Big(
C_{\vec{v},p}(F)\rho+C_{\vec{v},\gotp_1}(F)|f|_{\vec{v}_{1},p+\nu+1}\Big)
\end{equation}
 where $\vec{v}:=(\la,\calO,s,a)$, $\vec{v}_{1}:=(\la,\calO,s-\rho s_0,a)$ and 
 $\vec{v}_{2}:=(\la,\calO,s-2\rho s_0,a)$.
 Moreover if $f\in E^{(K)}$ then \eqref{uffa} holds up to order $q$.
 \end{lemma}
 \begin{proof}
One has
\begin{equation}\label{uffa2}
\Pi_{\calU}^\perp\Big(\Phi_{*}F\Big)=
\Pi_{\calU}^\perp\Big(F\circ\Psi -F +d_{u}f(\Psi)[F\circ\Psi]\Big)\,.
\end{equation}
The last summand clearly satisfies the estimates \eqref{uffa}.
Regarding the first terms we write
$$
\Pi_{\calU}^\perp(F\circ\Psi -F)= \Pi_{\calU}^\perp \big( d_\theta F [f^{(\theta)}] + \sum_{\mathtt u= y,w} d_{\mathtt u} F[f^{(\mathtt u)}]\big)\,;
$$
the second term satisfies \eqref{uffa} by property ($T_1$),
 while we claim that $d_\theta F [f^{(\theta)}]\in\calU$. Indeed if $F\in \VV^{(\tv,\tv_1,\ldots,\tv_h)}$,
 it has the form
 $$
 F= \frac{1}{\al(\tv_1,\ldots,\tv_h)!}\Big(\prod_{i=1}^h d_{\tv_i}\Big)F^{(\tv)}(\theta,0,0)[\tv_1,\ldots,\tv_h],
 $$
so that deriving w.r.t. $\theta$ on both sides and computing at $g^{(\theta)}$ commutes with the
derivation in $\tv_i\in {\mathtt U}$ (this follows from the fact that if $f\in \BB$ then $f^{(\theta)}(\theta,y,w)= f^{(\theta)}(\theta,0,0)$). If $\calU$ is only the average of some monomial space, then
clearly its $\theta$-derivative is zero.
\end{proof}

\subsection{Proof of Lemma \ref{grrr}}
 \begin{lemma}\label{stolemma}
  All regular vector fields $f$ as in  Definition \ref{linvec} are  $C^{\infty}$-tame up to order $q$  with tameness constant 
  \begin{equation}\label{tameconst}
  C_{\vec{v},p}(f)= C_{d,q} |f|_{\vec{v},p}.
   \end{equation}
 \end{lemma}
\begin{proof}
In view of Lemma \ref{derivate}--(iv), we only need to prove that a regular vector field is $C^1$-tame.
 Consider a regular vector field $f$ (see Definition \ref{linvec}) and 
a map $\Phi=\uno+g$ as in Definition \ref{tame}. For simplicity we drop the indices $\vec{v},\vec{v}_1,\vec{v}_2$.
Without loss of generality we can also assume that $g^{(\theta)}$ depends only on $\theta$, since in $($T$_m)$ we first perform the $y,w$-derivatives
and then compute at $\Phi = \uno+g$. 
 Let us check $($T$_0)$ for $f$. One has
 \begin{equation*}
 \begin{aligned}
(f\circ \Phi)^{(\theta)}&:=h^{(\theta,0)}(\theta), \qquad (f\circ \Phi)^{(w)}:=h^{(w,0)}(\theta),\\
(f\circ \Phi)^{(y)}&:=h^{(y,0)}(\theta)+h^{(y,y)}(\theta)\Phi^{(y)}(\theta,y,w)+h^{(y,w)}(\theta)\cdot \Phi^{(w)}(\theta,y,w),\\
\end{aligned}
 \end{equation*}
 where
 $$
 h^{(v,v')}(\theta):=f^{(v,v')}(\theta+g^{(\theta,0)}(\theta)), \quad v,v'=0,\theta,y,w.
 $$
 We first give bounds on the norm of $f\circ \Phi$ in terms of the norms of $h$ and $\Phi$.
The  terms depending only on $\theta$ are trivially bounded by the norm of $h$.
 In the $y$-component one has
 \begin{equation}\label{sto1}
 \begin{aligned}
 \|h^{(y,y)}\Phi^{(y)}\|^{2}_{s,a,p}&\leq C(d)\frac{1}{r_{0}^{2\mathtt s}} \sum_{\ell\in\ZZZ^{d}}\sum_{i=1}^{d_1}\sum_{k=1}^{d_1}|(h^{(y_{i},y_{k})}g^{(y_{k})})(\ell)|^{2}e^{2s|\ell|}\langle \ell\rangle^{2p}\\
 &=C(d)\frac{1}{r_{0}^{2\mathtt s}}\sum_{i=1}^{d_1}\sum_{k=1}^{d_1}\|h^{(y_{i},y_{k})}(\theta)\Phi^{(y_{k})}(\theta)\|^{2}_{s,p}\\
 &\stackrel{(\ref{A5})}{\leq} C(d)\frac{1}{r_{0}^{2\mathtt s}}\sum_{i=1}^{d_1}\sum_{k=1}^{d_1}(\|h^{(y_{i},y_{k})}\|_{s,p}\|\Phi^{(y_{k})}\|_{s,\gotp_{0}}+
 \|h^{(y_{i},y_{k})}\|_{s,\gotp_{0}}\|\Phi^{(y_{k})}\|_{s,p} )^{2},
 \end{aligned}
 \end{equation}
 \noindent
 hence one obtains
 \begin{equation}\label{sto2}
 \|h^{(y,y)}\Phi^{(y)}\|_{s,a,p}\leq K \left(r_{0}^{\mathtt s}\|h^{(y,y)}\|_{s,a,p}\|\Phi^{(y)}\|_{s,a,\gotp_{0}}+r_{0}^{\mathtt s}\|h^{(y,y)}\|_{s,a,\gotp_{0}}\|\Phi^{(y)}\|_{s,a,p}\right).
 \end{equation}
  Finally  one has
   \begin{equation}\label{sto3}
   \begin{aligned}
   \|h^{(y,w)}\Phi^{(w)}\|^{2}_{s,a,p}&\leq C\frac{1}{r_{0}^{2\mathtt s}}\sum_{i=1}^{d_1}\|h^{(y_{i},w)}\Phi^{(w)}\|^{2}_{s,p}=
   C \frac{1}{r_{0}^{2\mathtt s}}\sum_{i=1}^{d_1}\sum_{l\in\ZZZ^{d}}\langle l\rangle^{2p}e^{2s|l|}|(h^{(y_{i},w)}\cdot\Phi^{(w)})(l)|^{2}\\
   &\leq C\frac{1}{r_{0}^{2\mathtt s}}\sum_{i=1}^{d_1}\sum_{l\in\ZZZ^{d}}\langle l\rangle^{2p}e^{2s|l |}(\sum_{k\in\ZZZ^{d}}|h^{(y_{i},w)}(l-k)\cdot\Phi^{(w)}(k)|)^{2}\\
   &\leq \frac{C}{r_{0}^{2\mathtt s}}\sum_{i=1}^{d_1}
   \sum_{l,k\in\ZZZ^{d}}\langle l-k\rangle^{2p}e^{2s|l-k|}\langle k\rangle^{2\gotp_{0}}e^{2s|k|}|h^{(y_{i},w)}(l-k)\cdot\Phi^{(w)}(k)|^{2}\\
   &\qquad+ \frac{C}{r_{0}^{2\mathtt s}}\sum_{i=1}^{d_1}
   \sum_{l,k\in\ZZZ^{d}}\langle l-k\rangle^{2\gotp_{0}}e^{2s|l-k|}\langle k\rangle^{2p}e^{2s|k|}|h^{(y_{i},w)}(l-k)\cdot\Phi^{(w)}(k)|^{2}\\
   &\stackrel{(\ref{sonasega})}{\leq}
   \frac{C}{r_{0}^{2\mathtt s}}\sum_{i=1}^{d_1}
   \sum_{l,k\in\ZZZ^{d}}\langle l-k\rangle^{2p}e^{2s|l-k|}\langle k\rangle^{2\gotp_{0}}e^{2s|k|}
   \|h^{(y_{i},w)}(l-k)\|^{2}_{-a,-\gotp_0-\nu}\|\Phi^{(w)}(k)\|_{a,\gotp_{0}+\nu}^{2}\\
   &\qquad+ \frac{C}{r_{0}^{2\mathtt s}}\sum_{i=1}^{d_1}
   \sum_{l,k\in\ZZZ^{d}}\langle l-k\rangle^{2\gotp_0}e^{2s|l-k|}\langle k\rangle^{2p }e^{2s|k|}
   \|h^{(y_{i},w)}(l-k)\|^{2}_{-a,-\gotp_{0}-\nu}\|\Phi^{(w)}(k)\|_{a,\gotp_0+\nu}^{2},
   \end{aligned}
   \end{equation}
   where in the third line we used the standard interpolation estimates and the fact that $\gotp_{0}>d/2$.
  By \eqref{sto3}, since $p\leq q$, it follows that 
    \begin{equation*}
  \|h^{(y,w)}\Phi^{(w)}\|_{s,a,p}\leq C\left(
  r_{0}\| h^{(y,w)}\|_{H^{p}(\TTT^{d}_{s};\ell_{-a,-\gotp_0-\nu})}
  \|\Phi^{(w)}\|_{s,a,\gotp_{0}+\nu}+r_{0}\|h^{(y,w)}\|_{H^{\gotp_0}(\TTT^{d}_{s};\ell_{-a,-\gotp_0-\nu})}\|\Phi^{(w)}\|_{s,a,p}
  \right).
  \end{equation*}
Now each component  $h^{(v,v')}$ is a function  $\TTT^d_s\to V_{a,p}$ composed with a diffeomorphism of the torus given by 
$\theta \mapsto \theta+g^{(\theta,0)}(\theta)$. Hence we obtain, by using 
 Lemma \ref{lem.diffeo}$(ii)$ and Lemma \ref{compo}, 
\begin{equation}\label{sto8}
\|f\circ\Phi\|_{s,a,p}\leq C(d,q)( |f|_{s,a,p}+|f|_{s,a,\gotp_{0}}\|\Phi\|_{s,a,p+\nu}),
\end{equation}
 The property $($T$_1)$ follows in the same way.
  \end{proof}

 \begin{lemma}\label{cometichiami}
Consider a vector field
  $f\in \mathcal B$ such that
  \begin{equation}\label{puffetta}
f :\TTT^{d}_{s}\times D_{a,p}(r )\times\calO\to V_{a,p}
 \end{equation}
 and 
 \begin{equation}\label{pippo}
 |f|_{\vec{v},\gotp_1}\leq {\mathtt c}{\rho},
 \end{equation}
 for some $\rho>0$.
 If $\rho$ is small enough, then for all $\x\in \calO$
 the following holds.
 
  \noindent
 (i) The map $\Phi:=\uno+f$ is such that
 \begin{equation}\label{speriamobene}
 \Phi: \TTT^{d}_{s}\times D_{a,p}(r)\times \calO\longrightarrow \TTT^{d}_{s+\rho s_0}\times D_{a,p}(r+\rho r_0).
 \end{equation}

 \noindent
 (ii) There exists a vector field $h\in \BB$ such that
 \begin{itemize}
 \item $|h|_{\vec{v}_1,p}\leq 2|f|_{\vec{v},p}$, 
 the map $\Psi:=\uno+h$ is such that
 \begin{equation}\label{nome00}
 \Psi : \TTT^{d}_{s-\rho s_0}\times D_{a,p}(r-\rho r_0)\times \calO\to \TTT^{d}_{s}\times  D_{a,p}(r ).
 \end{equation}
 
 \item  for all $(\theta,y,w)\in \TTT^d_{s-2\rho s_0}\times D_{a,\gotp_1}(r-2\rho r_0)$ one has
 \begin{equation}\label{nome}
 \Psi\circ\Phi(\theta,y,w)=(\theta,y,w).
 \end{equation}
 \end{itemize}
 \end{lemma}
 \begin{proof}
\emph{(i)} We want to bound the components of $\Phi=\uno+f$. First of all we see that for $\theta\in\TTT^{d}_s$
one has
\begin{equation}\label{pippo2222}
|\Phi^{(\theta)}|_{\infty}\leq s+|f^{(\theta)}|_\infty\leq s+
\|f^{(\theta)}\cdot\del_\theta\|_{s,a,\gotp_0}\stackrel{(\ref{pippo})}{\leq}s+\rho s_0\,,
\end{equation}
where we used the standard Sobolev embedding Theorem.
The bound on $\|\Phi^{(w)}\|_{a,\gotp_0}\leq r+\rho r_0$ follows directly by hypothesis \eqref{pippo}.
In order to obtain the estimates on the $y-$components
we need to check that 
\begin{equation}\label{pippo3}
\begin{aligned}
|f^{(y,0)}(\theta)|_1\le {\mathtt c}^{-1}\|f^{(y,0)}&(\theta)\cdot\del_{y}\|_{s,a,\gotp_0}, \qquad
|f^{(y,y)}(\theta)y|_{1}\leq {\mathtt c}^{-1}\|f^{(y,y)}(\theta)y\cdot\del_{y}\|_{s,a,\gotp_0}, \\
&|f^{(y,w)}(\theta)w|_{1}\leq {\mathtt c}^{-1}\|f^{(y,w)}(\theta)w\cdot\del_{y}\|_{s,a,\gotp_0}\,.
\end{aligned}
\end{equation}
Since for a $d$-dimensional vector $\mathtt v$ one has $|\mathtt v|_1\le d|\mathtt v|_{\io}$ we get
\begin{equation}\label{pippom}
|f^{(y,w)}(\theta)\cdot w|_1\le d_1 \max_{v=y_1,\ldots,y_{d_1}} \|f^{(v,w)}(\theta)\cdot w\|_\io
\le K(n,\gotp_0)\| f^{(y,w)}(\theta)\cdot w\|_{{s,\gotp_0}}\,.
\end{equation}
The other bounds in \eqref{pippo3} follow in the same way. The extension of the bounds for the Lipschitz norm
is standard; see for instance \cite{FP}.
Thus we obtain
$|\Phi^{(y)}|_{1}\leq 
(r+\rho r_0)^{2}$ so that \eqref{speriamobene} follows.

\noindent
\emph{(ii)} The first $d$ components of the map
 $(\theta_{+},y_{+},w_{+})=\Phi(\theta,y,w)$ are  $\theta_{+}=\theta+f^{(\theta)}(\theta)$.
 If $\rho$ is small enough we can apply Lemma \ref{lem.diffeo} in order to define an inverse map
 $h^{(\theta)}(\theta_{+})\in W^{p,\infty}(\TTT^{d}_{s-\rho s_{0}})$ with
 $\|h^{(\theta)}\|_{s-\rho s_{0},p}\leq2\|f^{(\theta)}\|_{s,p}$.
 Hence we set
 \begin{equation}\label{nuovo}
 \Psi^{(\theta)}(\theta_{+}):=\theta_{+}+{h^{(\theta)}}(\theta_{+})\,,
 \qquad
 \theta_{+}\in\TTT^{d}_{s-\rho s_{0}}.
 \end{equation}
Regarding the other components we first solve $y,w$ as functions of $y_+,w_+,\theta$ and then substitute \eqref{nuovo}.
We have
$$
\begin{aligned}
w &= w_+- f^{(w,0)}(\Psi^{(\theta)}(\theta_{+})) \\
y&= (\uno - f^{(y,y)}(\Psi^{(\theta)}(\theta_{+})))^{-1}(y_+ -f^{(y,0)}(\Psi^{(\theta)}(\theta_{+})) -f^{(y,w)}(\Psi^{(\theta)}(\theta_{+}))\cdot (w_+-f^{(w,0)}(\Psi^{(\theta)}(\theta_{+}))))
\end{aligned}
$$
which fixes the remaining components of $h$. The estimates on the norm of $h$ follow by Lemma \ref{lem.diffeo} $(ii)$  and by Lemma \ref{compo}.
\end{proof}

\begin{lemma}\label{flusso}
Given any regular bounded vector field $g\in \mathcal B$, $p\geq \gotp_1$  with $| g|_{\vec{v},\gotp_1}\leq {\mathtt c} \rho$
 then  for $0\leq t\leq1$ there exists $f_{t}\in \mathcal B$ such that the time$-t$ map of the flow of $g$ is of the form 
  $\uno+ f_{t}$ moreover we have $|f_{t}|_{\vec{v},p} \leq  2| g|_{\vec{v}_1,p}$ where $\vec v_1= (\lambda,\calO,s-\rho s_0,a,r)$.
\end{lemma}

\begin{proof}
The dynamical system associated with $g$ is
\begin{subequations}
\begin{align}
&\dot\theta = g^{(\theta,0)}(\theta),
\label{theta}\\
&\dot{y}= g^{(y,0)}(\theta)+g^{(y,y)}(\theta)y+g^{(y,w)}(\theta)\cdot w,
\label{y}\\
&\dot w= g^{(w,0)}(\theta).
\label{w}
\end{align}
\end{subequations}
We solve first \eqref{theta}, then substitute into \eqref{w} and finally substitute both into \eqref{y} and
hence the result follows by proving that the solution of \eqref{theta}, with initial datum $\f$, has the form
$$
\theta(t)=\f+h(t,\f),
$$
with $h\in H^{p}(\TTT^d_{s-\rho s_0})$ a zero-average  function. This latter statement follows by the standard
theory of existence, uniqueness and smoothness w.r.t. the initial data. 
\end{proof}

\zerarcounters

\section{Proof of Proposition \ref{uffffa}}\label{app:homo}
\begin{proof}
Given a tame vector field  $F\in  \VV_{\vec{v},p}$ such that $F\in \calE$
for all $\xi\in \calO$, let us define
$$
\gotA:=\gotN+\gotR:=\Pi_K\Pi_\calX([{\Pi_\NN F},\cdot])+\Pi_K\Pi_\calX([\Pi_\RR F,\cdot]).
$$
We note that 
$\gotN,\gotR: E^{(K)}\cap \BB_\calE \to E^{(K)}\cap \calX\cap\calE$.

Then the ``approximate invertibility'' of $\gotN$
implies the ``approximate invertibility'' of $\gotA$.
Indeed let $\gotW:\, E^{(K)}\cap \calX\cap\calE \to E^{(K)}\cap \BB_\calE$ be  the
 ``approximate right inverse'' of $\gotN$ defined in \eqref{buoni} and
 denote 
 $$
 \gotU:=  \gotR \gotW: E^{(K)}\cap \calX\cap\calE \to E^{(K)}\cap \calX\cap\calE. 
 $$
 By \eqref{blodia} and \eqref{sodia} we have that   $\gotU$ is strictly upper triangular so
$
\gotU^{\mathtt b}=0.
$
Now we set  $\gotB=  \gotN\gotW-\uno$ which is ``small'' in the sense of \eqref{cribbio4}.
Then we have
\begin{equation}\label{corinzi7}
(\gotN+\gotR)\gotW(\uno+\gotU)^{-1}= (\uno+\gotU+\gotB)(\uno+\gotU)^{-1} = \uno+ \gotB (\uno+\gotU)^{-1}
\,,\quad
(\uno+\gotU)^{-1}= \sum_{j=0}^{\mathtt b-1} (-1)^j \gotU^j\,.
\end{equation}
Thus $\gotW(\uno+\gotU)^{-1}$ is an approximate inverse for $\gotA$ in the sense that it is a true inverse for $\gotB=0$.
Then for all $\xi\in \calO$ let us set 
\begin{equation}\label{primig}
\tilde g:= \gotW(\uno+\gotU)^{-1}\Pi_K\Pi_X F.
\end{equation}
As for the bounds we first notice that by \eqref{buoni} and \eqref{vino2} one has
\begin{equation}\label{questanuo}
|\gotU X|_{\vec{v},p}\leq K^{\mu_{1}+\nu+1 }\big[ \Theta_{\gotp_1} |X|_{\vec{v},p}+
\big(\Theta_{p}(1+\Gamma_{\gotp_1})+ \Theta_{\gotp_1}K^{\al(p-\gotp_1)}\Gamma_{p}\big)|X|_{\vec{v},\gotp_1}\big].
\end{equation}

Now we can prove inductively that 
\begin{equation}\label{indu1}
|\gotU^jX|_{\vec v,p}\le K^{j(\mu+\nu+1)}\Big[\Theta_{\gotp_1}^j|X|_{\vec{v},p}+
\big(\Theta_{p}(1+\Gamma_{\gotp_1})+ \Theta_{\gotp_1}K^{\al(p-\gotp_1)}\Gamma_{p}\big)
{\mathtt P}_j(\Theta_{\gotp_1},\Gamma_{\gotp_1})|X|_{\vec{v},\gotp_1} \Big],
\end{equation}
where ${\mathtt P}_j(\Theta_{\gotp_1},\Gamma_{\gotp_1})$ is a polynomial of degree $2(j-1)$ defined recursively as
\begin{equation}\label{recu}
\begin{aligned}
&{\mathtt P}_1:=1\,,\\
&{\mathtt P}_j(\Theta_{\gotp_1},\Gamma_{\gotp_1}):= \Theta_{\gotp_1}^{j-1}+2\Theta_{\gotp_1}(1+\Gamma_{\gotp_1})
{\mathtt P}_{j-1}(\Theta_{\gotp_1},\Gamma_{\gotp_1}).
\end{aligned}
\end{equation} 

Indeed for $j=1$ this is exactly the bound \eqref{questanuo}; then assuming \eqref{indu1} to hold up to $j$ we have
\begin{equation}\label{indu2}
\begin{aligned}
|\gotU^{j+1}X|_{\vec v,p} &= |\gotU(\gotU^{j}X)|_{\vec v,p}\le K^{\mu+\nu+1 }
\big[ \Theta_{\gotp_1} |\gotU^{j}X|_{\vec{v},p}+\big(\Theta_{p}(1+\Gamma_{\gotp_1})+ 
\Theta_{\gotp_1}K^{\al(p-\gotp_1)}\Gamma_{p}\big)|\gotU^{j}X|_{\vec{v},\gotp_1}\big]\\
&\!\!\!\!\!\le  K^{\mu+\nu+1 }\Big( \Theta_{\gotp_1} K^{j(\mu+\nu+1)}\Big[\Theta_{\gotp_1}^j|X|_{\vec{v},p}+
\big(\Theta_{p}(1+\Gamma_{\gotp_1})+ \Theta_{\gotp_1}K^{\al(p-\gotp_1)}\Gamma_{p}\big)
{\mathtt P}_j(\Theta_{\gotp_1},\Gamma_{\gotp_1})|X|_{\vec{v},\gotp_1} \Big]\\
&+
\big(\Theta_{p}(1+\Gamma_{\gotp_1})+ \Theta_{\gotp_1}K^{\al(p-\gotp_1)}\Gamma_{p}\big) 
K^{j(\mu+\nu+1)}\Big[\Theta_{\gotp_1}^j+
\Theta_{\gotp_1}(1+2\Gamma_{\gotp_1})
{\mathtt P}_j(\Theta_{\gotp_1},\Gamma_{\gotp_1})\Big]|X|_{\vec{v},\gotp_1} \Big)\\
&\!\!\!\!\!=K^{(j+1)(\mu+\nu+1)}\Big(
\Theta^{j+1}|X|_{\vec{v},p}\\
&+
\big(\Theta_{p}(1+\Gamma_{\gotp_1})+ \Theta_{\gotp_1}K^{\al(p-\gotp_1)}\Gamma_{p}\big)
\big(
\Theta_{\gotp_1}^j+
2\Theta_{\gotp_1}(1+\Gamma_{\gotp_1})
{\mathtt P}_j(\Theta_{\gotp_1},\Gamma_{\gotp_1})
\big)|X|_{\vec{v},\gotp_1} 
\Big)
\end{aligned}
\end{equation}
which is \eqref{indu1} for $j+1$ taking into account \eqref{recu}.
Moreover, again by induction, the polynomials ${\mathtt P}_j$ satisfy the bound
\begin{equation}\label{eddaje}
|{\mathtt P}_j|\le 3^j \Theta_{\gotp_1}^{j-1}(1+\Gamma_{\gotp_1})^{j-1}\,,
\end{equation}
uniformly in $\Theta_{\gotp_1},\Gamma_{\gotp_1}$.
Indeed for $j=1$ this is trivial while assuming \eqref{eddaje} up to $j$ we have
\begin{equation}\label{indu3}
\begin{aligned}
|{\mathtt P}_{j+1}|&\stackrel{\eqref{recu}}{\le} \Theta_{\gotp_1}^j+2 \Theta_{\gotp_1}(1+\Gamma_{\gotp_1})
|{\mathtt P}_j|\\
&\le \Theta_{\gotp_1}^j+2
\Theta_{\gotp_1}(1+\Gamma_{\gotp_1})
3^{j} \Theta_{\gotp_1}^{j-1}(1+\Gamma_{\gotp_1})^{j-1}
= \Theta_{\gotp_1}^{j}(1+2\cdot3^j(1+\Gamma_{\gotp_1})^{j})\\
&\le (1+2\cdot3^j)\Theta_{\gotp_1}^{j}(1+\Gamma_{\gotp_1})^{j}\le
 3^{j+1}\Theta_{\gotp_1}^{j}(1+\Gamma_{\gotp_1})^{j}\,,
\end{aligned}
\end{equation}
where in the last inequality we used the fact that $1+2 C^j\le C^{j+1}$ for $C\ge3$.
Summarizing we obtained
\begin{equation}\label{induno}
|\gotU^jX|_{\vec v,p}\le K^{j(\mu_{1}+\nu+1)}\Theta^{j-1}_{\gotp_1}\Big[\Theta_{\gotp_1}|X|_{\vec{v},p}+
\big(\Theta_{p}(1+\Gamma_{\gotp_1})+ \Theta_{\gotp_1}K^{\al(p-\gotp_1)}\Gamma_{p}\big)
3^j (1+\Gamma_{\gotp_1})^{j-1}|X|_{\vec{v},\gotp_1} \Big],
\end{equation}
so that (the second summand is zero for $j=0$)
\begin{equation}\label{induno2}
|\gotW\gotU^jX|_{\vec v,p}\le 4^j K^{j(\mu_{1}+\nu+1)+\mu_{1}}\Big[\Theta^{j}_{\gotp_1}|X|_{\vec{v},p}+
\Theta_{p}\Theta^{j-1}_{\gotp_1}(1+\Gamma_{\gotp_1})^j |X|_{\vec{v},\gotp_1}+ \Theta^j_{\gotp_1}K^{\al(p-\gotp_1)}(1+\Gamma_{\gotp_1})^{j}\Gamma_{p}  |X|_{\vec{v},\gotp_1}\Big]
\end{equation}
and finally
\begin{equation}
\begin{aligned}
|\gotW(1+\gotU)^{-1}X|_{\vec v,p} &\le K^{\mu_1} (1+ 4K^{\mu_{1}+\nu+1}\Theta_{\gotp_1})^b |X|_{\vec v,p} \\
&\qquad+
K^{\mu_1} (1+\Gamma_{\gotp_1})(1+ 4K^{\mu_{1}+\nu+1}\Theta_{\gotp_1}(1+\Gamma_{\gotp_1}))^{b-1}\Theta_p  |X|_{\vec{v},\gotp_1}\\
&\qquad+
K^{\mu_{1}+\al(p-\gotp_1)}(1+ 4K^{\mu_{1}+\nu+1}\Theta_{\gotp_1}(1+\Gamma_{\gotp_1}))^{b}\Gamma_p |X|_{\vec{v},\gotp_1}
\end{aligned}
\end{equation}
where again the second summand is in fact zero if $b=0$. Therefore, since $\Theta_p\le \Gamma_p$, we obtain
\begin{equation}
|\gotW(1+\gotU)^{-1}X|_{\vec v,p} \le K^{(b+1)(\mu_{1}+\nu+1)}(1+ \Theta_{\gotp_1}(1+\Gamma_{\gotp_1}))^{b}(1+\Gamma_{\gotp_1})\Big[  |X|_{\vec v,p} +K^{\al(p-\gotp_1)}\Gamma_p |X|_{\vec{v},\gotp_1} \Big] 
\end{equation}
this concludes the proof of \eqref{buoni22}. The proof of \eqref{cribbio42} follows the same lines.
Now we have defined a function $\tilde g$ on the set $\calO$. In order to conclude the proof we need to extend this function to the whole $\calO_0$. We know that regular vector fields in $E^{(K)}$ have a structure of Hilbert space w.r.t. the norm $|\cdot|_{s,a,\gotp_1}$ so we may apply Kirtzbraun Theorem in order to extend $\tilde g$ to a regular vector field in $E^{(K)}$ with the same Lipschitz norm, i.e. $|g|_{s,a,\gotp_1}^{\rm lip}\le \g^{-1}|\tilde g|_{\vec v,p}$ .  As for the sup norm one clearly has
$$
\sup_{\xi\in\calO_0}|g|_{s,a,\gotp_1}\le \sup_{\xi\in\calO_0}|\tilde g|_{s,a,\gotp_1}+ \rm{diam}(\calO_0)|g|_{s,a,\gotp_1}^{\rm lip}.
$$
\end{proof}

\section{Time analytic case}\label{solaminchia}
 
Theorem \ref{thm:kambis} does not make any assumptions on the analiticity parameters $a_0,s_0$ and relies on tame estimates in order to control the {\em high} Sobolev norm $\gotp_2\ge \gotp_1+\ka_0+\chi\ka_2$. However if one makes the Ansatz that $s_0>0$ then we may take $\gotp_2=\gotp_1$ and consequently have a simplified scheme, since we do not have to control the tameness constants in high norm but only the norm $|\cdot|_{\vec v,\gotp_1}$. 
In order to do so we need to modify
Definition \ref{linvec-abs} by substituting item {\it 3.} with the following:

\begin{itemize}
\item[{\it $3'$.}] For $K>1$ there exists  smoothing projection operators $\Pi_K: \calA_{\vec v,p}\to \calA_{\vec v,p} $ such that $\Pi_K^2=\Pi_K$
and setting $\vec v_1= (\g,\calO,s+s_1,a,r)$, 
for $p_1\ge 0$, one has 
\begin{align}
& \qquad| \Pi_{K}F |_{\vec v_1,p+p_1} \le  \mathtt C K^{{p_1 }} e^{s_1K}| F|_{\vec v,p} \label{P1'}\\
& \qquad| F-\Pi_{K}F |_{\vec v,p} \le  \mathtt C K^{{-p_1 }} e^{-s_1K} | F|_{\vec v_1,p+p_1}\label{P2'}
\end{align}
finally if $C_{\vec v,p}(F)$ is any tameness constant for $F$ then we may choose a tameness constant such that
\begin{equation}
\qquad C_{\vec v,p+p_1}(\Pi_{K}F )\le \mathtt C K^{{p_1 }} e^{s_1 K}C_{\vec v_1,p}(F)\qquad\label{P3'}
\end{equation}
\end{itemize}

%
We denote  by $E^{(K)}$ the subspace where $\Pi_K E^{(K)}=E^{(K)}$. 

\begin{const}[{\bf The exponents: analytic case}]\label{sceltaparbis}
We fix parameters $\e_0,\tR_0,\tG_0,\mu,\nu,\eta,\chi,\ka_2$ such that the following holds.
\begin{itemize}
\item $0<\e_0\le\tR_0\le \tG_0$ with $\e_0\tG_0^3,\e_0 \tG_0^2 \tR_0^{-1}<1$.


\item We have $\mu,\nu\ge 0$,  $1<\chi<2$, finally setting $\ka_0:= \mu+\nu+4$   
\begin{equation}\label{expa}
 \ka_2> \frac{2\ka_0}{2-\chi}\,,
\qquad 
 \eta > \mu+(\chi-1) \ka_2+1\,,
\end{equation}
\item  there exists $K_0>1$   such that
\begin{equation}\label{expexpexpa}
 \log K_0\geq \frac{1}{\log\chi}C,
\end{equation}
with $C$ a given function of $\mu,\nu,\eta,\ka_2, s_0$  (which goes to $\infty$ as $s_0\to 0$)
and moreover
\begin{subequations}\label{expexpa}
\begin{align}
&\tG_0^2\tR_0^{-1}\e_0 K_0^{\ka_0}\max (1, 	\tR_0 \tG_0 K_0^{\ka_0+(\chi-1)\ka_2})<1\,,
\label{1s1a}\\
&  K_0^{\ka_0 +(\chi-1) \ka_2 } e^{-  \frac{s_0 K_0}{32}}\tG_0\e_0^{-1} \max\Big ( 1,\tR_0 \Big)\le 1\,,
 \label{4s2a}
 \end{align}
 \label{smalla}
 \end{subequations}
\end{itemize}
\end{const} 
Now in order to state our result we define the good parameters and the changes of variables as in the general case but with $\gotp_2=\gotp_1$, $\ka_3=\ka_1=\al=0$. For clarity we restate our definition in this simpler case.
\begin{defi}[{\bf Homological equation}]\label{pippopuffo2a} 
Let $\gamma> 0$, $K\ge K_0$,
consider a compact set $\calO \subset \calO_0$ and set $\vec{v}=(\g,\calO,s,a,r)$ and $\vec{v}^{\text{\tiny 0}}=(\g,\calO_0,s,a,r)$.
Consider a  vector field $F\in \calW_{\vec v^{\text{\tiny 0}},p}$  i.e.
$$
F= N_0+G: \calO_0\times    D_{a,p+\nu}(r)\times\TTT^{d}_{s}\to V_{a,p}\,, 
$$
 which is  $C^{\mathtt n+2}$-tame up to order $q=\gotp_1+2$.
   We say $\mathcal O$ satisfies the  \emph{  homological equation}, for
   $(F,K,\vec{v}^{\text{\tiny 0}},\rho)$ if  
    the following holds.

 \noindent
1.\, For all $\xi\in \mathcal O$   one has  $F(\xi)\in \calE$.
 
 \noindent
2.\, there exist   a bounded regular vector field $g\in \calW_{\vec{v}^{\text{\tiny 0}},p}\cap E^{(K)}$ such that 

 \begin{itemize}
 \item[(a)]  $g\in \BB_\calE$ for all $\xi\in \calO$,
  \item[(b)] one has $|g|_{\vec{v}^{\text{\tiny 0}},\gotp_1}\le \mathtt C |g|_{\vec v,\gotp_1}\le \mathtt c \rho$
  and
 \begin{equation}\label{buoni22a}
|g|_{\vec v,\gotp_1}\leq
\gamma^{-1}K^{\mu}|\Pi_K\Pi_{\calX}G|_{\vec v,\gotp_1}
(1+
\g^{-1}C_{{\vec v,p}}(G))\,,
 \end{equation}
 \item[(c)] setting 
 $
 u:=\Pi_{K}\Pi_{\calX}({\rm ad}(\Pi_{\calX}^{\perp} F)[g]-F),
 $
one has 
  \phantom{assssasfffff}
  \begin{equation}\label{cribbio42a}
    |u|_{\vec{v},\gotp_1}\leq \e_0\g^{-1} K^{-\eta+\mu} C_{\vec{v},\gotp_1}(G)|\Pi_{K}\Pi_{\calX}G|_{\vec{v},\gotp_1},
   \end{equation}
  \end{itemize}
\end{defi}
\begin{rmk}
Note that if we take $\gotp_2=\gotp_1$ then the second inequality in \eqref{buoni22} as well as item 2(d) of Definition \ref{pippopuffo2} follow from \eqref{tameconst2} and \eqref{tameconst3}.
\end{rmk}

\begin{defi}[{\bf Compatible changes of variables: analytic case}]\label{compaa} Let the parameters in Constraint
\ref{sceltaparbis} be fixed.
Fix also $\vec v= (\g,\calO,s,a,r)$, $\vec{v}^{\text{\tiny 0}}= (\g,\calO_0,s,a,r)$ with $\calO\subseteq\calO_0$
a compact set,  parameters
$K\ge K_0,\rho<1$. Consider a
 vector field $F= N_0+G\in \calW_{\vec{v}^{\text{\tiny 0}},p}$  which is  $C^{\mathtt n+2}$-tame up to
 order $q=\gotp_1+2$ and such that $F\in \calE\quad \forall\x\in \calO$.
  We say that a 
left invertible $\calE$-preserving change of variables 
$$
\calL, \calL^{-1}: \TTT^d_{s}\times D_{a,\gotp_1}(r)\times \calO_0 \to
 \TTT^d_{s+\rho s_0}\times D_{a-\rho a_0,\gotp_1}(r+\rho r_{0})
 $$
 is {\em compatible} with $(F,K,\vec v,\rho)$ if the following holds:
\begin{itemize}
\item[(i)]  $\calL$ is ``close to identity'', i.e.
denoting $\vec{v}^{\text{\tiny 0}}_1:=(\g,\calO_0,s-\rho s_0,a-\rho a_0,r-\rho r_0)$ one has 
\begin{equation}\label{satanaa}
\begin{aligned}
&\|(\calL-\uno)h\|_{\vec{v}^{\text{\tiny 0}}_1,\gotp_1}\leq \mathtt C  \e_0 K^{-1}
\|h\|_{\vec{v}^{\text{\tiny 0}},\gotp_1}\,.\\
\end{aligned}
\end{equation}

\item[(ii)]  $\calL_*$ conjugates the $C^{\mathtt n+2}$-tame vector field $F$ to 
the vector field $\hat{F}:= 
\calL_{*}F=  N_0+ \hat G$
which is $C^{\mathtt n+2}$-{tame}; moreover denoting
 $\vec v_2:=(\g,\calO,s-2\rho s_0,a-2\rho a_0,r-2\rho r_0)$
one may choose the tameness constants of $\hat G$ so that
\begin{equation}\label{odioa}
C_{\vec v_2, \gotp_1}(\hat{G})\le C_{\vec v,\gotp_1}(G)(1+\e_0 K^{-1})\,,
\end{equation}
\item[(iii)] $\calL_*$ ``preserves the $(\NN,\calX,\RR)$-decomposition'', namely one has
\begin{equation}\label{satana2a}
\Pi_\NN^\perp (\calL_* \Pi_\NN F) = 0\,, \quad \qquad \Pi_{\calX}(\calL_{*}\Pi_{\calX}^\perp F)=0\,.
\end{equation}
\end{itemize}
\end{defi}

%

\medskip

Then the result is the following.

 \begin{theo}[{\bf Abstract KAM: analytic case}]\label{thm:kambisa}
Let $N_0$ be a diagonal vector field as in Definition \ref{norm} and consider a  vector field
 \begin{equation}\label{kam1a}
 F_0:=N_{0}+G_0 \in \calE\cap\calW_{\vec{v}_{0},p}
 \end{equation} 
 which is  $C^{\mathtt n+2}$-tame up to order $q=\gotp_1+2$.
Fix   parameters $\e_0,\tR_0,\tG_0,\mu,\nu,\eta,\chi,\ka_2$  satisfying Constraint \ref{sceltaparbis}
and assume that
 \begin{equation}\label{sizesa}
\g_{0}^{-1}C_{\vec{v}_0,\gotp_1}(G_0) \le \tG_0\,,\quad
 \g_0^{-1}C_{\vec{v}_0,\gotp_1}(\Pi_{\NN}^\perp G_0)\le \tR_0\,, \quad
 \g_0^{-1}|\Pi_{\calX}G_0|_{\vec{v}_0,\gotp_1}\le \e_0\,.
 \end{equation}

For all $n\geq 0$  we define recursively  changes of variables $\calL_n,\Phi_n$  and  compact sets
$ \calO_n$ as follows.

\smallskip

  Set $\HH_{-1}=\HH_0=\Phi_0=\calL_0=\uno$, and for $0\le j \le n-1$ set recursively
  $\HH_j= \Phi_j\circ \calL_j\circ \HH_{j-1}$ and  $F_j:=(\HH_j)_{*}F_0:=N_0+G_{j}$.
  Let $\calL_{n}$ be any change of variables   compatible with $(F_{n-1},K_{n-1}, \vec v_{n-1},\rho_{n-1})$ following Definition \ref{compaa},
    and  $\calO_n$ be any compact set 
   \begin{equation}\label{oscurosignorea}
   \calO_{n}\subseteq \calO_{n-1}\,,
   \end{equation}  
  which satisfies the Homological equation for $((\calL_n)_*F_{n-1},K_{n-1},\vec v^{\text{\tiny 0}}_{n-1},\rho_{n-1})$.
   For $n>0$ let $g_n$ be the regular vector field defined in item (2) of Definition \ref{pippopuffo2a} and set
    $\Phi_n$ the time-1 flow map generated by $g_n$. 
    
    Then $\Phi_n$ is left invertible and 
     $F_n:= (\Phi_n\circ\calL_n)_*F_{n-1}\in \calW_{\vec v^{\text{\tiny 0}}_n,p}$ is  $C^{\mathtt n+2}$-tame up to order $q=\gotp_1+2$.
Moreover the following holds.

      \begin{itemize}
     
  \item[{\bf (i)}]
  Setting $G_n=F_n-N_0$ then
 \begin{equation}\label{lamortea}
 \begin{aligned}
  \Gamma_{n,\gotp_1}&:=\g_{n}^{-1} C_{\vec{v}_n,\gotp_1}(G_n)\leq \tG_n, 
\qquad
&\Theta_{n,\gotp_1}:= \g_{n}^{-1}C_{\vec{v}_n,\gotp_1}(\Pi_{\NN}^\perp G_n)\leq \tR_n,
 \\
\de_n&:= \g_{n}^{-1} |\Pi_{\calX}G_n|_{\vec{v}_n,\gotp_1}\leq K_0^{\ka_2} \e_0 K_{n}^{-\ka_2},
\qquad
&|g_{n}|_{\vec u_n,\gotp_1}
     \leq  K_0^{\ka_2} \e_0 \tG_0K_{n-1}^{-\ka_2+\mu+1},
\end{aligned}
 \end{equation}
where 
$\vec u_{n}=(\g_{n},\calO_{n},s_{n}+12\rho_n s_0 ,a_{n}+12\rho_{n} a_0,r_{n}+12\rho_n r_0)$.
 \item[{\bf (ii)}]
 The sequence $\calH_n$ converges
     for all $\x\in\calO_0$ 
  to some change of variables
  \begin{equation}\label{dominio1000}
  {\calH}_\io={\calH}_\io(\x):  D_{{a_{0}},p}({s_{0}}/{2},{r_{0}}/{2})\longrightarrow D_{\frac{a_{0}}{2},p}({s_{0}},{r_{0}}).
  \end{equation}
  
 \item[{\bf (iii)}]
 Defining $F_{\infty}:=(\calH_\io)_{*}F_0$ one has 
 \begin{equation}\label{fine1000}
 \Pi_\calX F_{\infty}=0
 \quad \forall \xi \in \calO_\io:=\bigcap_{n\geq0}\calO_n
 \end{equation}
 and 
 $$
 \g_0^{-1}C_{\vec{v}_\io,\gotp_1}(\Pi_{\NN}F_{\infty}-N_0)\le 2\tG_0, \quad
  \g_0^{-1}C_{\vec{v}_\io,\gotp_1}(\Pi_{\RR}F_{\infty})\le 2\tR_0
 $$
 with $\vec{v}_\io:=(\g_0/2,\calO_\io,s_{0}/2,a_{0}/2)$. 
 
 \end{itemize}
\end{theo}

\begin{proof}
The proof of Theorem \ref{thm:kambisa} is essentially identical to the one of Theorem \ref{thm:kambis}. We give a sketch for completeness.
The induction basis is trivial with $g_0=0$.
Assuming \eqref{lamortea} up to $n$ we prove the inductive step using the ``KAM step'' of Proposition
\ref{kamstep}
with $\gotp_1=\gotp_2$ and $\al=\ka_3=0$ and the bound \eqref{cazzo} substituted by
\begin{equation}\label{cazzo2}
\begin{aligned}
\de_{+}&\le {\mathtt C} \Gl \big ( \de^2 
\Gamma^2_{\gotp_1}K^{2\mu+2\nu+4} + \de \e_0 K^{\mu-\eta}\big)+ e^{-2 \rho s_0 K} K^{\mu+\nu+1-(\gotp_2-\gotp_1) }  \big(\Theta_{\gotp_2}+\e_0 K^{\ka_3} \Theta_{\gotp_1} \big)\\
&\qquad+
\Gl K^{\mu+\nu+1-(\gotp_2-\gotp_1) } e^{-2 \rho s_0 K} \big (\Theta_{\gotp_2}+\e_0 K^{\ka_3} \Theta_{\gotp_1} + K^{\al(\gotp_2-\gotp_1)} \de(\Gamma_{\gotp_2}+\e_0 K^{\ka_3} \Gl)\big)\,.
\end{aligned}
\end{equation} 
Bound \eqref{cazzo2} follows using the smoothing properties  \eqref{P1'}, \eqref{P2'} in item $(3')$, in the equation \eqref{orchi8}.

First of all we note that 
\begin{equation}\label{porcocazzoa}
\rho_n^{-1} K_n^{\mu+\nu+3}\Gamma_{n,\gotp_1} \de_n \le \mathtt c,
\end{equation}
which, by the inductive hypothesis and \eqref{numeretti} reads
\begin{equation}\label{merdafrittaa}
2^{n+9}K_0^{(\mu+\nu+3-\ka_2)\chi^n} \mathtt G_0 \e_0 K_0^{\ka_2} \le \mathtt c;
\end{equation}
this is true since 
 by \eqref{expa} and \eqref{expexpexpa}
the left hand side \eqref{merdafrittaa} is decreasing in $n$
so that \eqref{porcocazzoa} follows form
$$
K_0^{\mu+\nu+4}\mathtt G_0\e_0 <1
$$
which is indeed implied by \eqref{1s1a}
because $\tG_0\ge \tR_0$.

Hence we can apply the ``KAM step''   to $F_n:= (\Phi_n\circ\calL_n)_*F_{n-1}\in \calW_{\vec v^{\text{\tiny 0}}_n,\gotp_2}$ which is a  $C^{\mathtt n+2}$-tame up to order $q=\gotp_2+2$. We fix $(K_{n},\g_n,a_n,s_n,r_n,\rho_n,\calO_n)\rightsquigarrow(K,\g,a,s,r,\rho,\calO)$,
 $\Gamma_{n,p}\rightsquigarrow \Gamma_{p}$,
$\Theta_{n,p}\rightsquigarrow \Theta_{p}$,
$\de_{n}\rightsquigarrow\delta$,
$(\g_{n+1},a_{n+1},s_{n+1},r_{n+1},\rho_{n+1},\calO_{n+1})\rightsquigarrow(\g_+,a_+,s_+,r_+,\rho_+,\calO_+)$. The KAM steps produces a bounded regular vector field $g_{n+1}$ and a left invertible change of variables $\Phi_{n+1}= \uno +f_{n+1}$  such that  
$F_{n+1}:= (\Phi_{n+1}\circ\calL_n)_*F_{n}\in \calW_{\vec v^{\text{\tiny 0}}_{n+1},\gotp_2}$ is  $C^{\mathtt n+2}$-tame up to order $q=\gotp_2+2$.
We  now verify that the bounds \eqref{lamortea} hold with $\Gamma_{n+1,\gotp_1}\rightsquigarrow \Gamma_{+,\gotp_1}$,
$\Theta_{n+1,\gotp_1}\rightsquigarrow \Theta_{+,\gotp_1}$, ,
$\de_{n+1}\rightsquigarrow\delta_+$.

Let us prove {\bf (i)}, the others follow exactly as in Theorem \eqref{thm:kambis}.

By substituting into \eqref{stimatotale2} we immediately obtain the bounds for $g_{n+1}$ of \eqref{lamortea}.

Now we recall that, by definition
$$
\frac{\g_n}{\g_{n+1}}=1+\frac{1}{2^{n+3}-1}.
$$

We use \eqref{kam8} together with the inductive hypotheses to obtain   
$$
\Gamma_{n+1,\gotp_1}\leq 
\Big(1+\frac{1}{2^{n+3}-1}\Big)
\tG_n + 2\e_0K_n^{-1}\tG_0+\mathtt C K_n^{\mu-\ka_2} \tG_0 \e_0 K_0^{\ka_2}  (K_n^{\nu+1}\tG_0+ \e_0  K_n^{-\eta})\le \tG_{n+1}\,,
$$
which follow by requiring
$$
\max( 2^nK_n^{-1}\e_0, 2^n K_n^{\mu+\nu+1-\ka_2} K_0^{\ka_2}\tG_0 \e_0  ,  2^n  K_n^{-\eta-\ka_2+\mu } K_0^{\ka_2}\e_0)\le \mathtt c\,,
$$
and as before this follows by \eqref{expa} and \eqref{1s1a}.

Regarding $\Theta_{n+1,\gotp_1}$, using \eqref{kam9} we get
$$
\Theta_{n+1,\gotp_1}\leq \Big(1+\frac{1}{2^{n+3}-1}\Big)
 \tR_n +2\e_0K_n^{-1}\tR_0+ \mathtt C K_n^{\mu+\nu+1-\ka_2} \tG_0^2 \e_0 K_0^{\ka_2}  
 +K_n^{-\eta-\ka_2+\mu}\tG_0 \e^2_0 K_0^{\ka_2}\le \tR_{n+1}.
$$
which again follows from  \eqref{expa} and \eqref{1s1a}.

For $\de_{n+1}\rightsquigarrow \de_+$, we apply \eqref{cazzo2} with $\gotp_2=\gotp_1$, $\al=\ka_3=0$  and get
\begin{align*}
\de_{n+1}\leq &{\mathtt C} \tG_0 \Big ( \e_0^2 K_0^{\ka_2} (
\tG_0^2  K_0^{\ka_2} K_n^{2\mu+2\nu+4-2\ka_2} +   K_n^{\mu-\eta-\ka_2})+ \\ &e^{-2 \rho_n s_0 K_n}(\tR_0 K_n^{\mu+\nu+1 } + 
\e_0 K_0^{\ka_2} \tG_0 K_n^{\mu+\nu+1 -\ka_2} )\Big)+ e^{-2 \rho_n s_0 K_n}\tR_0 K_n\le \e_0 K_0^{\ka_2}K_{n}^{-\chi \ka_2}
\end{align*}
Now 
$$
{\mathtt C} \tG_0 \Big ( \e_0^2 K_0^{\ka_2} (
\tG_0^2  K_0^{\ka_2} K_n^{2\mu+2\nu+4-2\ka_2} +   K_n^{\mu-\eta-\ka_2})\le \frac12 \e_0 K_0^{\ka_2}K_{n}^{-\chi \ka_2}
$$ by \eqref{expa} and \eqref{1s1a}. 
As for the second term, 
since $s_0>0$ all the summands  are decreasing in $n$ provided that $K_0$ is large enough.

\end{proof}

\bibliography{bibliografiaNLS}

\end{document}